\documentclass[reqno,twoside]{amsart}

\usepackage{amsmath}
\usepackage{amssymb}
\usepackage{dsfont}
\usepackage{braket}
\usepackage{xcolor}

\newcommand{\R}{{\mathbb R}}
\newcommand{\N}{{\mathbb N}}
\newcommand{\LL}{{\mathcal L}}
\newcommand{\HH}{{\mathcal H}}
\newcommand{\Div}{\mathrm{div}}
\newcommand{\dist}{\mathrm{dist}}
\newcommand{\supess}{\mathop{\rm ess\: sup }}
\newcommand{\supp}{{\mathrm{supp}}}
\newcommand{\im}{{\mathrm{Im}}}
\newcommand{\Ker}{{\mathrm{Ker}}}
\newcommand{\dd}{\, \mathrm{d}}
\newcommand{\dx}{\, \mathrm{d}x}
\newcommand{\dt}{\, \mathrm{d}t}
\newcommand{\ds}{\, \mathrm{d}s}
\newcommand{\dxdt}{\, \mathrm{d}x\, \mathrm{d}t}

\let\O\Omega
\let\e\varepsilon
\let\wto\rightharpoonup 

\theoremstyle{plain}
\begingroup
\newtheorem{theorem}{Theorem}[section]
\newtheorem{lemma}[theorem]{Lemma}
\newtheorem{proposition}[theorem]{Proposition}

\endgroup

\theoremstyle{definition}
\begingroup
\newtheorem{definition}[theorem]{Definition}
\newtheorem{remark}[theorem]{Remark}
\endgroup

\numberwithin{equation}{section}

\begin{document}
 
\title[Hyperbolic structure for a simplified model of dynamical perfect plasticity]{Hyperbolic structure for a simplified model of dynamical perfect plasticity}
\author[J.-F. Babadjian]{Jean-Fran\c cois Babadjian}
\author[C. Mifsud]{Cl\'ement Mifsud}

\address[J.-F. Babadjian]{Universit\'e Pierre et Marie Curie -- Paris 6, CNRS, UMR 7598 Laboratoire Jacques-Louis Lions, Paris, F-75005, France}
\email{jean-francois.babadjian@upmc.fr}

\address[C. Mifsud]{Universit\'e Pierre et Marie Curie -- Paris 6, CNRS, UMR 7598 Laboratoire Jacques-Louis Lions, Paris, F-75005, France}
\email{mifsud@ljll.math.upmc.fr}

\date{\today}

\begin{abstract}
This paper is devoted to confront two different approaches to the problem of dynamical perfect plasticity. Interpreting this model as a constrained boundary value Friedrichs' system enables one to derive admissible hyperbolic boundary conditions. Using variational methods,
we show the well-posedness of this problem in a suitable weak measure theoretic setting. Thanks to the property of finite speed propagation, we establish a new regularity result for the solution in short time. Finally, we prove that this variational solution is actually a solution of the hyperbolic formulation in a suitable dissipative/entropic sense, and that a partial converse statement holds under an additional time regularity assumption for the dissipative solutions.
\end{abstract}

\maketitle

\section{Introduction}
\label{intro}

\noindent Friedrichs' systems are linear symmetric hyperbolic systems of the form
$$
\begin{cases}
\displaystyle \partial_t U + \sum_{i=1}^n A_i\partial_{x_i}U=0 \quad \text{ in } \R^n \times (0,T),\\
U(t=0)=U_0,
\end{cases}$$
where $U:\R^n \times (0,T) \to \R^m$ is the unknown of the problem, $A_1,\ldots,A_n$ are symmetric $m \times m$ matrices, and $U_0:\R^n \to \R^m$ is a given initial data. They appear in a number of physical systems such as the wave equation or systems of conservation laws. In particular, the system of three-dimensional linearized elasto-dynamics can be put within this framework (see \cite{HughesMarsden}) where $U$ is a vector of size $9$ (with three components for the velocity, and six for the symmetric $3 \times 3$ Cauchy stress), and $A_1,A_2,A_3$ are explicit $9 \times 9$ matrices depending on the Lam\'e coefficients of the material. 

Problems of continuum mechanics are usually settled in a bounded domain $\O \subset \R^n$, which requires to impose a boundary condition. Of course, it is a difficult issue in hyperbolic equations since the initial condition is transported through the characteristics up to the boundary, where the value of the solution might thus be incompatible with the prescribed boundary data. In other words, one has to impose boundary conditions only on a part of the boundary which is not reached by the characteristics (see {\it e.g.} \cite{BLRN,O,MNRR} in the case of scalar conservation laws, or \cite{DBLF} for one-dimensional nonlinear systems). In Friedrichs' seminal work \cite{Friedrichs}, the following type of boundary conditions are considered
\begin{equation}\label{BC1}
(A_\nu -M)U=0\quad\text{on } \partial \Omega \times (0,T),
\end{equation}
where $A_\nu=A_\nu(x):=\sum_{i=1}^n A_i\nu_i(x)$ ($\nu(x)$ is the outer normal to $\O$ at the point $x \in \partial \O$), and $M=M(x)$ is a $m \times m$ matrix, for $x \in \partial \O$, satisfying the following algebraic conditions (in the non-characteristic case, {\it i.e.}, when $A_\nu$ is non-singular, see also \cite{MDS} for more details)
\begin{equation}\label{BC2}
\begin{cases}
M+M^T \text{ is non-negative},\\
\im(A_\nu- M) \cap \im(A_\nu+M) = \left\{ 0 \right\},\\
\R^m=\Ker(A_\nu- M) \oplus \Ker(A_\nu+M).
\end{cases}
\end{equation}
The fact that the symmetric part of $M$ is supposed to be non-negative is a way to ensure that the $L^2(\Omega)$-norm of the solution decreases in time, and thus this hypothesis is related to the uniqueness of the solution.  In other words, the non-negativity of $M+M^T$ is connected with the dissipativity of the equation. The two other assumptions are related to the existence of a solution.

Unfortunately, the previous formulation necessitates to define properly the trace of $U$ on the boundary, which might not be desirable if one is interested in weak solutions in Lebesgue-type spaces (in the spirit of {\it e.g.} \cite{O,MNRR} for a $L^\infty$-theory of boundary value scalar conservation laws). In \cite{MDS}, a general $L^2$-theory for such boundary value Friedrichs' systems has been introduced. The so-called {\it dissipative solutions} are defined as functions $U \in L^2(\Omega\times (0,T);\R^m)$ satisfying, for all constant vector $\kappa \in \R^m$ and all $\varphi \in W^{1,\infty}(\Omega \times (0,T))$ with $\varphi \geq 0$, 
\begin{multline}\label{1714}
 \int_{0}^T \int_{\Omega}  \left|U-\kappa\right|^2\partial_t \varphi\dxdt +  \sum_{i=1}^n\int_{0}^T \int_{\Omega} A_i(U-\kappa)\cdot (U-\kappa)\partial_{x_i}\varphi\dxdt \\
 + \int_\Omega \left|U_0-\kappa\right|^2 \varphi(0)\dx  +\int_0^T\int_{\partial \Omega} M \kappa^{+}\cdot \kappa^{+} \varphi\dd\HH^{n-1}\dt \ge 0,
\end{multline}
where $\kappa^+=\kappa^+(x)$ is the orthogonal projection of $\kappa$ onto the linear space $ \Ker(A_\nu+M) \cap \im A_\nu$. 
As discussed in \cite{MDS}, this formulation naturally discriminates the admissible boundary conditions. In particular, in this case, the matrix $M$ has to be a non-negative symmetric matrix (see Subsection~\ref{sec:entrop}). Note that this kind of dissipative formulation is reminiscent in hyperbolic equations (see {\it e.g.} \cite{K} in the case of scalar conservations laws). Moreover, the family of functions $U \mapsto |U-\kappa|^2$, where $\kappa \in \R^m$, can be thought of as the analogue of the Kru{\v{z}}kov entropy functions in \cite{K}. The term {\it dissipative} refers to the decreasing character of the $L^2$-norm of the solution which prevents conservation of the energy, and to a special class of boundary conditions for hyperbolic systems (see~\cite{BGS} and \cite[Section 4]{MDS}).

\medskip

A number of mechanical problems, such as in elasto-plasticity or generalized non-newtonian fluids, involve a convex constraint. For that reason, it becomes relevant to ask whether one can incoporate convex constraints within a general theory of Friedrichs'  systems. In \cite{DLS}, this problem has been addressed in the full space $\O=\R^n$. The authors define a notion of dissipative solutions (from which the previous formulation \eqref{1714} in \cite{MDS} has been inspired) which are shown in \cite{BMS} to be the (unique) limit of a sequence of viscosity solutions for a regularized diffusive model where the constraint is penalized. The formulation of general constrained Friedrichs' systems in bounded domains becomes therefore a natural extension. However, there might be non trivial interactions between the constraint and the boundary condition (see \eqref{interaction} below), which makes the problem difficult to address in its full generality. This is the reason why, in this paper, we focus our attention to the meaningful particular case of dynamical perfect plasticity. 

To be more precise, we consider a simplified two dimensional problem of anti-plane shear elasto-plasticity (see Subsection \ref{sec:anti-plane} for a formal derivation from three-dimensional small strain elasto-plasticity), where the displacement field $u:\Omega \times [0,T] \to \R$ is scalar valued and the stress $\sigma:\Omega \times [0,T] \to \R^2$ is vector-valued. General considerations of continuum mechanics state that the equation of motion
$$\ddot u-\Div \sigma=0 \quad \text{ in } \O \times (0,T)$$
must be satisfied. Then, following standard models of perfect plasticity, the stress is constrained to remain inside a fixed closed and convex set of $\R^2$. For simplifity, we assume that
\begin{equation}\label{stress-constraint}
|\sigma| \leq 1.
\end{equation}
Moreover, the displacement gradient decomposes additively as
$$\nabla u=e+p,$$
where $e$ and $p:\Omega \times [0,T] \to \R^2$ stand for the elastic and plastic strains, respectively. The elastic strain is related to the stress by means of a linear relation, and, again for simplicity, we set
\begin{equation}\label{e=sigma}
\sigma=e.
\end{equation}
Finally, the plastic variable evolves through the so-called flow rule, which stipulates that
\begin{equation}\label{flow-rule}
\begin{cases}
\dot p=0 & \text{ if }  |\sigma|<1,\\
\frac{\dot p}{|\dot p|}=\sigma & \text{ if }  |\sigma|=1.
\end{cases}
\end{equation}
This system must be supplemented by initial conditions on $(u,\dot u,\sigma,p)$ and boundary conditions. The mathematical analysis of dynamical elasto-plastic models has been performed in \cite{AL,BMora} (see also \cite{S2,TS,AG,DalMasoDeSimoneMora} in the static and quasi-static cases).

From the hyperbolic point of view, this problem can be interpreted as a constrained Friedrichs' system. Formally it can be put within a hyperbolic formulation of the type \eqref{1714} (see Subsection \ref{sec:entrop}) where $U = (\dot u,\sigma) \in \R^3$, and the $3 \times 3$ symmetric matrices $A_1$, $A_2$ are given by
$$A_1=
\begin{pmatrix}
0 & -1 & 0\\
-1 & 0 & 0\\
0 & 0 & 0
\end{pmatrix},
\quad
A_2=
\begin{pmatrix}
0 & 0 & -1\\
0 & 0 & 0\\
-1 & 0 & 0
\end{pmatrix}.
$$
The dissipative formulation is exactly given by \eqref{1714}, except that the constant vector $\kappa$ must belong to the contraint set $K=\R \times B$, where $B$ is the closed unit ball of $\R^2$ (see \eqref{stress-constraint} above). The hyperbolic vision of this problem motivates our choice of boundary conditions. It turns out that, in the unconstrained case ({\it i.e.} the wave equation), admissible dissipative boundary conditions in the sense of \eqref{BC1}--\eqref{BC2} are all of the form
\begin{equation}\label{boundary-condition}
\sigma\cdot \nu +\lambda^{-1}\dot u=0 \quad \text{ on }\partial \O \times (0,T),
\end{equation}
for some $\lambda>0$ (see Lemma \ref{lem:BC}). This choice will be {\it a posteriori} justified by the fact that the variational and dissipative formulations are essentially equivalent.  Note that this type of boundary condition, is quite unusual in solid mechanics. These are not of Robin type since it involves the velocity $\dot u$, and not the displacement $u$. It is closer to Navier's no-slip boundary condition rather found in fluid mechanics problems.

The goal of this paper consists thus in studying this particular model (in any space dimension and with a source term) related to dynamical perfect plasticity from both variational and hyperbolic points of view. First of all, using variational methods we establish a well-posedness result for this model. To this aim, we regularize the problem by considering a elasto-visco-plastic model where the constitutive law \eqref{e=sigma} is replaced by a Kelvin-Voigt visco-elastic law
$$\sigma=\tilde \sigma+\e\nabla \dot u,$$
where $\e>0$ is a viscosity parameter, and stress constraint \eqref{stress-constraint} together with the flow rule \eqref{flow-rule} are replaced by a Perzyna visco-plastic law
$$\dot p=\frac{\tilde \sigma-P_B(\tilde \sigma)}{\e},$$
where $P_B$ denotes the orthogonal projection operator onto $B$. The equation of motion and the boundary condition are thus to be modified into
\[
\begin{cases}
\ddot u-\Div (\tilde\sigma+\e\nabla \dot u)=0 \quad \text{ in } \O \times (0,T),\\
(\tilde \sigma+\e\nabla \dot u)\cdot \nu +\lambda^{-1}\dot u=0 \quad \text{ on }\partial \O \times (0,T).
\end{cases}
\]
The well-posedness of this regularized model is presented in Section~\ref{sec:3}. 

In Section~\ref{sec:4}, we prove the existence and uniqueness of a (variational) solution for the original model by means of a vanishing viscosity analysis as $\e\to 0$. However, since in the limit, the stress satisfies the constraint $|\sigma| \leq 1$, the original boundary condition \eqref{boundary-condition} cannot be satisfied at points of the boundary where $|\dot u|>\lambda$. Therefore, a relaxation phenomenon occurs (see Proposition \ref{relax}) which implies that the boundary condition \eqref{boundary-condition} relaxes as
\begin{equation}\label{interaction}
\sigma\cdot\nu+\lambda^{-1} T_\lambda(\dot u) = 0\quad \text{ on } \partial \Omega \times (0,T),
\end{equation}
where $T_\lambda(z)=\min(-\lambda,\max(z,\lambda))$ is the truncation of $z \in \R$ by the values $\pm\lambda$. It shows an interesting interaction which imposes the boundary condition to accomodate the constraint. Note that since $\lambda \in (0,+\infty)$, the important cases of Dirichlet and Neumann boundary conditions are prohibited by this formulation. In Subsection \ref{sec:lambda}, we show by means of asymptotic analysis that the Dirichlet (resp. Neumann) boundary condition can be recovered by letting $\lambda \to 0$ (resp. $\lambda \to +\infty$).

As usual in plasticity, the solution happens to concentrate, leading to a bounded variation solution for the displacement, and a measure solution for the plastic strain.  In Section~\ref{sec:5}, using the property of finite speed propagation, we prove a new regularity result in plasticity which states that, provided the data are smooth and compactly supported in space, the solution is smooth as well in short time. The argument rests on a Kato inequality (Proposition \ref{prop_kato_ineq}) which states a comparison principle between two solutions associated to different data. The fact that the data is compactly supported in $\O$ together with the finite speed propagation property ensures that, in short time, the boundary is not reached by the solution so that the boundary condition can be ignored, and one can argue as in the full space. To our knowledge, it seems to be the first regularity result in dynamical perfect-plasticity, and its generalization to more general vectorial models will be the object of a forthcoming work. 

Finally, in Section~\ref{sec:6}, we establish rigorous links between the variational and hyperbolic formulations. We show that any variational solutions generate dissipative solutions. Conversely, provided the solution of the hyperbolic problem are smoother in time, variational solutions can be deduced from the dissipative formulation.

\section{Mathematical preliminaries}
\label{sec:2}

\subsection{General notation}

If $a$ and $b \in \R^n$, we write $a \cdot b$ for the Euclidean scalar product, and we denote by $|a|=\sqrt{a \cdot a}$ the associated norm.  Let $B:=\{x \in \R^n : |x|\leq 1\}$  be the closed unit ball in $\R^n$, and $P_B$ be the orthonormal projection onto $B$, {\it i.e.}, $P_B(\sigma)=\sigma/|\sigma|$ if $\sigma \neq 0$, and $P_B(\sigma)=0$ if $\sigma=0$. It is a standard fact of convex analysis that the function 
$$\sigma \mapsto \frac{|\sigma-P_B(\sigma)|^2}{2\e}$$
is convex, of class $\mathcal C^1$, and that its differential is given by $\sigma \mapsto \big(\sigma-P_B(\sigma)\big)/\e$.
In addition, its convex conjugate is $p \mapsto |p|+\e |p|^2/2$, and in particular, 
$$p=\frac{\sigma-P_B(\sigma)}{\e}\quad \Longleftrightarrow \quad \sigma \cdot p=|p|+\e |p|^2.$$

\medskip

We write $\mathbb M^{n \times n}$ for the set of real $n \times n$ matrices, and $\mathbb M^{n \times n}_{\rm sym}$ for that of all real symmetric $n \times n$ matrices. Given a matrix $A \in \mathbb M^{n \times n}$, we let $|A|:=\sqrt{{\rm tr}(A A^T)}$ ($A^T$ is the transpose of $A$, and ${\rm tr }A$ is its trace) which defines the usual Euclidean norm over $\mathbb M^{n \times n}$. We recall that for any two vectors $a$ and $b \in \R^n$, $a \otimes b \in \mathbb M^{n \times n}$ stands for the tensor product, {\it i.e.}, $(a \otimes b)_{ij}=a_i b_j$ for all $1 \leq i,j \leq n$, and $a \odot b:= (a \otimes b + b \otimes a) /2\in \mathbb M^{n \times n}_{\rm sym}$ denotes  the symmetric tensor product.
 
\subsection{Functional spaces}

Let $\Omega \subset \R^n$ be an open set. We use standard notation for Lebesgue and Sobolev spaces. In particular, for $1\leq p\leq \infty$, the $L^p(\O)$-norms of the various quantities are denoted by $\| \cdot\|_p$. 

We write $\mathcal{M}(\Omega;\R^m)$ (or simply $\mathcal M(\Omega)$ if $m=1$) for the space of bounded Radon measures in $\Omega$ with values in $\R^m$, endowed with the  norm $|\mu|(\Omega)$, where $|\mu|\in \mathcal{M}(\Omega)$ is the total variation of the measure $\mu$. The Lebesgue measure in $\R^n$ is denoted by $\LL^n$, and the $(n-1)$-dimensional Hausdorff measure by $\HH^{n-1}$.  

We denote by $H(\Div,\Omega)$ the Hilbert space of all $\sigma \in L^2(\Omega;\R^n)$ such that $\Div \sigma \in L^2(\Omega)$. We recall that if $\O$ is bounded with Lipschitz boundary and $\sigma \in H(\Div,\Omega)$, its normal trace, denoted by $\sigma\cdot\nu$, is well defined as an element of $H^{-1/2}(\partial \Omega)$. If further $\sigma \in H(\Div,\Omega) \cap L^\infty(\Omega;\R^n)$, it turns out that $\sigma\cdot \nu \in L^\infty(\partial \Omega)$ with $\|\sigma\cdot\nu\|_{L^\infty(\partial \O)}\leq \|\sigma\|_\infty$ (see~\cite[Theorem 1.2]{Anzellotti}). Moreover, according to \cite[Theorem 2.2]{CF}, if $\O$ is of class $\mathcal C^2$, then for all $\varphi \in L^1(\partial \O)$,
\begin{equation}\label{eq:sigmanu}
\lim_{\e \to 0}\int_0^1\int_{\partial \O} \big(\sigma(y-\e s \nu(y))\cdot \nu(y) - (\sigma\cdot \nu)(y) \big)\varphi(y)\dd\HH^{n-1}(y)\ds=0,
\end{equation}
where $\nu$ denotes the outer unit normal to $\partial \O$.

The space $BV(\Omega)$ of functions of bounded variation in $\Omega$ is made of all functions $u \in L^1(\Omega)$ such that the distributional gradient $Du\in \mathcal M(\Omega;\R^n)$. We refer to~\cite{AmbrosioFuscoPallara,EG,Giusti} for a detailed presentation of this space. We just recall here few facts. If $\O$ has Lipschitz boundary, any function $u \in BV(\O)$ admits a trace, still denoted by $u \in L^1(\partial \O)$, such that Green's formula holds (see {\it e.g.} Theorem 1, Section 5.3 in \cite{EG}). Moreover, according to \cite[Theorem 4]{AnzellottiGiaquinta}, if  $\O$ is further of class $\mathcal{C}^1$, for every $\e>0$, there exists a constant $c_{\e}(\Omega)>0$ such that for every $u\in BV(\Omega)$,
\begin{equation}\label{eq:traceBV}
\|u\|_{L^1(\partial \O)}\le (1+\e)|Du|(\Omega) + c_{\e}(\Omega)\|u\|_1.
\end{equation} 
Moreover, if $\O$ is of class $\mathcal C^2$, one has
\begin{equation}\label{eq:traceu}
\lim_{\e \to 0} \int_0^1\int_{\partial \O} |u(y-\e s \nu(y))-u(y)|\dd\HH^{n-1}(y)\ds=0.
\end{equation}

Conversely, Gagliardo's extension result (see \cite[Theorem 2.16 \& Remark 2.17]{Giusti}) states that if $\O \subset \R^n$ is a bounded open set of class $\mathcal C^1$, and $g \in L^1(\partial \O)$, for each $\e>0$, there exists a function $u_\e \in W^{1,1}(\O)$ such that
\begin{equation}\label{eq:lift}\begin{cases}
u_\e=g \text{ on }\partial \O,\\
\|u_\e\|_1 \leq \e \|g\|_{L^1(\partial \O)},\\
\|\nabla u_\e\|_1 \leq (1+\e)\|g\|_{L^1(\partial \O)}.
\end{cases}
\end{equation}
Let us finally mention a variant of the usual approximation result for $BV$ functions \cite[Lemma 5.2]{Anzellotti}.
\begin{proposition}\label{prop:MS}
Let $\O \subset \R^n$ be a bounded open set with Lipschitz boundary, $u \in BV(\O)$ and $\sigma \in L^2(\O;\R^n)$. There exists a sequence $(u_j) \subset W^{1,1}(\O)$ such that
$u_j \wto u$ weakly* in $BV(\O)$, $|Du_j-\sigma|(\O) \to |Du-\sigma|(\O)$ and $u_j=u$ on $\partial \O$ for all $j \in \mathbb N$.
\end{proposition}

\subsection{Generalized stress/strain duality}

According to \cite[Definition 1.4]{Anzellotti}, we define a duality pairing between stresses and plastic strains as follows.
\begin{definition}
\label{def_Sigma_p}
Let $u \in BV(\Omega)\cap L^2(\Omega)$ be such that $D u = e+p$ for some $e\in L^2(\Omega;\R^n)$ and $p\in \mathcal{M}(\Omega;\R^n)$, and let $\sigma \in H(\Div,\Omega) \cap L^\infty(\Omega;\R^n)$. We define the distribution $\left[\sigma\cdot p \right]\in \mathcal D'(\O)$ by
\[
\Braket{\left[\sigma\cdot p \right], \varphi} = \Braket{\left[\sigma\cdot Du \right], \varphi} - \int_\Omega \sigma \cdot e \varphi\dx \quad \text{ for all } \varphi\in \mathcal{C}^{\infty}_c(\Omega),
\]
where $\left[\sigma\cdot Du \right]$ is given by
\[
\Braket{\left[\sigma\cdot Du \right], \varphi} = -\int_{\Omega} u (\Div \sigma) \varphi \dx - \int_{\Omega} u (\sigma \cdot \nabla \varphi) \dx \quad \text{ for all } \varphi\in \mathcal{C}^{\infty}_c(\Omega).
\]
\end{definition}

\begin{remark}\label{rmk:ineq1}
\label{rem_radon_measure}
Using an approximation procedure as in \cite[Lemma 5.2]{Anzellotti}, one can show that $[\sigma\cdot p]$ is actually a bounded Radon measure in $\Omega$ satisying
$$|[\sigma\cdot p]| \leq \|\sigma\|_\infty |p| \quad \text{ in } \mathcal M(\Omega).$$
\end{remark}

A slight adaptation of \cite[Theorem 1.9]{Anzellotti} shows the following integration by parts formula.

\begin{proposition}
\label{IPP_sig_u_varphi}
Let $u \in BV(\Omega)\cap L^2(\Omega)$, $\sigma \in H(\Div,\Omega) \cap L^\infty(\Omega;\R^n)$ and $\varphi \in W^{1,\infty}(\O)$. Then, 
\[
\int_\Omega \varphi \dd[\sigma\cdot Du ] + \int_\Omega u \varphi (\Div \sigma) \dx + \int_\Omega u (\sigma \cdot \nabla \varphi) \dx = \int_{\partial \Omega} (\sigma \cdot \nu) u \varphi \dd \mathcal{H}^{n-1}.
\]
\end{proposition}

\section{Description of the model}

\subsection{Small strain elasto-plasticity}

To simplify the presentation of the model, we consider the physical three-dimensional case. We assume that the reference configuration of the elasto-plastic body under consideration occupies the volume $\mathcal B \subset \R^3$. In the framework of small strain elasto-plasticity, the natural kinematic variable is the displacement field $\boldsymbol u:\mathcal B \times [0,T] \to \R^3$ (or the velocity $\boldsymbol v:=\dot {\boldsymbol u}$). Denoting by $E\boldsymbol u:=(D\boldsymbol u+D\boldsymbol u^T)/2:\mathcal B \times [0,T]\to \mathbb M^{3 \times 3}_{\rm sym}$ the linearized strain tensor, small strain elasto-plasticity assumes that $E\boldsymbol u$ decomposes additively as
\begin{equation}\label{plast1}
E\boldsymbol u=\boldsymbol e+\boldsymbol p,
\end{equation}
where $\boldsymbol e$ and $\boldsymbol p:\mathcal B \times [0,T] \to \mathbb M^{3 \times 3}_{\rm sym}$ stand for the elastic and plastic strains, respectively. The elastic strain is related to the Cauchy stress tensor $\boldsymbol \sigma:\mathcal B \times [0,T]\to \mathbb M^{3 \times 3}_{\rm sym}$ by means of Hooke's law $\boldsymbol \sigma=\mathbb C \boldsymbol e$, where $\mathbb C$ is the symmetric fourth order elasticity tensor. For example, in the isotropic case, one has
\begin{equation}\label{plast2}
\boldsymbol \sigma=\lambda ({\rm tr}\boldsymbol e) I +2\mu \boldsymbol e,
\end{equation}
where $\lambda$ and $\mu$ are the Lam\'e coefficients satisfying $\mu>0$ and $3\lambda+2\mu>0$. In a dynamical framework and in the presence of external body loads $\boldsymbol f:\mathcal B \times [0,T] \to \R^3$, the equations of motion are a system of partial differential equations which writes as
\begin{equation}\label{plast3}
\ddot{\boldsymbol u}-\Div\boldsymbol \sigma=\boldsymbol f \quad \text{ in }\mathcal B \times (0,T).
\end{equation}
Plasticity is characterized by the existence of a yield zone in the stress space beyong which the Cauchy stress is not permitted to live. The stress tensor is indeed constrained to belong to a fixed nonempty, closed and convex subset of $\mathbb M^{3 \times 3}_{\rm sym}$. In the case of Von Mises plasticity, the constraint only acts on the (trace free) deviatoric stress $\boldsymbol \sigma_D:=\boldsymbol \sigma-({\rm tr}\boldsymbol \sigma) I/3$, and reads as
\begin{equation}\label{plast4}
|\boldsymbol \sigma_D|\leq k,
\end{equation}
where $k>0$ is a critical stress value. The evolution of the plastic strain is described by means of the flow rule and is expressed with the Prandtl-Reuss law
\begin{equation}\label{plast5}
\begin{cases}
\dot{\boldsymbol p}= 0 & \text{ if }  |\boldsymbol \sigma_D|<k,\\
\frac{\dot{\boldsymbol p}}{|\dot{\boldsymbol p}|}=\frac{\boldsymbol\sigma_D}{k} & \text{ if }  |\boldsymbol \sigma_D|=k.
\end{cases}
\end{equation}
The system \eqref{plast1}--\eqref{plast5} is supplemented by initial and boundary conditions which will be discussed later.

\subsection{Anti-plane shear}\label{sec:anti-plane}

Denoting by $(\boldsymbol e_1,\boldsymbol e_2,\boldsymbol e_3)$ the canonical basis of $\R^3$, we assume that $\mathcal B$ is invariant in the $\boldsymbol e_3$ direction so that $\mathcal B=\O \times \R$, where $\O \subset \R^2$ is a bounded open set. We also suppose that the displacement is anti-plane
$\boldsymbol u (x_1,x_2,x_3)=u(x_1,x_2) \boldsymbol e_3$ for some scalar function $u:\O\times [0,T] \to \R$, so that computing the linearized strain yields
$$E\boldsymbol u=(\partial_{x_1} u) \boldsymbol e_1 \odot \boldsymbol e_3+(\partial_{x_2} u) \boldsymbol e_2 \odot \boldsymbol e_3$$
corresponding to pure shear strain.  We thus assume that the elastic and plastic strains conserve this special structure so that
$$\boldsymbol p=p_1  \boldsymbol e_1 \odot \boldsymbol e_3+p_2  \boldsymbol e_2 \odot \boldsymbol e_3, \quad \boldsymbol e=e_1  \boldsymbol e_1 \odot \boldsymbol e_3+e_2  \boldsymbol e_2 \odot \boldsymbol e_3,$$
for some functions $e_1$, $e_2$, $p_1$ and $p_2:\O \times [0,T]\to \R$. Denoting by $e=(e_1,e_2)$ and $p=(p_1,p_2)$, the additive decomposition \eqref{plast1} now reads as
$$\nabla u=e+p.$$
Computing the Cauchy stress according to \eqref{plast2} yields a pure shear stress $\boldsymbol \sigma=2\mu \boldsymbol e$, so that denoting by $\sigma:=(\sigma_{13},\sigma_{23})$ its only nonzero components, we have
$$\sigma=\mu e.$$
We also assume that the body load is compatible with the anti-plane assumption, {\it i.e.}, $\boldsymbol f=f\boldsymbol e_3$, for some $f:\O\times [0,T] \to \R$, so that the equations of motion  \eqref{plast3} becomes a scalar equation
$$\ddot u-\Div \sigma=f \quad \text{ in }\O \times (0,T).$$
Finally, the stress constraint \eqref{plast4} now reads as $|\sigma|\leq k/\sqrt 2$, and the flow rule \eqref{plast5} is given by
$$
\begin{cases}
\dot p=0 & \text{ if }  |\sigma|<k/\sqrt 2,\\
\frac{\dot p}{|\dot p|}=\frac{\sqrt 2 \sigma}{k} & \text{ if }  |\sigma|=k/\sqrt 2.
\end{cases}
$$
In order to simplify notation, we assume henceforth that $\mu=1$ (so that $\sigma=e$) and $k=\sqrt 2$. The simplified model of plasticity thus consists in looking for functions $u:\O \times [0,T] \to \R$, $\sigma:\O\times [0,T] \to \R^2$ and $p:\O\times [0,T] \to \R^2$ such that the following system holds in $\O \times (0,T)$:
\begin{equation}\label{general_model}
\begin{cases}
\nabla u=\sigma+p,\\
\ddot u-\Div \sigma=f,\\
|\sigma| \leq 1,\\
\dot p=0 & \text{ if }  |\sigma|<1,\\
\frac{\dot p}{|\dot p|}= \sigma & \text{ if }  |\sigma|=1.
\end{cases}
\end{equation}
Note that the flow rule can be equivalently be written as
\begin{equation}\label{Hill}
\sigma\cdot \dot p=|\dot p|.
\end{equation}
which expresses Hill's principle of maximal plastic work.

We supplement the system with initial conditions on the displacement, the velocity, the stress and the plastic strain
$$(u,\dot u,\sigma,p)(0)=(u_0,v_0,\sigma_0,p_0).$$
The precise mathematical formulation of this model will be the object of Section \ref{sec:4}. In particular, the flow rule \eqref{Hill} will have to be interpreted in a suitable measure theoretic sense according to the generalized stress/strain duality introduced in Definition \ref{def_Sigma_p}.

The discussion of admissible boundary conditions is the object of the following paragraphs, once the hyperbolic structure of the system will be described.

\subsection{Dissipative formulation of the model}\label{sec:entrop}

In this section, we perform formal manipulations on the system \eqref{general_model} in order to write it in a different form, more appropriate to describe  hyperbolicity. To do that, we denote by $U:=(\dot u,\sigma)$, and observe that the first two equations of \eqref{general_model} can be written as
\begin{equation}\label{entrop}
\partial_t U + A_1\partial_{x_1} U + A_2 \partial_{x_2}U+P=F,
\end{equation}
where $F=(f,0,0)$, $P=(0,\dot p_1,\dot p_2)$ and 
\begin{equation}\label{eq:Ai}
A_1=
\begin{pmatrix}
0 & -1 & 0\\
-1 & 0 & 0\\
0 & 0 & 0
\end{pmatrix},
\quad
A_2=
\begin{pmatrix}
0 & 0 & -1\\
0 & 0 & 0\\
-1 & 0 & 0
\end{pmatrix}.
\end{equation}
Taking the scalar product of \eqref{entrop} with $U$, yields
$$\frac12 \partial_t|U|^2 +\frac12 \partial_{x_1} (A_1 U \cdot U) +  \frac12\partial_{x_2} (A_2 U \cdot U) +P\cdot U= F \cdot U,$$
while, for every constant vector $\kappa=(k,\tau) \in K:=\R \times B$, taking the scalar product of \eqref{entrop} with $\kappa$ leads to
$$\partial_t (U\cdot \kappa) + \partial_{x_1}(A_1 U \cdot \kappa)  + \partial_{x_2}(A_2 U \cdot \kappa ) + P\cdot \kappa=F \cdot \kappa.$$
Substracting both previous equalities, and using that $P\cdot (U-\kappa)=\dot p \cdot (\sigma-\tau) \geq 0$ according to the flow rule written as \eqref{Hill}, we infer that
\begin{equation}\label{1652}
\partial_t|U-\kappa|^2 + \sum_{i=1}^2 \partial_{x_i} \big(A_i (U-\kappa) \cdot (U-\kappa)\big)   \leq 2F \cdot (U-\kappa).
\end{equation}
We then multiply the previous inequality by a test function $\varphi \in \mathcal C^\infty_c(\O \times (-\infty,T))$ with $\varphi \geq 0$, and integrate by parts to obtain
\begin{multline*}
 \int_0^T\int_\O |U-\kappa|^2\dot \varphi\dx\dt + \sum_{i=1}^2 \int_0^T\int_\O A_i (U-\kappa) \cdot (U-\kappa)\partial_{x_i}\varphi \dx\dt \\
+ \int_\O |U_0-\kappa|^2\varphi(0)\dx +2 \int_0^T \int_\O F \cdot (U-\kappa)\varphi \dx\dt \geq 0,
\end{multline*}
which is precisely the formulation of constrained Friedrichs' systems as defined in \cite{DLS} without taking care of boundary conditions since $\varphi$ vanishes in a neighborhood of $\partial \O \times (0,T)$.  

\medskip

In order to account for boundary condition, we follow an approach introduced in \cite{MDS}. Following the pioneering work \cite{Friedrichs}, we are formally interested in dissipative boundary conditions of the form
\begin{equation}
\label{friedrichs_BC}
(A_\nu -M)U=0\quad\text{on } \partial \Omega \times (0,T),
\end{equation}
where $A_\nu=A_1\nu_1 + A_2\nu_2$ ($\nu=(\nu_1,\nu_2)$ is the outer normal to $\partial \O$), and $M \in \mathbb M^{3 \times 3}$ is a boundary matrix satisfying the following algebraic conditions
\begin{equation}\label{eq:alg_cond}
\begin{cases}
M=M^T,\\
M \text{ is non-negative},\\
\Ker A_\nu \subset \Ker M,\\
\R^3=\Ker(A_\nu- M) + \Ker(A_\nu+M).
\end{cases}
\end{equation}
Note that in the non-characteristic case ({\it i.e.} when $A_\nu$ is non-singular), conditions \eqref{eq:alg_cond} imply those  \eqref{BC2}  introduced by Friedrichs (see \cite{MDS} for a detailed discussion). Thus, multiplying inequality \eqref{1652} by a test function $\varphi \in \mathcal C^\infty_c(\R^n \times (-\infty,T))$ with $\varphi \geq 0$, and integrating by parts, we get that
\begin{multline}\label{Fried2}
 \int_0^T\int_\O |U-\kappa|^2\dot \varphi\dx\dt + \sum_{i=1}^2\int_0^T \int_\O A_i (U-\kappa) \cdot (U-\kappa)\partial_{x_i}\varphi \dx\dt 
+ \int_\O |U_0-\kappa|^2\varphi(0)\dx\\+2 \int_0^T\int_\O  F \cdot (U-\kappa)\varphi \dx\dt- \int_0^T \int_{\partial \O} A_\nu (U-\kappa) \cdot (U-\kappa)\partial_{x_i}\varphi \dx\dt \geq 0.
\end{multline}
According to  \cite[Lemma 1]{MDS}, we have that
$$\R^{3} = \Ker A_\nu  \oplus   \big(\Ker(A_{\nu}- {M})\cap \im A_{\nu} \big)   \oplus  \big( \Ker(A_{\nu}+ {M}) \cap \im A_{\nu} \big).$$
For each $\kappa\in \R^3$, we denote by $\kappa^\pm$ the projection of $\kappa$ onto $\Ker(A_{\nu}\pm {M})\cap \im A_{\nu} $. Using the (strong) boundary condition \eqref{friedrichs_BC}, we have that $U \in \Ker(A_\nu-M)$, or still $U^+=0$. The algebraic conditions \eqref{eq:alg_cond} together with \cite[Lemma 1]{MDS} thus yield
\begin{multline*}
A_\nu (U-\kappa)\cdot  (U-\kappa) =-M (U-\kappa)^+\cdot  (U-\kappa)^+ +M (U-\kappa)^-\cdot  (U-\kappa)^-\\
=-M\kappa^+\cdot \kappa^+ +M (U-\kappa)^-\cdot  (U-\kappa)^- \geq -M\kappa^+\cdot \kappa^+.
\end{multline*}
Inserting in \eqref{Fried2}, we get that for all constant vector $\kappa \in K$ and all $\varphi \in W^{1,\infty}(\Omega \times (0,T))$ with $\varphi \geq 0$,
\begin{multline}\label{FS}
 \int_{0}^T \int_{\Omega}  \left|U-\kappa\right|^2\dot \varphi\dxdt + \sum_{i=1}^2 \int_{0}^T \int_{\Omega} A_i(U-\kappa)\cdot (U-\kappa)\partial_{x_i}\varphi\dxdt \\
 + \int_\Omega \left|U_0-\kappa\right|^2 \varphi(0)\dx + 2\int_{0}^T \int_{\Omega} F\cdot (U-\kappa)\varphi\dxdt +\int_0^T\int_{\partial \Omega} M \kappa^{+}\cdot \kappa^{+} \varphi \dd \mathcal H^1 \dt \ge 0.
\end{multline}
The previous family of inequalities defines a notion dissipative solutions $U \in L^2(\O \times (0,T);K)$ to the dynamical elasto-plastic problem. Note that it is meaningfull within a $L^2$ theory of Friedrichs' systems (as suggested by \eqref{FS}) since the trace of $U$ on the boundary $\partial \O \times (0,T)$, which is not well defined, is not involved in this definition (see also \cite{O,MNRR} for an $L^\infty$-theory of initial/boundary value conservation laws).

\subsection{Derivation of the boundary conditions}

The well-posedness of this kind of dissipative formulations in the full space $\O=\R^2$ has been established in \cite{DLS}. On the other hand, among the results of \cite{MDS}, it is shown the existence and uniqueness of a solution to this problem in the unconstrained case ($K=\R^3$), and the dissipative boundary condition \eqref{friedrichs_BC} is proved to be satisfied in a suitable weak sense. 

In order to formulate more precisely the admissible boundary conditions in our particular situation, we need to characterize all boundary matrices satisfying the required algebraic conditions \eqref{eq:alg_cond}.

\begin{lemma}\label{lem:BC}
Assume that $A_1$ and $A_2$ are given by  \eqref{eq:Ai} and $\nu \in \R^2$ satisfies $|\nu|=1$. The following assertions are equivalent:
\begin{enumerate}
\item A matrix $M \in \mathbb M^{3 \times 3}$ satisfies \eqref{eq:alg_cond};
\item There exists $\lambda \in (0,+\infty)$ such that 
\begin{equation}\label{eq:M}
M=
\begin{pmatrix}
\lambda^{-1} & 0 & 0\\
0 & \lambda \nu_1^2 & \lambda \nu_1\nu_2\\
0 & \lambda \nu_1\nu_2 & \lambda \nu_2^2
\end{pmatrix}.
\end{equation}
\end{enumerate}
\end{lemma}

\begin{proof}
It is immediate to check that any matrix $M$ of the form \eqref{eq:M} with $\lambda>0$ fullfills all conditions \eqref{eq:alg_cond}. Conversely, 
assume that
$$M= \begin{pmatrix}
d_1 & a & b  \\
a & d_2 & c  \\
b & c & d_3  \\
\end{pmatrix}  \in \mathbb M^{3 \times 3}_{\rm sym}$$
satisfies \eqref{eq:alg_cond}. Since
\[
A_\nu = \begin{pmatrix}
0 & -\nu_1 & -\nu_2 \\
-\nu_1 & 0 & 0  \\
-\nu_2 & 0 & 0 \\
\end{pmatrix},
\]
we get that
\[
\Ker A_\nu = \left\{ (v,\sigma) = (v,\sigma_1,\sigma_2) \in \R^3: v=0 \text{ and } \sigma\cdot \nu = \sigma_1 \nu_1 + \sigma_2 \nu_2 =0 \right\}.
\]
Denoting by $\nu^\perp:=\left\{ \sigma \in \R^2: \sigma\cdot \nu=0 \right\}$ the linear space of dimension $1$ in $\R^2$, condition $\Ker A_\nu \subset \Ker M$ reads as
$$\begin{cases}
\nu^\perp \subset \left\{ \sigma\in \R^2: a\sigma_1 + b\sigma_2  =0 \right\} &=: E_1,\\
\nu^\perp \subset \left\{\sigma\in \R^2: d_2\sigma_1 + c\sigma_2  =0 \right\}&=: E_2,\\
\nu^\perp \subset \left\{\sigma\in \R^2: c\sigma_1 + d_3\sigma_2  =0 \right\}&=: E_3.
\end{cases}$$
Consequently, we obtain that the dimension of the linear spaces $E_1$, $E_2$ and $E_3$ is larger than or equal to $1$. If $\dim E_1=1$, then $\nu$ is orthogonal to $E_1$, while if $\dim E_1 = 2$, then $a=b=0$. In both cases, one can find $\mu_1\in \R$  such that
$$(a,b)=\mu_1 \nu.$$
Arguing similarly for $E_2$ and $E_3$, there exist $\mu_2$ and $\mu_3 \in \R$ such that
$$(d_2,c) = \mu_2 \nu, \quad (c,d_3) = \mu_3 \nu,$$
so that
$$ M = 
\begin{pmatrix}
d_1 & \mu_1 \nu_1 & \mu_1\nu_2 \\
\mu_1\nu_1 & \mu_2 \nu_1 & \mu_2 \nu_2 \\
\mu_1\nu_2 & \mu_3 \nu_1 & \mu_3 \nu_2\\
\end{pmatrix}.
$$
Using that $M$ is symmetric, we must have $\mu_2 \nu_2  =  \mu_3 \nu_1$. Since $|\nu|=1$, then either $\nu_1 \neq 0$ or $\nu_2 \neq 0$. 
Suppose without loss of generality that $\nu_1\neq 0$, and define $\lambda = \mu_2/\nu_1$, then
\[
 M = \begin{pmatrix}
d_1 & \mu_1 \nu_1 & \mu_1\nu_2  \\
\mu_1\nu_1 & \lambda \nu_1^2 & \lambda \nu_1\nu_2  \\
\mu_1\nu_2 & \lambda \nu_1\nu_2 & \lambda \nu_2^2 \\
\end{pmatrix}.
\]
Using next that $M$ is non-negative, it follows that for all $(v,\sigma) \in \R^3$,
\[
M \begin{pmatrix}
v \\
\sigma
\end{pmatrix}\cdot \begin{pmatrix}
v \\
\sigma
\end{pmatrix} = d_1 v^2 + 2\mu_1 v \sigma \cdot \nu + \lambda \left(\sigma \cdot \nu \right)^2 \geq 0,
\]
which ensures that $d_1 \ge 0$, $\lambda \ge 0$. In fact, if $d_1 = \lambda = 0$, the previous expression can easily be made negative so that either $d_1>0$ or $\lambda>0$ (since the case $\mu_1=0$ is impossible). 

From the conditions  $\Ker A_\nu \subset \Ker M$ and $\dim \Ker A_\nu =1$, we obtain that $\dim \Ker(A_\nu \pm M) \ge 1$
and $\dim \big( \Ker(A_\nu - M) \cap \Ker(A_\nu +  M) \big) \ge 1$. 
The last condition $\R^3=\Ker(A_\nu- M) + \Ker(A_\nu+M)$ then implies  that $\dim \Ker(A_\nu \pm M) = 2$ (since $A_\nu$ is neither non-negative, nor non-positive). Computing
\[
A_\nu \pm M = \begin{pmatrix}
\pm d_1 & (\pm \mu_1 -1)\nu_1 & (\pm \mu_1 -1)\nu_2  \\
(\pm \mu_1 -1)\nu_1 & \pm \lambda \nu_1^2 & \pm \lambda \nu_1\nu_2  \\
(\pm \mu_1 -1)\nu_2 & \pm \lambda \nu_1\nu_2 & \pm \lambda \nu_2^2  \\
\end{pmatrix},
\]
we infer that
\[
\left(A_\nu - M\right) \begin{pmatrix}
v\\
\sigma
\end{pmatrix}
= 0\qquad \Longleftrightarrow \qquad \left\{
\begin{array}{rcl}
d_1 v + (\mu_1+1) \sigma\cdot \nu &=& 0,\\
(\mu_1+1) v + \lambda \sigma\cdot \nu &=& 0.
\end{array}
\right.
\]
Observe that, since $\dim \Ker(A_\nu - M) = 2$, we obtain that
\begin{eqnarray*}
d_1 v + (\mu_1+1) \sigma\cdot \nu = 0 & \Longleftrightarrow &  (\mu_1+1) v + \lambda \sigma\cdot \nu = 0  \\
& \Longleftrightarrow & d_1\lambda - (\mu_1+1)^2=0,
\end{eqnarray*}
and similarly, since $\dim \Ker(A_\nu + M) = 2$, 
\begin{eqnarray*}
d_1 v + (\mu_1-1) \sigma\cdot \nu = 0 & \Longleftrightarrow & (\mu_1-1) v + \lambda \sigma\cdot \nu = 0\\
& \Longleftrightarrow & d_1\lambda - (\mu_1-1)^2=0,
\end{eqnarray*}
which implies that $\mu_1 = 0$ and $d_1\lambda = 1$, hence $\lambda>0$. 
\end{proof}

\begin{remark}
A similar characterization result can be proved in any dimension $n\ge 2$, {\it i.e.}, when the matrices $M$ belong to $\mathbb M^{(n+1) \times (n+1)}$.
\end{remark}

As a consequence of Lemma \ref{lem:BC}, it follows that all admissible boundary conditions \eqref{friedrichs_BC} for the dissipative formulation are of the form
\begin{equation}\label{eq:lambdaBC}
\sigma\cdot \nu+\lambda^{-1}\dot u=0 \quad \text{ on }\partial \O \times (0,T),
\end{equation}
where $\lambda:\partial \O\to (0,+\infty)$. In the sequel, we will assume for simplicity that $\lambda>0$ is independent of the space variable.

\begin{remark}
Note that, strictly speaking, (homogeneous) Dirichlet and Neumann conditions are not contained within this framework. However, they can be recovered by means of an asymptotic analysis as $\lambda\to 0^+$ and $\lambda \to +\infty$, respectively (see Section \ref{sec:lambda}). 

Moreover, since $\lambda$ is actually a (Borel) function of the space variable, letting $\lambda  \to 0$ in some subset $\Gamma_D \subset \O$, and $\lambda \to +\infty$ on its complementary $\Gamma_N:=\partial \O \setminus \Gamma_D$ would lead to mixed boundary conditions of Dirichlet type on $\Gamma_D$ and Neumann type on $\Gamma_N$. This problem will not be addressed in the present work.
\end{remark}

\section{The dynamic elasto-visco-plastic model}
\label{sec:3}

\noindent In order to establish the existence and uniqueness of solution to \eqref{general_model} and \eqref{eq:lambdaBC}, we consider an elasto-visco-plastic approximation model which consists in regularizing the constitutive law by means of a Kelvin-Voigt type visco-elasticity, and the flow rule thanks to a Perzyna type visco-plasticity. Except for our choice of boundary conditions \eqref{eq:lambdaBC}, the model described below is very similar to that studied in \cite{DalMasoScala} (see also \cite{Suquet,BMora}).

This choice of regularization is motivated by the approximation employed in \cite{DLS,BMS} in order to show the well-posedness of constrained Friedrichs' systems in the whole space. It consists in penalizing the constraint (which is described by Perzyna visco-plasticity), and adding up a diffusive term (which corresponds to Kelvin-Voigt visco-elasticity).

Note also that since the space dimension does not really matter in the subsequent arguments, we perform the analysis in any space dimension.

\medskip

The main result of this section is the following existence and uniqueness result.
\begin{theorem}
\label{final_theorem_at_level_epsilon}
Let $\Omega$ be a bounded open set of $\R^n$ with Lipschitz boundary and $\lambda>0$. Consider a source term $f \in H^1([0,T];L^2(\Omega))$ and an initial data $(u_0,v_0,\sigma_0,p_0) \in H^1(\Omega) \times H^2(\Omega) \times H(\Div,\O)  \times L^2(\Omega;\R^n) $ such that
\[
\begin{cases}
\nabla u_0 = \sigma_0 + p_0 \text{ in } L^2(\O;\R^n),\\
\sigma_0\cdot \nu +\lambda^{-1} v_0 = 0 \quad\HH^{n-1}\text{-a.e. on } \partial \O,\\
|\sigma_0| \leq 1 \text{ a.e. in } \Omega.
\end{cases}
\]
For each $\e>0$, we define $g_\e:=\e\nabla v_0\cdot \nu \in L^2(\partial \O)$. Then, there exist a unique triple $(u_\e,\sigma_\e,p_\e)$ with the regularity
$$
\begin{cases}
u_\e\in W^{2,\infty}([0,T];L^2(\Omega)) \cap H^2([0,T];H^1(\Omega)),\\
\sigma_\e \in W^{1,\infty}([0,T];L^2(\Omega;\R^n)),\\
p_\e \in  H^1([0,T];L^2(\Omega;\R^n)),
\end{cases}$$
which satisfies the following properties:

\begin{enumerate}
\item The initial conditions:
\[
u_\e(0)   =  u_0, \quad \dot  u_\e(0) = v_0, \quad 
\sigma_\e(0)   =  \sigma_0, \quad p_\e(0)   =  p_0;
\]
\item The additive decomposition: for all $t \in [0,T]$
\begin{equation}\label{decomposition_nablau_eps}
\nabla u_\e(t) = \sigma_\e(t) + p_\e(t)\quad \text{in }L^2(\O;\R^n);
\end{equation}
\item The equation of motion: 
\[\ddot u_\e - \Div(\sigma_\e+\e \nabla \dot u_\e) = f\quad  \text{ in } L^2(0,T;L^2(\Omega));\]
\item The dissipative boundary condition:
\begin{equation}\label{eq:bdry_cond_eps}
(\sigma_\e + \e \nabla \dot u_\e)\cdot\nu+\lambda^{-1}\dot u_\e = g_\e\quad \text{ in } L^2(0,T;L^2(\partial \Omega));
\end{equation}
\item The visco-plastic flow rule:
\begin{equation}\label{snd_eq_continuous}
\dot p_\e = \frac{\sigma_\e-P_B(\sigma_\e)}{\e}\quad \text{in }L^2(0,T;L^2(\O;\R^n)).
\end{equation}
\end{enumerate}
In addition, we have the following energy balance: for all $t\in [0,T]$,
\begin{multline}\label{eq:energy-balance}
\frac12\| \dot u_\e(t) \|^2_2 + \frac12\| \sigma_\e(t) \|^2_2+\e\left( \int_0^t \int_\Omega \left|\nabla \dot u_\e \right|^2\dd x \dd s + \int_0^t \int_\Omega \left| \dot p_\e \right|^2 \dd x \dd s\right)  + \frac{1}{\lambda} \int_0^t \int_{\partial \Omega} \left| \dot u_\e \right|^2\dd \mathcal{H}^{n-1}ds\\
+ \int_0^t \int_\Omega |\dot p_\e|\dd x \dd s =\frac12\|v_0\|_2^2 + \frac12\| \sigma_0 \|^2_2 + \int_0^t \int_\Omega f \dot u_\e\dd x \dd s+\int_0^t \int_{\partial \O} g_\e \dot u_\e\dd\HH^{n-1}\dd s,
\end{multline}
and the estimate
\begin{multline}\label{eq:apost_est}
\sup_{t\in [0,T]} \left\| \ddot u_\e(t) \right\|^2_2 + \sup_{t\in [0,T]} \left\| \dot \sigma_\e(t) \right\|^2_2+\e \int_0^T \left\| \nabla\ddot u_\e(t) \right\|^2_2\dt + \frac{1}{\lambda}\int_0^T \left\| \ddot u_\e(t) \right\|^2_{L^2(\partial \Omega)} \dt \\
\le  C\bigg(\|\Div(\sigma_0+\e\nabla v_0)+f(0)\|_2^2+  \| \nabla v_0 \|^2_2 + \Big(\int_0^T \| \dot f(t) \|_2\dt\Big)^2 \bigg),
\end{multline}
for some constant $C>0$ independent of $\e$ and $\lambda$.
\end{theorem}

The proof of Theorem~\ref{final_theorem_at_level_epsilon} is standard and follows the lines of {\it e.g.} \cite{DalMasoScala} for the existence and uniqueness in the energy space, and of \cite{BMora} for the additional regularity results \eqref{eq:apost_est}. We will not present the proof of that result, whose arguments can be easily reconstructed by the reader from the above mentioned references. The main difference with \cite{BMora,DalMasoScala} is concerned with the boundary condition. However, as long as $\e>0$, since the analysis takes place in Sobolev type spaces, this difference will not really matter.

\medskip

The proof of existence relies on a time discretization procedure. At each time step, we solve a minimization problem. After defining suitable interpolants, we derive some standard {\it a priori} estimates that allow one to obtain some weak convergences in the energy space as the time discretization parameter goes to zero. It enables one to establish the additive decomposition as well as the initial conditions for the displacement, the stress and the plastic strain. At this stage, the equation of motion and the boundary condition are just formulated in a weak sense. Nonetheless, it allows one to get the initial condition for the velocity using results on Banach-valued Sobolev spaces. Indeed, since $\dot u_\e \in L^2(0,T;H^1(\O))$ and $\ddot u_\e \in L^2(0,T;[H^1(\O)]')$, it follows from \cite[Theorem 1.19]{Barbu} that $\dot u_\e \in \mathcal C^0([0,T];L^2(\O))$. In order to get the flow rule, we derive a strong convergence result. Finally, we obtain, thanks to suitable {\it a posteriori} estimates, the equation of motion in a strong sense as well as the boundary condition. Note that the introduction of the term $g_\e$ allows one to get rid off undesirable boundary terms in the proof of estimate \eqref{eq:apost_est}.

\begin{remark}\label{rem:wf}
The equation of motion implies, for all $\varphi \in L^2(0,T;H^1(\O))$,
\begin{multline}
\label{eq_motion_eps_equal_zero}
\int_0^T \int_\Omega\ddot u_\e \varphi\dxdt + \int_0^T \int_\Omega \left(\sigma_\e+\e \nabla\dot  u_\e\right) \cdot \nabla \varphi\dxdt
 + \frac{1}{\lambda} \int_0^T \int_{\partial \Omega} \dot u_\e \varphi\dd \mathcal{H}^{n-1} \dt \\= \int_0^T \int_\Omega f\varphi\dxdt+\int_0^T \int_{\partial \Omega} g_\e\varphi\dd\HH^{n-1}\dt.
\end{multline}
\end{remark}

\section{The dynamic elasto-plastic model}
\label{sec:4}

The object of this section is to show the well-posedness of the model \eqref{general_model} and \eqref{eq:lambdaBC} by letting the viscosity parameter $\e$ tend to zero. It turns out that a relaxation phenomenon occurs, and it leads to a modification of the boundary condition which has to accomodate to the stress constraint. Indeed, since it is expected that $|\sigma|\leq 1$, the boundary condition $\sigma\cdot \nu+\lambda^{-1}\dot u=0$ can only be satisfied at the points of the boundary where $|\dot u|\leq \lambda$, while on the part of the boundary where $|\dot u|>\lambda$, the velocity has to be truncated by the values $\pm \lambda$.
This phenomena is easily explained by looking at the energy balance \eqref{eq:energy-balance}. In order to pass to the limit in this equality, we must (at least) ensure the  sequential lower semicontinuity of the mapping 
$$(u,\sigma) \mapsto \int_\O |\nabla u-\sigma| \dx +\frac{1}{2\lambda} \int_{\partial \O} |u|^2\dd \HH^{n-1}$$
with respect to a reasonable topology provided by the energy estimates. Unfortunately, this property fails according to the following result (see also \cite{MazonRossiSegura,Modica}).

\begin{proposition}\label{relax}
Let $\O \subset \R^n$ be a bounded open set with $\mathcal C^1$ boundary. Let us define the functional
$F:W^{1,1}(\O)\times L^2(\O;\R^n) \to [0,+\infty]$ by 
$$F(u,\sigma)=\int_\O |\nabla u-\sigma| \dx +\frac{1}{2\lambda} \int_{\partial \O} |u|^2\dd \HH^{n-1},$$
with the convention that $F(u,\sigma)=+\infty$ if $u \not\in L^2(\partial \O)$. Then, the lower semicontinuous enveloppe of $F$ with respect to the weak* convergence in $BV(\O)$ and the strong convergence in $L^2(\O;\R^n)$ is given by
$$\overline F(u,\sigma)=|Du-\sigma|(\O)+\int_{\partial \O} \psi_\lambda(u) \dd \HH^{n-1},$$
where  $\psi_\lambda : \R\to [0,+\infty)$ is defined by
\begin{equation}\label{eq:psi}
\psi_\lambda(z) = 
\begin{cases}
\frac{z^2}{2\lambda} &\text{if } |z|\le \lambda,\\
|z| - \frac{\lambda }{2} &\text{if } |z|\ge \lambda .
\end{cases}
\end{equation}
\end{proposition}

\begin{proof}
Let us fix $(u,\sigma) \in BV(\O) \times L^2(\O;\R^n)$.

\medskip

\noindent {\bf Step 1: Lower bound.} We must show that for every sequences $(u_k) \subset W^{1,1}(\O)$ and $(\sigma_k) \subset L^2(\O;\R^n)$ with $u_k \wto u$ weakly* in $BV(\O)$ and $\sigma_k \to \sigma$ strongly in $L^2(\O;\R^n)$, then
$$\overline F(u,\sigma) \leq \liminf_{k \to \infty} F(u_k,\sigma_k).$$

\noindent Since $\psi_\lambda \leq |\cdot|^2/(2\lambda)$, we first observe that
\begin{equation}\label{eq:ineq1}
\liminf_{k \to \infty} F(u_k,\sigma_k) \geq  \liminf_{k \to \infty} \overline F(u_k,\sigma_k).
\end{equation}
Possibly extracting a subsequence, we can assume without loss of generality that 
$$\liminf_{k \to \infty} \overline F(u_k,\sigma_k)=\lim_{k \to \infty} \overline F(u_k,\sigma_k)<+\infty,$$ 
and, up to another subsequence, we can also suppose that $|Du_k-\sigma_k| \wto \mu$ weakly* in $\mathcal M(\O)$ for some non-negative measure $\mu \in \mathcal M(\O)$.

The argument presented below is very close to that of \cite[Proposition 1.2]{Modica}. Let $\delta>0$, and $\zeta_\delta \in \mathcal C^\infty_c(\O;[0,1])$  be a cut-off function such that $\zeta_\delta=1$ on $A_\delta:=\{x \in \O:\dist(x,\partial \O)\geq \delta\}$ and $|\nabla \zeta_\delta|\leq 2/\delta$ in $\O$. We consider the function $w_{\delta,k}:=(1-\zeta_\delta)(u-u_k) \in BV(\O)$ which satisfies $w_{\delta,k}=u-u_k$  in a neighborhood of $\partial \O$,  $w_{\delta,k}=0$ in $A_\delta$, and $Dw_{\delta,k}=-(u-u_k)\nabla \zeta_\delta+(1-\zeta_\delta)(Du-Du_k)$ in $\O$. According to the trace inequality \eqref{eq:traceBV}, we infer that
\begin{multline*}
\int_{\partial \O} |u-u_k|\dd \HH^{n-1} \leq (1+\e)|Dw_{\delta,k}|(\O) + c_\e \int_\O |w_{\delta,k}|\dx\\
\leq \left(\frac{2(1+\e) }{\delta}+c_\e\right) \int_{\O\setminus A_\delta} |u-u_k|\dx + (1+\e)|Du-Du_k|(\O \setminus A_\delta).
\end{multline*}
Since the function $\psi_\lambda$ is $1$-Lipschitz, we have that
$$\int_{\partial \O}|\psi_\lambda(u)-\psi_\lambda(u_k)|\dd\HH^{n-1}\leq \int_{\partial \O} |u-u_k|\dd \HH^{n-1},$$
while
$$|Du-Du_k|(\O \setminus A_\delta) \leq |Du-\sigma|(\O \setminus A_\delta)+\int_{\O \setminus A_\delta}|\nabla u_k-\sigma_k|\dx +\int_\O|\sigma_k-\sigma|\, dx.$$
As a consequence,
\begin{multline*}
\overline F(u,\sigma)-\overline F(u_k,\sigma_k) \leq  C_{\e,\delta}\int_{\O \setminus A_\delta}|u-u_k|\dx + (1+\e)\int_\O |\sigma-\sigma_k|\dx\\
+(1+\e)|Du-\sigma|(\O\setminus A_\delta) + (1+\e)|Du_k-\sigma_k|(\O\setminus A_\delta)+|Du-\sigma|(\O) -|Du_k-\sigma_k|(\O).
\end{multline*}
Choosing a sequence $(\delta_j)_{j \in \mathbb N}$ with $\delta_j \searrow 0^+$ and $\mu(\partial A_{\delta_j})=0$ for all $j \in \mathbb N$, we get that $|Du_k-\sigma_k|(\O \setminus A_{\delta_j}) \to |Du-\sigma|(\O\setminus A_{\delta_j})$ as $k \to \infty$, and thus
$$\overline F(u,\sigma)-\liminf_{k \to \infty} \overline F(u_k,\sigma_k) 
\leq 2(1+\e)|Du-\sigma |(\O\setminus A_{\delta_j}) .$$
Finally, since $A_{\delta_j}$ is increasing to $\O$ as $j \to \infty$, we get that 
$$\overline F(u,\sigma) \leq \liminf_{k \to \infty} \overline F(u_k,\sigma_k),$$
which together with \eqref{eq:ineq1} completes the proof of the lower bound.

\medskip

\noindent {\bf Step 2: Upper bound.} We show the existence of sequences $(u_k) \subset W^{1,1}(\O)$ and $(\sigma_k) \subset L^2(\O;\R^n)$ such that
$u_k \wto u$ weakly* in $BV(\O)$, $\sigma_k \to \sigma$ strongly in $L^2(\O;\R^n)$, and
$$\limsup_{k \to \infty}F(u_k,\sigma_k)  \leq \overline F(u,\sigma).$$

This proof follows the lines of  \cite[Lemma 2.1]{BFM}. Let us denote by $\theta=\max(-\lambda,\min(u,\lambda))$ the truncation of the trace of $u$ on $\partial \O$ by the values $\pm \lambda$. Then $\theta \in L^\infty(\partial \O)$, and using Gagliardo's extension theorem \eqref{eq:lift}, for each $k \in \mathbb N^*$, one can find a function $w_k\in W^{1,1}(\O)$ such that $w_k=\theta -u$ on $\partial \O$,
$$\int_\O |w_k|\dx \leq \frac{1}{k} \int_{\partial \O} |\theta-u|\dd\HH^{n-1},$$
and
$$\int_\O |\nabla w_k|\dx \leq \left(1+\frac{1}{k}\right) \int_{\partial \O} |\theta-u|\dd\HH^{n-1}.$$
Applying next Proposition \ref{prop:MS}, there exists a sequence $(z_k) \subset W^{1,1}(\O)$ such that $z_k \wto u$ weakly* in $BV(\O)$, $|Dz_k-\sigma|(\O) \to |Du-\sigma|(\O)$ and $z_k=u$ on $\partial \O$ for each $k$.  Setting $\sigma_k\equiv \sigma$ and $u_k:=w_k+z_k \in W^{1,1}(\O)$, then $u_k=\theta$ on $\partial\O$, $u_k \wto u$ weakly* in $BV(\O)$, and
$$\limsup_{k \to \infty} \int_\O |\nabla u_k-\sigma|\dx \leq |Du-\sigma|(\O) + \int_{\partial \O}|u-\theta|\dd \HH^{n-1}.$$
Consequently,
\begin{multline*}
\limsup_{k \to \infty}F(u_k,\sigma_k)\leq |Du-\sigma|(\O) + \int_{\partial \O}|u-\theta|\dd \HH^{n-1}+\frac{1}{2\lambda}\int_{\partial\O}|\theta|^2\dd\HH^{n-1}\\
= |Du-\sigma|(\O) + \int_{\partial \O\cap \{u \leq -\lambda\}}(-u-\lambda)\dd \HH^{n-1}+ \int_{\partial \O\cap \{u\geq \lambda\}}(u-\lambda)\dd \HH^{n-1}\\
+\frac{1}{2\lambda}\int_{\partial\O\cap \{|u|\leq \lambda\}}|u|^2\dd\HH^{n-1}+\frac{\lambda}{2}\HH^{n-1}(\{|u|>\lambda\})=\overline F(u,\sigma),
\end{multline*}
which concludes the proof of the upper bound.
\end{proof}

\begin{remark}
In the proof of the lower bound we strongly used the $\mathcal C^1$ regularity of the boundary. Indeed \cite[Remark 1.3]{Modica} shows that energy functionals of the form
$$BV(\O) \ni u \mapsto |Du|(\O)+\int_\O \psi(u)\dd\HH^{n-1},$$
with $\O \subset \R^n$ only Lipschitz, and $\psi:\R \to \R$ $1$-Lipschitz, might fail to be   sequentially weakly* lower semicontinuous in $BV(\O)$. On the other hand, the use of the $\mathcal C^1$ character of the boundary does not seem to be necessary in the proof of the upper bound. Indeed, as observed in \cite[Lemma 2.1]{BFM}, Gagliardo's extension result with estimates as in \eqref{eq:lift} holds for Lipschitz boundaries  as well.
\end{remark}

\begin{remark}\label{rem:psi}
Let us observe for future use that the function $\psi_\lambda$ defined in \eqref{eq:psi} is convex, $1$-Lipschitz, and of class $\mathcal C^1$, with
$$\psi'_\lambda(z)=
\begin{cases}
-1 & \text{ if }z\leq -\lambda,\\
\frac{z}{\lambda} & \text{ if }  |z|<\lambda,\\
1 & \text{ if }  z \geq \lambda.
\end{cases}$$
Moreover, its convex conjugate is given, for all $z \in \R$, by
\begin{equation}\label{psi*}
\psi^*_\lambda(z)=\frac{\lambda}{2} |z|^2+I_{[-1,1]}(z),
\end{equation}
where $I_{[-1,1]}$ is the indicator function of the interval $[-1,1]$ which is equal to $0$ in $[-1,1]$, and $+\infty$ outside.
\end{remark}

We now state the main result of this section.

\begin{theorem}
\label{final_theorem_eps_eq_zero}
Let $\Omega\subset \R^n$ be a bounded open set with $\mathcal C^1$ boundary and $\lambda>0$. Consider a source term $f \in H^1([0,T];L^2(\Omega))$ and an initial data $(u_0,v_0,\sigma_0,p_0) \in H^1(\O) \times H^2(\Omega) \times H(\Div,\O) \times  L^2(\Omega;\R^n) $ such that
\begin{equation}\label{eq:initial_conditions}
\begin{cases}
\nabla u_0 = \sigma_0 + p_0 \text{ in }L^2(\O;\R^n),\\
\sigma_0\cdot \nu + \lambda^{-1} v_0 = 0\quad\HH^{n-1}\text{ on }\partial \Omega,\\
|\sigma_0|\leq 1 \text{ a.e. in }\Omega.
\end{cases}
\end{equation}
Then there exist a unique triple $(u,\sigma,p)$ with the regularity
$$
\begin{cases}
u\in W^{2,\infty}([0,T];L^2(\Omega)) \cap \mathcal C^{0,1}([0,T];BV(\Omega)),\\
\sigma \in W^{1,\infty}([0,T];L^2(\Omega;\R^n)),\\
p \in  \mathcal C^{0,1}([0,T];\mathcal M(\Omega;\R^n)),
\end{cases}$$
which satisfies the following properties:
\begin{enumerate}
\item The initial conditions:
\[
u(0)   =  u_0, \quad \dot  u(0) = v_0, \quad 
\sigma(0)   =  \sigma_0, \quad p(0)   =  p_0;
\]
\item The additive decomposition: for all $t \in [0,T]$
\[
D u(t) = \sigma(t) + p(t)\quad \text{in }\mathcal M(\O;\R^n);
\]
\item The equation of motion: 
$$\ddot u - \Div\sigma = f \quad  \text{ in } L^2(0,T;L^2(\Omega));$$
\item The relaxed boundary condition:
$$
\sigma\cdot\nu+\psi'_\lambda(\dot u) = 0\quad \text{ in } L^2(0,T;L^2(\partial \Omega));
$$
\item The stress constraint: for all $t \in [0,T]$,
$$|\sigma(t)|\leq 1 \quad \text{ a.e. in }\O;$$
\item The flow rule: for a.e. $t \in [0,T]$,
$$|\dot p(t)|=[\sigma(t) \cdot \dot p(t)] \quad \text{ in }\mathcal M(\O).$$
\end{enumerate}
In addition, we have the following energy balance: for all $t\in [0,T]$,
\begin{multline}\label{eq:energy-balance2}
\frac12 \|\dot u(t)\|_2^2+\frac12 \|\sigma(t)\|_2^2 + \int_0^t |\dot p(s)|(\Omega)\ds+ \frac{\lambda}{2}\int_0^t \int_{\partial \Omega} | \sigma \cdot \nu |^2\dd \mathcal{H}^{n-1} \ds + \int_0^t \int_{\partial \Omega} \psi_\lambda (\dot u )\dd \mathcal{H}^{n-1}\ds \\
 = \frac12 \|v_0\|_2^2+\frac12 \|\sigma_0\|_2^2  + \int_0^t\int_{\Omega} f \dot u \dx\ds.
\end{multline}
\end{theorem}

\begin{remark}
In the sequel we will refer to the solution given by Theorem \ref{final_theorem_eps_eq_zero} as the {\it variational solution to the elasto-plastic problem} associated to the initial data $(u_0,v_0,\sigma_0,p_0)$ and the source term $f$. Unless otherwise specified, we always assume in the sequel that $f \in H^1([0,T];L^2(\Omega))$ and $(u_0,v_0,\sigma_0,p_0) \in H^1(\Omega) \times H^2(\Omega) \times H(\Div,\O) \times  L^2(\Omega;\R^n)$ satisfy \eqref{eq:initial_conditions}.
\end{remark}

The rest of this section is devoted to prove Theorem~\ref{final_theorem_eps_eq_zero}. We  consider the unique solution $(u_\e,\sigma_\e,p_\e)$ to the elasto-visco-plastic problem (given by Theorem \ref{final_theorem_at_level_epsilon}) associated to  the initial condition $(u_{0},v_{0},\sigma_0,p_{0})$ and the source terms $f$ and $g_\e=\e\nabla v_0\cdot \nu$. Using the estimates obtained in Theorem~\ref{final_theorem_at_level_epsilon}, we derive weak convergences which enable one to get the initial conditions, the equation of motion and the stress constraint. We then obtain in Subsection~\ref{subsec:4.3} some strong convergence results despite we did not yet identify the correct boundary condition. Together with the relaxation result Proposition \ref{relax}, it allows one to derive a first energy inequality between two arbitrary times. In Subsection~\ref{subsec:4.5} we show that this inequality is actually an equality  which leads to the flow rule in a measure theoretic sense, and the relaxed boundary condition. Eventually, uniqueness of the solution is established in Subsection~\ref{subsec:4.8} as a consequence of a Kato inequality which states a comparison principle between two solutions.

\subsection{Weak convergences}\label{subsec:4.1}

In the proof of Theorem \ref{final_theorem_at_level_epsilon}, we have established the following estimate (see \eqref{eq:apost_est})
\begin{multline}\label{estimate_snd_derivative_epsilon}
\sup_{t\in [0,T]} \left\| \ddot u_\e(t) \right\|^2_2 + \sup_{t\in [0,T]} \left\| \dot \sigma_\e(t) \right\|^2_2 + \frac{1}{\lambda}\int_0^T \left\| \ddot u_\e(t) \right\|^2_{L^2(\partial \Omega)} \dt \\
\le  C\bigg(\|\Div(\sigma_0+\e\nabla v_{0})+f(0)\|_2^2
+  \| \nabla v_{0} \|^2_2 + \Big(\int_0^T \| \dot f(t) \|_2\dt\Big)^2 \bigg),
\end{multline}
while the energy balance \eqref{eq:energy-balance} gives
\begin{multline}
\label{estimate_first_derivative_epsilon}
\sup_{t\in [0,T]} \|\dot u_\e(t) \|^2_2 + \sup_{t\in [0,T]} \|\sigma_\e (t)\|_2^2+ \frac{1}{{\lambda}} \int_0^T \| \dot u_\e(t) \|^2_{L^2(\partial \O)}\dt \\
+ {\e} \int_0^T\| \nabla \dot u_\e(t) \|^2_2\dt+ {\e} \int_0^T\| \dot p_\e(t) \|_2^2\dt+ \int_0^T \int_\Omega |\dot p_\e|\dxdt \\
\le  C \left( \|v_{0}\|_2^2 + \|\sigma_0\|^2_2+ \Big(\int_0^T\| f (t)\|_2\dt\Big)^2 +\lambda\e \|\nabla v_0\cdot \nu\|_{L^2(\partial \O)}^2\right),
\end{multline}
where the constants $C>0$ occuring in \eqref{estimate_snd_derivative_epsilon} and \eqref{estimate_first_derivative_epsilon} are independent of $\e$ and $\lambda$.
Using also that $u_\e \in W^{2,\infty}([0,T];L^2(\O))$ and $u_{0} \in H^1(\O)$, then
\[
u_\e(t) = u_{0} + \int_0^t \dot u_\e(s) \ds \quad \text{ for all }t \in [0,T],
\]
where the integral is intended as a Bochner integral in $L^2(\O)$, and we get
\begin{equation}
\label{u_eps_W2inf_t_L2_x_estim}
\sup_{\e>0}\| u_\e \|_{W^{2,\infty}([0,T];L^2(\O))} <\infty.
\end{equation}
Arguing similarly yields
\begin{equation}
\label{trace_u_eps_H2_t_L2_x_estim}
\sup_{\e>0} \| u_\e \|_{H^2([0,T]; L^2(\partial \Omega))} <\infty,
\end{equation} 
and
\begin{equation}
\label{sigma_eps_W2inf_t_L2_x_estim}
\sup_{\e>0}\| \sigma_\e \|_{W^{1,\infty}([0,T];L^2(\O))} <\infty.
\end{equation}

Thanks to the estimates~\eqref{u_eps_W2inf_t_L2_x_estim}, \eqref{trace_u_eps_H2_t_L2_x_estim} and \eqref{sigma_eps_W2inf_t_L2_x_estim} we can extract a subsequence (not relabeled), and find functions $u\in W^{2,\infty}([0,T];L^2(\O))$, $\sigma \in W^{1,\infty}([0,T];L^2(\O;\R^n))$ and $w \in H^2([0,T];L^2(\partial \Omega))$  such that
\begin{equation}\label{eq:conv_0}
\begin{cases}
u_\e \wto u &\text{weakly* in } W^{2,\infty}([0,T];L^2(\O)), \\
\sigma_\e \wto \sigma &\text{weakly* in } W^{1,\infty}([0,T];L^2(\O;\R^n)), \\
u_\e \wto w &\text{weakly in } H^2([0,T];L^2(\partial \Omega)).
\end{cases}
\end{equation}
Moreover, according to \cite[Lemma 2.7]{DalMasoScala}, for every $t\in [0,T]$,
\begin{equation}\label{eq:conv_1}
\begin{cases}
u_\e(t) \wto u(t) & \text{weakly in } L^2(\O), \\
\dot u_\e(t) \wto \dot u(t) &\text{weakly in } L^2(\O), \\
\sigma_\e(t) \wto \sigma(t) &\text{weakly in } L^2(\O;\R^n), \\
u_\e(t) \wto w(t) &\text{weakly in } L^2(\partial \Omega), \\
\dot u_\e(t) \wto \dot w(t) &\text{weakly in } L^2(\partial \Omega).
\end{cases}
\end{equation}

\medskip

In order to derive weak compactness of the sequence of plastic strains, we use the energy balance \eqref{eq:energy-balance} between two arbitrary times $0 \leq t_1<t_2 \leq T$, which leads to
\begin{multline}
\label{energy_equa_level_eps2}
\int_{t_1}^{t_2} \int_\Omega |\dot p_\e|\dx\ds  \leq  \int_{t_1}^{t_2} \int_\Omega f \dot u_\e\dx\ds+ \int_{t_1}^{t_2} \int_{\partial \Omega} g_\e \dot u_\e\dd\HH^{n-1}\ds\\
+\frac{1}{2}\left(\| \dot u_\e(t_1) \|^2_2-\| \dot u_\e(t_2) \|_2^2 \right) + \frac{1}{2}\left( \| \sigma_\e(t_1) \|_2^2 - \| \sigma_\e(t_2) \|_2^2 \right).
\end{multline}
Using the fact that $f\in L^\infty(0,T;L^2(\O))$ and $(\dot u_\e)_{\e>0}$ is uniformly bounded in $L^\infty(0,T;L^2(\O))$ by \eqref{u_eps_W2inf_t_L2_x_estim}, we infer that
\begin{equation}
\label{source_term_eps_time_estim}
\left| \int_{t_1}^{t_2} \int_\Omega f \dot u_\e\dx\dd s \right| \le C(t_2-t_1),
\end{equation}
for some constant $C>0$ independent of $\e$. Then, using that $g_\e=\e\nabla v_0 \cdot \nu \in L^2(\partial \O)$ and $(\dot u_\e)_{\e>0}$ is uniformly bounded in $L^\infty(0,T;L^2(\partial \O))$ by \eqref{trace_u_eps_H2_t_L2_x_estim}, 
\begin{equation}
\label{g_eps_time_estim}
\left|\int_{t_1}^{t_2} \int_{\partial \Omega} g_\e \dot u_\e\dd\HH^{n-1}\ds\right| \leq C(t_2-t_1),
\end{equation}
for some constant $C>0$ independent of $\e$. Next, using again \eqref{u_eps_W2inf_t_L2_x_estim} yields
\begin{eqnarray}\label{eq:estim2}
 \left|  \| \dot u_\e(t_1) \|^2_2-\| \dot u_\e(t_2) \|_2^2 \right|   & =&  \left| \int_\O (\dot u_\e(t_1)-\dot u_\e(t_2))(\dot u_\e(t_1)+\dot u_\e(t_2))\dx \right|  \nonumber\\
& \leq & C \|\dot u_\e(t_2)-\dot u_\e(t_1)\|_2 \leq C(t_2-t_1),
\end{eqnarray}
where again, $C>0$ is independent of $\e$.  Similarly, using  \eqref{sigma_eps_W2inf_t_L2_x_estim} leads to
\begin{equation}\label{eq:estim3}
 \left| \| \sigma_\e(t_1) \|^2_2-\| \sigma_\e(t_2) \|_2^2 \right|  \leq C \|\sigma_\e(t_2)-\sigma_\e(t_1)\|_2 \leq C(t_2-t_1),
\end{equation}
Gathering \eqref{energy_equa_level_eps2}--\eqref{eq:estim3} and using Jensen's inequality yields
$$\int_\O |p_\e(t_2)-p_\e(t_1)|\dx \le C (t_2-t_1).$$
It is thus possible to apply Ascoli-Arzela Theorem to get, up to another subsequence independent of time, the existence of $p \in \mathcal C^{0,1}([0,T];\mathcal M(\O;\R^n))$ such that for all $t \in [0,T]$,
\begin{equation}\label{eq:conv_2}
p_\e(t) \wto p(t) \quad\text{ weakly* in }\mathcal M(\O;\R^n).
\end{equation}

Using next the decomposition $\nabla u_\e = p_\e + \sigma_\e$ the convergences \eqref{eq:conv_1} and \eqref{eq:conv_2} and the already established regularity properties for $(u,\sigma,p)$, we obtain that   $u \in \mathcal C^{0,1}([0,T];BV(\O))$, and for all $t \in [0,T]$,
\begin{equation}\label{eq:conv_3}
u_\e(t) \wto u(t) \quad\text{ weakly* in }BV(\O).
\end{equation}

\begin{remark}
Let us stress that, as $u(t) \in BV(\O)$ for all $t \in [0,T]$, its trace is well defined as an element of $L^1(\partial \O)$. However, the trace mapping from $BV(\O)$ to $L^1(\partial \O)$ is not sequentially weakly* continuous, and therefore we cannot ensure that $w(t)$ is the trace of $u(t)$.
\end{remark}

\noindent \textit{The initial condition.}
Since $u_\e(0)=u_{0}$, $\dot u_\e(0)= v_{0}$, $\sigma_\e(0)=\sigma_0$, $p_\e(0)=p_{0}$  for all $\e>0$, we obtain that $u(0)=u_0$, $\dot u(0) = v_0$, $\sigma(0)=\sigma_0$ and $p(0)= p_0$.

\noindent\textit{The additive decomposition.}
Using the additive decomposition~\eqref{decomposition_nablau_eps} and the weak convergences \eqref{eq:conv_1}, \eqref{eq:conv_2} and \eqref{eq:conv_3}, we infer that  for all $t \in [0,T]$,
$$Du(t)=\sigma(t)+p(t) \text{ in } \mathcal{M}(\O;\R^n).$$

\noindent\textit{The stress constraint.}
Let $\tau_\e:=P_B(\sigma_\e)$. Since $\|\tau_\e\|_{L^\infty(\O \times (0,T))}\leq 1$, we can assume, up to another subsequence, that $\tau_\e \wto \tau$ weakly* in $L^\infty(\O \times (0,T);\R^n)$ with $\|\tau\|_{L^\infty(\O \times (0,T))} \leq 1$. According to estimate~\eqref{estimate_first_derivative_epsilon} and the flow rule \eqref{snd_eq_continuous}, we get that
$$\int_0^T \|\sigma_\e(t)-\tau_\e(t)\|^2_2\dt \leq C\e \to 0,$$
so that $\sigma=\tau$. Consequently, for all $t \in [0,T]$, we have
$$|\sigma(t)|\leq 1 \quad\text{ a.e. in }\O.$$

\noindent\textit{The equation of motion.}
According to \eqref{estimate_first_derivative_epsilon}, we get that $\e \nabla \dot u_\e \to 0$ strongly in $L^2(0,T;L^2(\O;\R^n))$. We thus deduce that $\sigma_\e+\e \nabla \dot u_\e \to \sigma$ strongly in $L^2(0,T;L^2(\O;\R^n))$, and since from the equation of motion at fixed $\e$, one has  $\Div(\sigma_\e + \e \nabla \dot u_\e) = f-\ddot u_\e$, we obtain thanks to the estimate~\eqref{u_eps_W2inf_t_L2_x_estim} that
$$\sup_{\e>0} \| \sigma_\e + \e \nabla \dot u_\e \|_{L^2(0,T;H(\Div,\O))}<+\infty.$$
As a consequence $\Div(\sigma_\e + \e \nabla \dot u_\e)  \wto \Div\sigma$ weakly in $L^2(0,T;L^2(\O))$, and
$$\ddot u-\Div \sigma=f \quad\text{ in }L^2(0,T;L^2(\O)).$$
Note also that $(\sigma_\e+\e \nabla \dot u_\e) \cdot \nu \wto \sigma \cdot \nu$ weakly in $L^2(0,T;H^{-1/2}(\partial \Omega))$. Using the boundary condition at fixed $\e$, we also have that $(\sigma_\e+\e \nabla \dot u_\e) \cdot \nu=g_\e-\lambda^{-1}\dot u_\e$ which is bounded in $L^2(0,T;L^2(\partial \Omega))$ according to \eqref{trace_u_eps_H2_t_L2_x_estim}, and consequently,
\begin{equation}
\label{weak_cv_sigma_eps_nu_L2_t_L2_dOmega}
(\sigma_\e+\e \nabla \dot u_\e) \cdot \nu \wto \sigma \cdot \nu\quad \text{in } L^2(0,T;L^2(\partial \Omega)).
\end{equation}
\begin{remark}\label{rem:wf1}
Passing to the limit in \eqref{eq_motion_eps_equal_zero}, we get for all $\varphi \in L^2(0,T;H^1(\O))$,
\begin{equation}
\label{eq_ueps_with_global_regularity}
\int_0^T \int_\Omega\ddot u \varphi\dxdt + \int_0^T \int_\Omega\sigma \cdot \nabla \varphi\dxdt + \frac{1}{\lambda} \int_0^T \int_{\partial \Omega} \dot w \varphi\dd \mathcal{H}^{n-1} \dt = \int_0^T \int_\Omega f\varphi\dxdt.
\end{equation}
\end{remark}

\subsection{Strong convergences}
\label{subsec:4.3}

At this stage,  it still remains to prove the boundary condition and the flow rule. To do that, we need to improve some of the weak convergences established above into strong convergences. This is the object of the following result.

\begin{proposition}\label{prop:strong_conv}
The following strong convergences hold:
\[
\begin{cases}
\dot u_\e \to\dot u &\text{strongly in } \mathcal C^0([0,T];L^2(\O)), \\
\sigma_\e \to \sigma &\text{strongly in } \mathcal C^0([0,T];L^2(\O;\R^n)), \\
\sqrt{\e} \nabla \dot u_\e \to 0 &\text{strongly in } L^2(0,T;L^2(\O;\R^n)), \\
\dot u_\e \to \dot w &\text{strongly in } L^2(0,T;L^2(\partial \Omega)).
\end{cases}
\]
\end{proposition}

\begin{proof}
Substracting equations \eqref{eq_motion_eps_equal_zero} to \eqref{eq_ueps_with_global_regularity}, and taking $\varphi:=\mathds{1}_{[0,t]} \dot u_\e \in L^2(0,T;H^1(\O))$ where $t \in [0,T]$ as test function,  we get that
\begin{multline*}
\int_0^t \int_\Omega (\ddot u_\e - \ddot u) \dot u_\e\dx\ds + \int_0^t \int_\Omega (\sigma_\e-\sigma) \cdot \nabla \dot u_\e\dx\ds \\
+ \e \int_0^t \int_\Omega \left| \nabla \dot u_\e \right|^2\dx\ds + \frac{1}{\lambda} \int_0^t \int_{\partial \Omega} (\dot u_\e-\dot w) \dot u_\e\dd \mathcal{H}^{n-1} \ds=\int_0^t \int_{\partial \O}g_\e \dot u_\e\dd\HH^{n-1}\ds.
\end{multline*}
Thanks to the additive decomposition~\eqref{decomposition_nablau_eps}, we have
\[
\int_0^t \int_\Omega (\sigma_\e-\sigma) \cdot \nabla \dot u_\e\dx\ds = \int_0^t \int_\Omega (\sigma_\e-\sigma)\cdot \dot p_\e\dx\ds+  \int_0^t \int_\Omega (\sigma_\e-\sigma) \cdot \dot \sigma_\e\dx\ds.
\]
According to the flow rule~\eqref{snd_eq_continuous} and the fact that  for all $t \in [0,T]$, $\|\sigma(t)\|_\infty\leq 1$, we deduce that
\[
\int_0^t \int_\Omega (\sigma_\e-\sigma) \cdot \dot p_\e\dx\ds \ge 0,
\]
and thus,
\begin{multline*}
\int_0^t \int_\Omega (\ddot u_\e - \ddot u) (\dot u_\e-\dot u)\dx\ds  + \e \int_0^t \int_\Omega \left| \nabla \dot u_\e \right|^2\dx\ds\\
 + \int_0^t\int_\Omega (\sigma_\e-\sigma) \cdot ( \dot \sigma_\e-\dot \sigma)\dx\ds+ \frac{1}{\lambda}\int_0^t \int_{\partial \Omega} (\dot u_\e-\dot w)^2\dd \mathcal{H}^{n-1} \ds \\
 \le 
 - \int_0^t \int_\Omega (\ddot u_\e - \ddot u) \dot u\dx\ds  - \int_0^t \int_\Omega (\sigma_\e-\sigma) \cdot \dot \sigma\dx\ds  \\
 - \frac{1}{\lambda}\int_0^t \int_{\partial \Omega} (\dot u_\e-\dot w)\dot w\dd \mathcal{H}^{n-1} \ds+\int_0^t \int_{\partial \O}g_\e \dot u_\e\dd\HH^{n-1}\ds.
\end{multline*}
The weak convergences \eqref{eq:conv_0} imply that the right hand side of the previous inequality tends to $0$ as $\e\to 0$. Thus, noticing that
$$\int_0^t \int_\Omega (\ddot u_\e - \ddot u) (\dot u_\e-\dot u)\dx\ds
= \frac{1}{2}  \|\dot u_\e(t) - \dot u(t)\|^2_2,$$
and
$$\int_0^t \int_\Omega (\sigma_\e-\sigma) \cdot (\dot \sigma_\e-\dot \sigma)\dx\ds = \frac{1}{2}  \|\sigma_\e(t) - \sigma(t)\|_2^2,$$
we get the desired strong convergences. 
\end{proof}

An important consequence of the strong convergences is the derivation of an energy inequality between two arbitrary times. The following result makes use of the lower bound inequality established in Proposition \ref{relax}.

\begin{proposition}\label{prop:nrjineq}
For every $0\le t_1 \le t_2 \le T$,
\begin{multline*}
\frac{1}{2}\|\dot u(t_2)\|_2^2 + \frac12 \|\sigma(t_2)\|_2^2+\frac{\lambda}{2} \int_{t_1}^{t_2}\int_{\partial \O} |\sigma\cdot\nu|^2\dd\HH^{n-1}\ds
 +  \int_{t_1}^{t_2} \int_{\partial \Omega}\psi_\lambda(\dot  u)\dd \mathcal{H}^{n-1}\ds+ \int_{t_1}^{t_2} |\dot p(s)|(\O)\ds\\
\le \frac{1}{2}\|\dot u(t_1)\|_2^2 + \frac12 \|\sigma(t_1)\|_2^2 + \int_{t_1}^{t_2} \int_\Omega f \dot u\dx\ds.
\end{multline*}
In addition, for a.e. $t \in [0,T]$,
$$\frac{\lambda}{2} \int_{\partial \O}|\sigma(t)\cdot\nu|^2\dd\HH^{n-1}
+\int_{\partial \O} \psi_\lambda (\dot u(t))\dd \mathcal{H}^{n-1} +  |\dot p(t)|(\O)\le [\sigma(t)\cdot \dot p(t)](\O)-\int_{\partial \O}(\sigma(t)\cdot\nu) \dot u(t)\dd\HH^{n-1}.$$
\end{proposition}

\begin{proof}
Using the energy balance~\eqref{eq:energy-balance} and  the boundary condition, we have for all $0\le t_1 \le t_2 \le T$
\begin{multline}\label{1st_step_modica}
\frac{1}{2}\|\dot u_\e(t_2)\|_2^2 + \frac12 \|\sigma_\e(t_2)\|_2^2 + \frac{\lambda}{2} \int_{t_1}^{t_2} \int_{\partial \Omega}|(\sigma_\e+\e \nabla \dot u_\e)\cdot \nu-g_\e|^2 \dd \mathcal{H}^{n-1}\ds\\
+\frac{1}{2\lambda} \int_{t_1}^{t_2} \int_{\partial \Omega}|\dot u_\e|^2 \dd \mathcal{H}^{n-1}\ds+ \int_{t_1}^{t_2} \int_\O |\dot p_\e|\dx\ds\\
\le \frac{1}{2}\|\dot u_\e(t_1)\|_2^2 + \frac12 \|\sigma_\e(t_1)\|_2^2 + \int_{t_1}^{t_2} \int_\Omega f \dot u_\e\dx\ds+ \int_{t_1}^{t_2} \int_{\partial \O} g_\e \dot u_\e\dd\HH^{n-1}\ds.
\end{multline}
By Jensen's inequality, we have
\begin{multline*}
\frac{1}{t_2-t_1}\left(\frac{1}{2\lambda} \int_{t_1}^{t_2} \int_{\partial \Omega}|\dot u_\e|^2 \dd \mathcal{H}^{n-1}\ds+ \int_{t_1}^{t_2} \int_\O |\dot p_\e|\dx\ds\right)\\
\geq  \frac{1}{2\lambda} \int_{\partial \Omega} \left(\frac{ u_\e(t_2)- u_\e(t_1)}{t_2-t_1}\right)^2\dd \mathcal{H}^{n-1} + \int_\Omega \left|\frac{ p_\e(t_2)- p_\e(t_1)}{t_2-t_1}\right|\dx,
\end{multline*}
where we used that, as Bochner integrals,
$$\int_{t_1}^{t_2} \dot u_\e(s) \ds= u_\e(t_2)- u_\e(t_1) \quad \text{ in }L^2(\partial \O),$$
and
$$\int_{t_1}^{t_2} \dot p_\e(s)\ds= p_\e(t_2)- p_\e(t_1) \quad \text{ in }L^2(\O;\R^n).$$
Using~\eqref{eq:conv_3} and Proposition \ref{prop:strong_conv}, we know that 
\[
\frac{u_\e (t_2) - u_\e (t_1)}{t_2-t_1} \wto \frac{u(t_2)-u(t_1)}{t_2-t_1} \quad \text{weakly* in } BV(\O),
\]
and
\[
\frac{\sigma_\e (t_2) - \sigma_\e (t_1)}{t_2-t_1} \to \frac{\sigma(t_2)-\sigma(t_1)}{t_2-t_1}\quad\text{ strongly in } L^2(\O;\R^n),
\]
from which we deduce, according to Proposition~\ref{relax}, that
\begin{multline*}
\int_{\partial \Omega} \psi_\lambda \left(\frac{u(t_2)- u(t_1)}{t_2-t_1}\right)\dd \mathcal{H}^{n-1} +  \left|\frac{p(t_2)- p(t_1)}{t_2-t_1}\right|(\O)\\
 \le \liminf_{\e \to 0}\left\{\frac{1}{2\lambda} \int_{\partial \Omega}\! \left(\frac{u_\e(t_2)- u_\e(t_1)}{t_2-t_1}\right)^2\! \!\!\dd \mathcal{H}^{n-1}\! + \!\int_\Omega \left|\frac{p_\e(t_2)- p_\e(t_1)}{t_2-t_1}\right|\! \dx \right\}.
\end{multline*}
Therefore, thanks to the strong convergences results established in Proposition \ref{prop:strong_conv} together with the weak convergence \eqref{weak_cv_sigma_eps_nu_L2_t_L2_dOmega} of the normal trace, it is possible to pass to the (lower) limit in \eqref{1st_step_modica} to get that
\begin{multline}\label{2nd_step_modica}
\frac{1}{2}\frac{\|\dot u(t_2)\|_2^2-\|\dot u(t_1)\|_2^2}{t_2-t_1} + \frac12 \frac{\|\sigma(t_2)\|_2^2- \|\sigma(t_1)\|_2^2}{t_2-t_1}
 + \frac{\lambda}{2(t_2-t_1)} \int_{t_1}^{t_2} \int_{\partial \Omega}|\sigma\cdot \nu|^2\! \dd \mathcal{H}^{n-1}\ds\\
+ \int_{\partial \Omega} \psi_\lambda \left(\frac{u(t_2)- u(t_1)}{t_2-t_1}\right)\dd \mathcal{H}^{n-1} +  \left|\frac{p(t_2)- p(t_1)}{t_2-t_1}\right|(\O)
\le  \frac{1}{t_2-t_1}\int_{t_1}^{t_2} \int_\Omega f \dot u\dx\ds.
\end{multline}
Since $\dot u \in W^{1,\infty}([0,T];L^2(\O))$ and $\sigma  \in W^{1,\infty}([0,T];L^2(\O;\R^n))$, using~\cite[Theorem 1.19]{Barbu}, we obtain that functions $t\mapsto \| \dot u(t) \|^2_2$ and $t\mapsto \| \sigma(t) \|^2_2$ are absolutely continuous, and for a.e. $t \in [0,T]$
\begin{equation}\label{eq:abs_cont}
\frac{\dd }{\dd t} \| \dot u(t) \|^2_2= 2\int_\O\dot u(t) \ddot u(t)\dx, \quad \frac{\dd }{\dd t} \| \sigma(t) \|^2_2= 2\int_\O\sigma(t)\cdot \dot \sigma(t)\dx.
\end{equation}
In addition, since $u\in \mathcal{C}^{0,1}([0,T],BV(\Omega))$ then according to \cite[Theorem 7.1]{DalMasoDeSimoneMora}, for a.e. $t \in [0,T]$,
\[
\frac{u(s)-u(t)}{s-t} \wto \dot u(t) \quad \text{weakly* in }BV(\O)\text{ as }s \to t,
\]
while the fact that $\sigma \in W^{1,\infty}([0,T],L^2(\O;\R^n))$ ensures, according to \cite[Corollaire A.2]{BrezisOperateurs}, that for a.e. $t \in [0,T]$,
\[
\frac{\sigma(s)-\sigma(t)}{s-t} \to \dot \sigma(t) \quad \text{strongly in }L^2(\O;\R^n)\text{ as }s \to t.
\]
Consequently, we can pass to the limit in \eqref{2nd_step_modica} as $t_2 \to t_1=t$. Applying again Proposition \ref{relax} yields, for a.e. $t \in [0,T]$,
\begin{multline}\label{final_modica}
\int_\O \dot u(t) \ddot u(t)\dx+\int_\O\sigma(t)\cdot \dot \sigma(t)\dx+\frac{\lambda}{2} \int_{\partial \O}|\sigma(t)\cdot\nu|^2\dd\HH^{n-1}\\
+\int_{\partial \O} \psi_\lambda (\dot u(t))\dd \mathcal{H}^{n-1} +  |\dot p(t)|(\O)\le   \int_\Omega f(t) \dot u(t)\dx.
\end{multline}
On the one hand, using the equation of motion together with the Definition \ref{def_Sigma_p}
of duality and the integration by parts formula stated in Proposition \ref{IPP_sig_u_varphi} (with $\varphi=1$), infer that for a.e. $t\in [0,T]$,
$$\frac{\lambda}{2} \int_{\partial \O}|\sigma(t)\cdot\nu|^2\dd\HH^{n-1}
+\int_{\partial \O} \psi_\lambda (\dot u(t))\dd \mathcal{H}^{n-1} +  |\dot p(t)|(\O)\le [\sigma(t)\cdot \dot p(t)](\O)-\int_{\partial \O}(\sigma(t)\cdot\nu) \dot u(t)\dd\HH^{n-1}.$$
On the other hand, integrating \eqref{final_modica} between two arbitray times $t_1$ and $t_2$, and using \eqref{eq:abs_cont} yields
\begin{multline*}
\frac{1}{2}\|\dot u(t_2)\|_2^2 + \frac12 \|\sigma(t_2)\|_2^2+\frac{\lambda}{2} \int_{t_1}^{t_2}\int_{\partial \O} |\sigma\cdot\nu|^2\dd\HH^{n-1}\ds\\
 +  \int_{t_1}^{t_2} \int_{\partial \Omega}\psi_\lambda(\dot  u)\dd \mathcal{H}^{n-1}\ds+ \int_{t_1}^{t_2} |\dot p(s)|(\O)\ds\le \frac{1}{2}\|\dot u(t_1)\|_2^2 + \frac12 \|\sigma(t_1)\|_2^2 + \int_{t_1}^{t_2} \int_\Omega f \dot u\dx\ds,
\end{multline*}
which is the announced energy inequality.
\end{proof}

\subsection{Flow rule and boundary condition}
\label{subsec:4.5}

We now show that the previous energy inequality is actually an equality, and as a byproduct, we obtain the flow rule and the boundary condition.

\begin{proposition}
For every $0\le t_1 \le t_2 \le T$,
\begin{multline}\label{nrj}
\frac{1}{2}\|\dot u(t_2)\|_2^2 + \frac12 \|\sigma(t_2)\|_2^2+\frac{\lambda}{2} \int_{t_1}^{t_2}\int_{\partial \O} |\sigma\cdot\nu|^2\dd\HH^{n-1}\ds
 +  \int_{t_1}^{t_2} \int_{\partial \Omega}\psi_\lambda(\dot  u)\dd \mathcal{H}^{n-1}\ds+ \int_{t_1}^{t_2} |\dot p(s)|(\O)\ds\\
=\frac{1}{2}\|\dot u(t_1)\|_2^2 + \frac12 \|\sigma(t_1)\|_2^2 + \int_{t_1}^{t_2} \int_\Omega f \dot u\dx\ds.
\end{multline}
In addition, 
$$\sigma\cdot \nu+\psi'_\lambda(\dot u)=0 \quad \text{ in }L^2(0,T;L^2(\partial \O)),$$
and, for a.e. $t \in [0,T]$,
$$|\dot p(t)|=[\sigma(t)\cdot \dot p(t)] \quad \text{ in }\mathcal M(\O).$$
\end{proposition}

\begin{proof}
Deriving the additive decomposition with respect to time yields, for a.e. $t \in [0,T]$,
$$D\dot u(t)=\dot \sigma(t)+\dot p(t) \quad \text{ in }\mathcal M(\O;\R^n),$$
where $\dot u(t) \in BV(\O) \cap L^2(\O)$, $\dot \sigma(t) \in L^2(\Omega;\R^n)$ and $\dot p(t) \in \mathcal M(\Omega;\R^n)$. Moreover, since for all $t \in [0,T]$, $\sigma(t) \in H(\Div,\O)$ and $\|\sigma(t)\|_\infty \leq 1$, we get from Remark \ref{rmk:ineq1} that for a.e. $t \in [0,T]$,
\begin{equation}\label{eq:fr0}
|\dot p(t)| \geq  [\sigma(t)\cdot \dot p(t)] \quad \text{ in }\mathcal M(\Omega),
\end{equation}
and in particular
\begin{equation}\label{eq:fr1}
|\dot p(t)|(\O) \geq [\sigma(t)\cdot\dot p(t)](\O).
\end{equation}

On the other hand, since $\sigma(t) \in H(\Div,\O) \cap L^\infty(\O;\R^n)$ with $\|\sigma(t)\|_\infty\leq 1$, we infer that $\sigma(t)\cdot \nu\in L^\infty(\partial \O)$ with $\|\sigma(t)\cdot\nu\|_{L^\infty(\partial \O)}\leq 1$. As a consequence, for all $t \in [0,T]$
\begin{equation}\label{eq:fr4}
\psi_\lambda(\dot u(t)) + \frac{\lambda}{2}|\sigma(t)\cdot \nu|^2 \geq -(\sigma(t)\cdot \nu) \dot u(t)  \quad \HH^{n-1}\text{-a.e. on }\partial \O,
\end{equation}
or still
\begin{equation}\label{eq:fr2}
\int_{\partial \O} \psi_\lambda(\dot u(t))\dd\HH^{n-1} + \frac{\lambda}{2}\int_{\partial \O}|\sigma(t)\cdot \nu|^2\dd\HH^{n-1} \geq -\int_{\partial \O} (\sigma(t)\cdot \nu) \dot u(t)\dd\HH^{n-1}.
\end{equation}
Summing up \eqref{eq:fr1} and \eqref{eq:fr2}, using Definition \ref{def_Sigma_p} of duality together with the integration by parts formula given by Proposition \ref{IPP_sig_u_varphi} and the equation of motion yields, for a.e. $t \in [0,T]$,
\begin{multline*}
|\dot p(t)|(\O)+\int_{\partial \O} \psi_\lambda(\dot u(t))\dd\HH^{n-1} + \frac{\lambda}{2}\int_{\partial \O}|\sigma(t)\cdot \nu|^2\dd\HH^{n-1} \\
\geq \int_\O f(t)\dot u(t)\dx-\int_{\Omega} \ddot u(t) \dot u(t) \dx - \int_{\Omega} \dot \sigma(t)\cdot \sigma(t)\dx.
\end{multline*}
Thanks to \eqref{eq:abs_cont} together with an integration between two arbitrary times $0 \leq t_1 \leq t_2 \leq T$ gives the converse energy inequality
\begin{multline*}
\frac{1}{2}\|\dot u(t_2)\|_2^2 + \frac12 \|\sigma(t_2)\|_2^2+\frac{\lambda}{2} \int_{t_1}^{t_2}\int_{\partial \O} |\sigma\cdot\nu|^2\dd\HH^{n-1}\ds +  \int_{t_1}^{t_2} \int_{\partial \Omega}\psi_\lambda(\dot  u)\dd \mathcal{H}^{n-1}\ds+ \int_{t_1}^{t_2} |\dot p(s)|(\O)\ds\\
\geq \frac{1}{2}\|\dot u(t_1)\|_2^2 + \frac12 \|\sigma(t_1)\|_2^2 + \int_{t_1}^{t_2} \int_\Omega f \dot u\dx\ds
\end{multline*}
which gives the energy equality \eqref{nrj} according to Proposition \ref{prop:nrjineq}.

We now derive the energy balance \eqref{nrj} with respect to time. It follows that for a.e. $t\in [0,T]$,
\begin{multline*}
|\dot p(t)|(\O)+\int_{\partial \O} \psi_\lambda(\dot u(t))\dd\HH^{n-1} + \frac{\lambda}{2}\int_{\partial \O}|\sigma(t)\cdot \nu|^2\dd\HH^{n-1} \\
= \int_\O f(t)\dot u(t)\dx-\int_{\Omega} \ddot u(t) \dot u(t) \dx - \int_{\Omega} \dot \sigma(t)\cdot \sigma(t)\dx,
\end{multline*}
or still, thanks to the equation of motion, Definition \ref{def_Sigma_p} of duality and the integration by parts formula given by Proposition \ref{IPP_sig_u_varphi},
\begin{multline*}
|\dot p(t)|(\O)+\int_{\partial \O} \psi_\lambda(\dot u(t))\dd\HH^{n-1} + \frac{\lambda}{2}\int_{\partial \O}|\sigma(t)\cdot \nu|^2\dd\HH^{n-1} \\
= [\sigma(t)\cdot\dot p(t)](\O) -\int_{\partial \O}(\sigma(t)\cdot \nu) \dot u(t)  \dd\HH^{n-1}.
\end{multline*}
Remembering \eqref{eq:fr1} and \eqref{eq:fr2} implies that
$$|\dot p(t)|(\O) = [\sigma(t)\cdot\dot p(t)](\O),$$
and
$$\int_{\partial \O} \psi_\lambda(\dot u(t))\dd\HH^{n-1} + \frac{\lambda}{2}\int_{\partial \O}|\sigma(t)\cdot \nu|^2\dd\HH^{n-1} = -\int_{\partial \O} (\sigma(t)\cdot \nu) \dot u(t)\dd\HH^{n-1},$$
and by \eqref{eq:fr0} and \eqref{eq:fr4}, 
$$|\dot p(t)| =  [\sigma(t)\cdot \dot p(t)] \quad \text{ in }\mathcal M(\Omega),$$
as well as
$$\psi_\lambda(\dot u(t)) + \frac{\lambda}{2}|\sigma(t)\cdot \nu|^2 = - (\sigma(t)\cdot \nu) \dot u(t) \quad \HH^{n-1}\text{-a.e. on }\partial \O.$$
By \eqref{psi*} and standard arguments of convex analysis, this last formula is then equivalent to the boundary condition
$\sigma(t)\cdot \nu+\psi'_\lambda(\dot u(t))=0$ $\HH^{n-1}$-a.e. on $\partial \O$.
\end{proof}

\begin{remark}
According to the boundary condition at fixed $\e>0$ and the convergence \eqref{weak_cv_sigma_eps_nu_L2_t_L2_dOmega} of the normal stress, we have that, as $\e \to 0$, $\dot u_\e=\lambda g_\e-\lambda (\sigma_\e +\e\nabla \dot u_\e)\cdot \nu \wto - \lambda \sigma\cdot \nu=\lambda \psi'_\lambda(\dot u)$ 
weakly in $L^2(0,T;L^2(\partial \O))$. It enables one to identify the function $w$ (the limit of the trace of $u_\e$ in \eqref{eq:conv_0}) as $\dot w=\lambda \psi'_\lambda(\dot u)$, and by Proposition \ref{prop:strong_conv},
\begin{equation}\label{eq:strong_conv5}
\dot u_\e \to  \lambda \psi'_\lambda(\dot u) \quad\text{strongly in } L^2(0,T;L^2(\partial \Omega)).
\end{equation}
\end{remark}

\subsection{Uniqueness}
\label{subsec:4.8}

In order to establish the uniqueness, we derive a general comparison principle between two solutions which, in the context of hyperbolic equations, is known as a {\it Kato inequality}. This result is much more than needed in order to establish uniqueness. However, it will be useful later to show regularity results for the variational solution of the elasto-plastic problem (see Section~\ref{sec:5}).

\begin{proposition}[Kato inequality]
\label{prop_kato_ineq}
Let $(u,\sigma,p)$ (resp. $(\tilde{u},\tilde{\sigma},\tilde{p})$) be a variational solution of the elasto-plastic problem associated to the initial condition $(u_0,v_0,\sigma_0,p_0)$ (resp. $(\tilde u_0,\tilde v_0 ,\tilde \sigma_0 ,\tilde p_0)$) and the source term $f$ (resp. $\tilde{f}$). Then for all $\varphi \in W^{1,\infty}(\O \times (0,T))$ with $\varphi \geq 0$,
\begin{multline}
\label{kato_ineq_eps_zero}
\int_0^T \int_{\Omega} (\dot u -\dot {\tilde u})^2\dot \varphi \dxdt + \int_0^T \int_{\Omega} |\sigma -\tilde{\sigma}|^2\dot \varphi \dxdt +\int_{\Omega} (v_0 -\tilde v_0 )^2 \varphi(0) \dx + \int_{\Omega} |\sigma_0 -\tilde \sigma_0 |^2\varphi(0) \dx \\
	 - 2\int_0^T \int_{\Omega} (\sigma -\tilde{\sigma}) \cdot \nabla \varphi(\dot u -\dot{ \tilde u})\dxdt  + 2 \int_0^T \int_{\Omega} (f-\tilde{f} )(\dot u -\dot{ \tilde u})\varphi \dxdt\\
	  \geq 2\lambda\int_0^T\int_{\partial \O} \big(\psi'_\lambda(\dot u)-\psi'_\lambda(\dot{\tilde u}) \big)^2\varphi\dd\HH^{n-1}\dt.
\end{multline}
\end{proposition}

\begin{proof}
For fixed $\e >0$, let $(u_\e,\sigma_\e,p_\e)$ (resp. $(\tilde u_\e,\tilde\sigma_\e,\tilde p_\e)$) be the solution of the elasto-visco-plastic problem given by Theorem \ref{final_theorem_at_level_epsilon} for the initial condition $(u_{0},v_{0},\sigma_0,p_{0})$ (resp. $(\tilde u_{0},\tilde v_{0},\tilde \sigma_0,\tilde p_{0})$) and the source terms $f$ and $g_\e:=\e\nabla v_0\cdot\nu$ (resp. $\tilde f$ and $\tilde g_\e:=\e\nabla \tilde v_0\cdot\nu$). According to Remark \ref{rem:wf} with the test function $(\dot u_\e - \dot{\tilde u}_\e) \varphi \in L^2(0,T;H^1(\O))$, we infer that
\begin{multline*}
\int_0^T \int_\Omega \ddot u_\e \left(\dot u_\e - \dot{\tilde u}_\e \right) \varphi\dxdt + \int_0^T \int_\Omega (\sigma_\e+\e \nabla\dot u_\e) \cdot \nabla\big( (\dot u_\e - \dot{ \tilde u}_\e) \varphi  \big)\dxdt \\
+ \frac{1}{\lambda} \int_0^T \int_{\partial \Omega} \dot u_\e \left(\dot u_\e - \dot{ \tilde u}_\e \right) \varphi\dd \mathcal{H}^{n-1} \dt\\
 =  \int_0^T \int_\Omega f\left(\dot u_\e - \dot{ \tilde u}_\e \right) \varphi\dxdt+\int_0^T \int_{\partial \O} g_\e \left(\dot u_\e - \dot{ \tilde u}_\e \right) \varphi\dd\HH^{n-1}\dt,
\end{multline*}
and similarly,
\begin{multline*}
\int_0^T \int_\Omega \ddot{\tilde u}_\e \left(\dot {\tilde u}_\e - \dot{u}_\e \right) \varphi\dxdt + \int_0^T \int_\Omega ({\tilde \sigma}_\e+\e \nabla\dot{\tilde  u}_\e) \cdot \nabla\big( (\dot{\tilde u}_\e - \dot{u}_\e) \varphi \big)\dxdt \\
+ \frac{1}{\lambda} \int_0^T \int_{\partial \Omega} \dot {\tilde u}_\e (\dot{\tilde u}_\e - \dot{u}_\e ) \varphi\dd \mathcal{H}^{n-1} \dt\\
 =  \int_0^T \int_\Omega {\tilde f}(\dot {\tilde u}_\e - \dot{u}_\e ) \varphi\dxdt+\int_0^T \int_{\partial \O} {\tilde g}_\e(\dot {\tilde u}_\e - \dot{u}_\e ) \varphi\dd\HH^{n-1}\dt.
\end{multline*}
Summing up both previous equalities leads to 
\begin{multline}\label{eq:ueps1}
\int_0^T \int_\Omega (\ddot u_\e-\ddot{\tilde u}_\e) (\dot u_\e - \dot{\tilde u}_\e) \varphi\dxdt 
+ \int_0^T \int_\Omega \big(\sigma_\e-\tilde\sigma_\e+\e (\nabla\dot u_\e-\nabla \dot{ \tilde u}_\e)\big) \cdot \nabla\big( (\dot u_\e - \dot{ \tilde u}_\e) \varphi\big)\dxdt \\
+ \frac{1}{\lambda}\int_0^T \int_{\partial \Omega} (\dot u_\e-\dot {\tilde u}_\e)^2 \varphi\dd \mathcal{H}^{n-1} \dt\\
   =  \int_0^T \int_\Omega (f-{\tilde f})(\dot u_\e - \dot{ \tilde u}_\e ) \varphi\dxdt+\int_0^T \int_{\partial \O} (g_\e-\tilde g_\e) \left(\dot u_\e - \dot{ \tilde u}_\e \right) \varphi\dd\HH^{n-1}\dt.
\end{multline}
Using the additive decompositon \eqref{decomposition_nablau_eps}, we get that
\begin{multline}\label{eq:ueps2}
 \int_0^T \int_\Omega \big(\sigma_\e-\tilde\sigma_\e+\e (\nabla\dot u_\e-\nabla \dot{ \tilde u}_\e)\big) \cdot \nabla\big( (\dot u_\e - \dot{ \tilde u}_\e) \varphi\big)\dxdt \\
=  \int_0^T \int_\Omega (\sigma_\e-\tilde \sigma_\e) \cdot (\dot \sigma_\e-\dot{\tilde \sigma}_\e)\varphi\dxdt  + \int_0^T \int_\Omega (\sigma_\e-\tilde\sigma_\e) \cdot  (\dot p_\e-\dot{\tilde p}_\e)\varphi\dxdt 
 + \e \int_0^T \int_\Omega \left|\nabla\dot u_\e-\nabla \dot {\tilde u}_\e \right|^2 \varphi\dxdt \\
 + \int_0^T  \int_\Omega (\sigma_\e-\tilde\sigma_\e)\cdot \nabla \varphi\left(\dot u_\e - \dot{ \tilde u}_\e \right)\dxdt + \e \int_0^T  \int_\Omega (\nabla\dot u_\e-\nabla \dot{ \tilde u}_\e) \cdot \nabla \varphi (\dot u_\e-\dot{ \tilde u}_\e)\dxdt . 
\end{multline}
Observe that, thanks to the flow rule \eqref{snd_eq_continuous} the second term of the right hand side is non-negative. Consequently, gathering \eqref{eq:ueps1} and \eqref{eq:ueps2}, we obtain that
\begin{multline}\label{eq:ueps3}
\frac{1}{\lambda} \int_0^T \int_{\partial \Omega} (\dot u_\e-\dot {\tilde u}_\e)^2 \varphi\dd \mathcal{H}^{n-1} \dt\\
\le -\int_0^T \int_\Omega (\ddot u_\e-\ddot{\tilde u}_\e) (\dot u_\e - \dot{\tilde u}_\e ) \varphi\dxdt
 -  \int_0^T \int_\Omega (\sigma_\e-\tilde \sigma_\e) \cdot (\dot \sigma_\e-\dot{\tilde \sigma}_\e)\varphi\dxdt \\
 - \int_0^T  \int_\Omega (\sigma_\e-\tilde\sigma_\e)\cdot \nabla \varphi\left(\dot u_\e - \dot{ \tilde u}_\e \right)\dxdt+ \int_0^T \int_\Omega (f-{\tilde f})(\dot u_\e - \dot{ \tilde u}_\e ) \varphi\dxdt\\
+\int_0^T \int_{\partial \O} (g_\e-\tilde g_\e) \left(\dot u_\e - \dot{ \tilde u}_\e \right) \varphi\dd\HH^{n-1}\dt- \e \int_0^T  \int_\Omega (\nabla\dot u_\e-\nabla \dot{ \tilde u}_\e) \cdot \nabla \varphi (\dot u_\e-\dot{ \tilde u}_\e)\dxdt.
\end{multline}
Using next that both functions $\dot u_\e - \dot{\tilde u}_\e\in W^{1,\infty}([0,T];L^2(\O))$ and $\sigma_\e -\tilde \sigma_\e\in W^{1,\infty}([0,T];L^2(\O;\R^n))$, we can integrate by parts with respect to the time variable to get that
\begin{multline*}
- \int_0^T \int_\Omega (\ddot u_\e-\ddot{\tilde u}_\e) (\dot u_\e - \dot{\tilde u}_\e ) \varphi\dxdt -  \int_0^T \int_\Omega (\sigma_\e-\tilde \sigma_\e) \cdot (\dot \sigma_\e-\dot{\tilde \sigma}_\e)\varphi\dxdt  \\
= \frac{1}{2}\int_0^T \int_\Omega (\dot u_\e - \dot{ \tilde u}_\e )^2 \dot \varphi\dxdt + \frac{1}{2} \int_0^T \int_\Omega |\sigma_\e-\tilde \sigma_\e|^2 \dot \varphi\dxdt \\
 + \frac{1}{2}\int_\Omega (v_{0} - \tilde v_{0} )^2 \varphi(0)\dx -\frac{1}{2}\int_\Omega (\dot u_\e(T) - \dot{ \tilde u}_\e(T) )^2 \varphi(T)\dx\\
 + \frac{1}{2}\int_\Omega |\sigma_0-\tilde \sigma_0 |^2 \varphi(0)\dx - \frac{1}{2}\int_\Omega |\sigma_\e(T)-\tilde \sigma_\e(T)|^2 \varphi(T)\dx,
\end{multline*}
and inserting inside \eqref{eq:ueps3} leads to
\begin{multline*}
\frac{1}{\lambda} \int_0^T \int_{\partial \Omega} (\dot u_\e-\dot {\tilde u}_\e)^2 \varphi\dd \mathcal{H}^{n-1} \dt\le  \frac{1}{2}\int_0^T \int_\Omega (\dot u_\e - \dot{ \tilde u}_\e )^2 \dot \varphi\dxdt + \frac{1}{2} \int_0^T \int_\Omega |\sigma_\e-\tilde \sigma_\e|^2\dot \varphi\dxdt \\
 + \frac{1}{2}\int_\Omega (v_{0} - \tilde v_{0} )^2 \varphi(0)\dx + \frac{1}{2}\int_\Omega |\sigma_0-\tilde \sigma_0 |^2 \varphi(0)\dx  - \int_0^T  \int_\Omega (\sigma_\e-\tilde\sigma_\e)\cdot \nabla \varphi\left(\dot u_\e - \dot{ \tilde u}_\e \right)\dxdt\\
 + \int_0^T \int_\Omega (f-{\tilde f})(\dot u_\e - \dot{ \tilde u}_\e ) \varphi\dxdt +\int_0^T \int_{\partial \O} (g_\e-\tilde g_\e) \left(\dot u_\e - \dot{ \tilde u}_\e \right) \varphi\dd\HH^{n-1}\dt\\
 - \e \int_0^T  \int_\Omega (\nabla\dot u_\e-\nabla \dot{ \tilde u}_\e) \cdot \nabla \varphi (\dot u_\e-\dot{ \tilde u}_\e)\dxdt.
\end{multline*}
Using  Proposition \ref{prop:strong_conv} and \eqref{eq:strong_conv5}, and letting $ \e \to 0$ in the previous inequality, we get \eqref{kato_ineq_eps_zero}.
\end{proof}

As a consequence of Proposition \ref{prop_kato_ineq}, if $(u_1,\sigma_1,p_1)$ and $(u_2,\sigma_2,p_2)$ are two variational solutions of the elasto-plastic problem given by Theorem~\ref{final_theorem_eps_eq_zero} for the same initial data $(u_0,v_0,e_0,p_0)$ and source term $f$, taking $\varphi(x,t):=T-t$ as test function in \eqref{kato_ineq_eps_zero} implies that $\sigma_1=\sigma_2$ and $\dot u_1=\dot u_2$. Consequently, since $u_1(0)=u_2(0)=u_0$, we deduce that $u_1=u_2$, and using the additive decomposition, that $p_1=p_2$.

\begin{remark}
The uniqueness of the solution shows that there is no need to extract subsequences in all weak and strong convergences obtained before.
\end{remark}

\subsection{Homogeneous Dirichlet and Neumann boundary conditions}\label{sec:lambda}

The boundary condition studied so far does not account for the important Dirichlet and Neumann cases. It is however possible to recover these particular situations by means of asymptotic analysis as the coefficient $\lambda \to 0^+$ for the Dirichlet, or $\lambda \to +\infty$ for the Neumann case. The object of this section is to show such rigorous convergence results. 

\begin{theorem}
\label{DN}
Let $\Omega\subset \R^n$ be a bounded open set of class $\mathcal C^1$. Consider a source term $f \in H^1([0,T];L^2(\Omega))$ and an initial data $(u_0,v_0,\sigma_0,p_0) \in H^1(\Omega)  \times H^2(\Omega) \times H(\Div,\O)\times  L^2(\Omega;\R^n) $ such that
\begin{equation}
\label{initial_data_DirichletNeumann}
\begin{cases}
\nabla u_0 = \sigma_0 + p_0 \text{ in }L^2(\O;\R^n),\\
|\sigma_0|\leq 1 \text{ a.e. in }\Omega,
\end{cases}
\end{equation}
and 
\begin{equation}\label{eq:ell}
\sigma_0\cdot \nu = v_0 = 0 \quad \HH^{n-1}\text{-a.e. on } \partial \Omega.
\end{equation}
For any $\lambda>0$, let $(u_\lambda,\sigma_\lambda,p_\lambda)$ the unique variational solution to the elasto-plastic problem given by Theorem \ref{final_theorem_eps_eq_zero}. For $\ell \in \{0,+\infty\}$, there exist a unique triple $(u^{(\ell)},\sigma^{(\ell)},p^{(\ell)})$ with the regularity
\begin{equation}\label{eq:reg}
\begin{cases}
u^{(\ell)}\in W^{2,\infty}([0,T];L^2(\Omega)) \cap \mathcal C^{0,1}([0,T];BV(\Omega)),\\
\sigma^{(\ell)} \in W^{1,\infty}([0,T];L^2(\Omega;\R^n)),\\
p^{(\ell)} \in  \mathcal C^{0,1}([0,T];\mathcal M(\Omega;\R^n)),
\end{cases}
\end{equation}
such that, as $\lambda \to \ell$,
\begin{equation}\label{eq:conv_11}
\begin{cases}
u_\lambda \wto u^{(\ell)} &\text{weakly* in } W^{2,\infty}([0,T];L^2(\O)), \\
\sigma_\lambda \wto \sigma^{(\ell)} &\text{weakly* in } W^{1,\infty}([0,T];L^2(\O;\R^n)),\\
p_\lambda(t) \wto p^{(\ell)}(t) &\text{weakly* in }\mathcal M(\O;\R^n) \text{ for all }t \in [0,T],
\end{cases}
\end{equation}
and which satisfies the following properties:

\begin{enumerate}
\item The initial conditions:
\[
u^{(\ell)}(0)   =  u_0, \quad \dot  u^{(\ell)}(0) = v_0, \quad 
\sigma^{(\ell)}(0)   =  \sigma_0, \quad p^{(\ell)}(0)   =  p_0;
\]
\item The additive decomposition: for all $t \in [0,T]$
\[
D u^{(\ell)}(t) = \sigma^{(\ell)}(t) + p^{(\ell)}(t)\quad \text{in }\mathcal M(\O;\R^n);
\]
\item The equation of motion: 
$$\ddot u^{(\ell)} - \Div\sigma^{(\ell)} = f \quad  \text{ in } L^2(0,T;L^2(\Omega));$$
\item The boundary condition:
$$
\sigma^{(\infty)}\cdot \nu=0 \quad \text{ in } L^2(0,T;L^2(\partial \Omega));$$
\item The stress constraint: for all $t \in [0,T]$,
$$|\sigma^{(\ell)}(t)|\leq 1 \quad \text{ a.e. in }\O;$$
\item The flow rule: for a.e. $t \in [0,T]$,
$$
\begin{cases}
|\dot p^{(0)}(t)|=[\sigma^{(0)}(t) \cdot \dot p^{(0)}(t)] & \text{in }\mathcal M(\O),\\
|\dot u^{(0)}(t)|=-(\sigma^{(0)}(t)\cdot \nu) \dot u^{(0)}(t) & \HH^{n-1}\text{-a.e. on }\partial \O,
\end{cases}$$
or
$$|\dot p^{(\infty)}(t)|=[\sigma^{(\infty)}(t) \cdot \dot p^{(\infty)}(t)] \quad \text{ in }\mathcal M(\O).$$
\end{enumerate}
\end{theorem}

\begin{remark}
Note that in the Dirichlet case ($\ell=0$), as classical in variational problems with linear growth, the velocity may concentrate on the boundary so that its inner trace might not vanish as required by the boundary condition. It explains why some plastic strain can accumulate on the boundary and that the flow rule is formulated in $\overline \O$ and not only in $\O$.
\end{remark}

\begin{remark}
The assumption \eqref{eq:ell} seems to be artificial, however it allows us to easily satisfy the boundary condition $\sigma_0\cdot \nu+\lambda^{-1}v_0=0$ on $\partial \O$ for the initial data for every $\lambda>0$. It seems more natural to consider initial data that satisfy~\eqref{initial_data_DirichletNeumann} and, for example in the Neumann case, only
\begin{equation}
\label{BC_init_Neumann}
\sigma_0\cdot \nu = 0\quad\HH^{n-1}\text{-a.e. on }\partial \Omega.\\
\end{equation}
To do so, one should be able to construct for every $(u_0,v_0,\sigma_0,p_0)\in H^1(\O) \times H^2(\Omega) \times H(\Div,\O) \times  L^2(\Omega;\R^n)$, satisfying~\eqref{initial_data_DirichletNeumann} and~\eqref{BC_init_Neumann},
a sequence $(u_0^\lambda,v_0^\lambda,\sigma_0^\lambda,p_0^\lambda)\in H^1(\O) \times H^2(\Omega) \times H(\Div,\O) \times  L^2(\Omega;\R^n)$ such that
\[
\begin{cases}
\nabla u_0^\lambda = \sigma_0^\lambda + p_0^\lambda \text{ in }L^2(\O;\R^n),\\
\sigma_0^\lambda\cdot \nu + \lambda^{-1} v_0^\lambda = 0\quad\HH^{n-1}\text{-a.e. on }\partial \Omega,\\
|\sigma_0^\lambda|\leq 1 \text{ a.e. in }\Omega,
\end{cases}
\]
and, at least, 
\[
\begin{cases}
u_0^\lambda \wto u_0 &\text{weakly in } L^2(\O), \\
v_0^\lambda \wto v_0 &\text{weakly in } H^1(\O), \\
\sigma_0^\lambda \wto \sigma_0 &\text{weakly in } H(\Div,\O),\\
p_0^\lambda \wto p_0 &\text{weakly* in }\mathcal M(\O;\R^n).
\end{cases}
\]
This issue will not be addressed in this paper and for that reason, we assume \eqref{eq:ell} for simplicity.
\end{remark}

\begin{proof} The proof follows closely the lines of that of Theorem \ref{final_theorem_eps_eq_zero}. It is divided into three steps.

\medskip

\noindent {\bf Step 1: Weak convergences.} Passing to the limit as $\e \to 0$ in estimates \eqref{estimate_snd_derivative_epsilon} and \eqref{estimate_first_derivative_epsilon} yields
$$\sup_{t\in [0,T]} \| \ddot u_\lambda(t) \|^2_2 + \sup_{t\in [0,T]} \| \dot \sigma_\lambda(t) \|^2_2
\le  C\bigg(\|\Div\sigma_0+f(0)\|_2^2+  \| \nabla v_0 \|^2_2 + \Big(\int_0^T \| \dot f(t) \|_2\dt\Big)^2 \bigg),$$
and
\begin{multline}\label{eq:estim:69}
\sup_{t\in [0,T]} \|\dot u_\lambda(t) \|^2_2 + \sup_{t\in [0,T]} \|\sigma_\lambda (t)\|_2^2+ \int_0^T \psi_\lambda(\dot u_\lambda(t))\dt +\frac{\lambda}{2} \int_0^T\int_{\partial \O} |\sigma_\lambda \cdot\nu|^2\dd\HH^{n-1}\dt\\
\le  C \left( \|v_0\|_2^2 + \|\sigma_0\|^2_2+ \Big(\int_0^T\| f (t)\|_2\dt\Big)^2 \right),
\end{multline}
where the constants $C>0$ are independent of $\lambda$.   Using similar arguments than in Subsection \ref{subsec:4.1}, we can find a subsequence (not relabeled) and, for $\ell\in \{0,+\infty\}$, functions $u^{(\ell)}$, $\sigma^{(\ell)}$ and $p^{(\ell)}$ as in \eqref{eq:reg} such that the weak convergences \eqref{eq:conv_11} hold. In particular, we can easily derive the initial conditions, the additive decomposition, the stress constraint, and the equation of motion. 

Using the equation of motion for fixed $\lambda>0$, we infer that the sequence $(\sigma_\lambda)_{\lambda>0}$ is bounded in $L^2(0,T;H(\Div,\O))$, so that $\sigma_\lambda\cdot\nu \wto \sigma^{(\ell)}\cdot \nu$ weakly in $L^2(0,T;H^{-1/2}(\partial \O))$. On the other hand, according to estimate \eqref{eq:estim:69}, we have that $\sigma_\lambda\cdot \nu \to 0$ strongly in $L^2(0,T;L^2(\partial \O))$ as $\lambda \to +\infty$. Therefore, the Neumann boundary condition $\sigma^{(\infty)}\cdot \nu=0$ in $L^2(0,T;L^2(\partial \O))$ follows.

\medskip

\noindent {\bf Step 2: Strong convergences.} We now show that, as $\lambda \to \ell$,
\begin{equation}\label{eq:sc}
\begin{cases}
\dot u_\lambda \to\dot u^{(\ell)} &\text{strongly in } \mathcal C^0([0,T];L^2(\O)), \\
\sigma_\lambda \to \sigma^{(\ell)} &\text{strongly in } \mathcal C^0([0,T];L^2(\O;\R^n)).
\end{cases}
\end{equation}
To this aim, substracting the equations of motion, we get that 
$$\ddot u_\lambda -\ddot u^{(\ell)}-\Div(\sigma_\lambda-\sigma^{(\ell)})=0\quad \text{ in }L^2(0,T;L^2(\O)).$$
For $t \in [0,T]$, multiplying by $\mathds{1}_{[0,t]} \dot u_\lambda$, integrating over $\O \times (0,T)$, and using Definition \ref{def_Sigma_p} of duality together with the integration by parts formula given by Proposition \ref{IPP_sig_u_varphi} yields
\begin{multline*}
\int_0^t \int_\Omega (\ddot u_\lambda - \ddot u^{(\ell)}) \dot u_\lambda\dx\ds
 + \int_0^t\int_\Omega (\sigma_\lambda-\sigma^{(\ell)}) \cdot \dot \sigma_\lambda \dx\ds\\
 +\int_0^t [(\sigma_\lambda(s)-\sigma^{(\ell)}(s)) \cdot \dot p_\lambda(s)](\O)\ds -\int_0^t \int_{\partial \Omega} (\sigma_\lambda \cdot \nu-\sigma^{(\ell)}\cdot\nu)\dot u_\lambda\dd \mathcal{H}^{n-1} \ds=0.
 \end{multline*}
According to the flow rule for fixed $\lambda>0$, the fact that  for all $s \in [0,T]$, $\|\sigma^{(\ell)}(s)\|_\infty\leq 1$ and Remark \ref{rmk:ineq1}, we get that
$[(\sigma_\lambda(s)-\sigma^{(\ell)}(s)) \cdot \dot p_\lambda(s)](\O)\geq 0$ for a.e. $s \in [0,T]$. Thus,
\begin{multline*}
\int_0^t \int_\Omega (\ddot u_\lambda - \ddot u^{(\ell)}) (\dot u_\lambda-\dot u^{(\ell)})\dx\ds  + \int_0^t\int_\Omega (\sigma_\lambda-\sigma^{(\ell)}) \cdot ( \dot \sigma_\lambda-\dot \sigma^{(\ell)})\dx\ds\\ 
\le  - \int_0^t \int_\Omega (\ddot u_\lambda - \ddot u^{(\ell)}) \dot u^{(\ell)}\dx\ds  - \int_0^t \int_\Omega (\sigma_\lambda-\sigma^{(\ell)}) \cdot \dot \sigma^{(\ell)}\dx\ds+ \int_0^t \int_{\partial \Omega} (\sigma_\lambda \cdot \nu-\sigma^{(\ell)}\cdot\nu)\dot u_\lambda\dd \mathcal{H}^{n-1} \ds,
\end{multline*}
and integrating by parts the left hand side with respect to time yields
\begin{multline}\label{eq:rhs}
 \frac{1}{2} \|\dot u_\lambda(t) - \dot u^{(\ell)}(t)\|^2_2+ \frac{1}{2}  \|\sigma_\lambda(t) - \sigma^{(\ell)}(t)\|_2^2\\
 \le  - \int_0^t \int_\Omega (\ddot u_\lambda - \ddot u^{(\ell)}) \dot u^{(\ell)}\dx\ds  - \int_0^t \int_\Omega (\sigma_\lambda-\sigma^{(\ell)}) \cdot \dot \sigma^{(\ell)}\dx\ds  \\
+ \int_0^t \int_{\partial \Omega} (\sigma_\lambda \cdot \nu-\sigma^{(\ell)}\cdot\nu)\dot u_\lambda\dd \mathcal{H}^{n-1} \ds.
\end{multline}
The weak convergences \eqref{eq:conv_11} ensure that the first two integrals in the right hand side of \eqref{eq:rhs} tend to $0$ as $\lambda \to \ell$. Concerning the boundary integral, if $\ell=+\infty$, since $\sigma^{(\infty)}\cdot\nu=0$,  using the boundary condition for fixed $\lambda>0$, we infer that 
$$ (\sigma_\lambda \cdot \nu-\sigma^{(\infty)}\cdot\nu)\dot u_\lambda=-\psi'_\lambda(\dot u_\lambda) \dot u_\lambda \leq 0 \quad \text{ a.e. on }\partial \O \times (0,T)$$ 
by convexity of $\psi_\lambda$. On the other hand, if $\ell=0$, from the boundary condition for fixed $\lambda>0$ and the fact that $\|\sigma^{(0)}\cdot\nu\|_{L^\infty(\partial \O \times (0,T))}\leq 1$, we get that
\begin{multline*}
(\sigma_\lambda \cdot \nu-\sigma^{(0)}\cdot\nu)\dot u_\lambda = -\left( |\dot u_\lambda| \mathds{1}_{\{ |\dot u_\lambda|>\lambda\}} + \frac{|\dot u_\lambda|^2}{\lambda}\mathds{1}_{\{ |\dot u_\lambda|\leq\lambda\}}+(\sigma^{(0)}\cdot \nu) \dot u_\lambda \right) \\
\leq -\left( \frac{|\dot u_\lambda|^2}{\lambda}+(\sigma^{(0)}\cdot \nu) \dot u_\lambda\right)\mathds{1}_{\{|\dot u_\lambda|\leq\lambda\}}  \leq \lambda \quad \text{ a.e. on }\partial \O \times (0,T).
\end{multline*}
In both cases, we obtain that the right hand side of \eqref{eq:rhs} is infinitesimal as $\lambda \to \ell$, which completes the proof of the strong convergences.

\medskip

\noindent {\bf Step 3: The flow rules.}
If $\ell=+\infty$, writting the energy balance \eqref{eq:energy-balance2} between two arbitrary times $0 \leq t_1<t_2 \leq T$ yields
$$\frac12 \|\dot u_\lambda(t_2)\|_2^2+\frac12 \|\sigma_\lambda(t_2)\|_2^2 +\int_{t_1}^{t_2} |\dot p_\lambda(t)|(\Omega)\dt
\leq \frac12 \|\dot u_\lambda(t_1)\|_2^2+\frac12 \|\sigma_\lambda(t_1)\|_2^2+\int_{t_1}^{t_2} \int_{\Omega} f \dot u_\lambda \dxdt.$$
Using the strong convergences \eqref{eq:sc} together with the sequential lower semicontinuity of the mapping
$$p \mapsto \int_{t_1}^{t_2} |\dot p(t)|(\Omega)\dt$$
with respect to the weak convergence \eqref{eq:conv_11} (see {\it e.g.} \cite[Appendix]{DalMasoDeSimoneMora}) leads to
\begin{multline*}
\frac12 \|\dot u^{(\infty)}(t_2)\|_2^2+\frac12 \|\sigma^{(\infty)}(t_2)\|_2^2 +\int_{t_1}^{t_2} |\dot p^{(\infty)}(t)|(\Omega)\dt\\
\leq \frac12 \|\dot u^{(\infty)}(t_1)\|_2^2+\frac12 \|\sigma^{(\infty)}(t_1)\|_2^2+\int_{t_1}^{t_2} \int_{\Omega} f \dot u^{(\infty)} \dxdt.
\end{multline*}
Deriving this inequality with respect to time, and using the equation of motion then implies that for a.e. $t \in [0,T]$, 
$|\dot p^{(\infty)}(t)|(\O) \leq [\sigma^{(\infty)}(t)\cdot \dot p^{(\infty)}(t)](\O)$.
On the other hand, since we have $\|\sigma^{(\infty)}\|_{L^\infty(\O \times (0,T))}\leq 1$, Remark \ref{rmk:ineq1} ensures that $|\dot p^{(\infty)}(t)|\geq [\sigma^{(\infty)}(t) \cdot \dot p^{(\infty)}(t)]$ in $\mathcal M(\O)$, 
from which we deduce that
$$|\dot p^{(\infty)}(t)|= [\sigma^{(\infty}(t) \cdot \dot p^{(\infty)}(t)] \quad \text{ in }\mathcal M(\O) \text{ for a.e. }t \in [0,T].$$

\medskip

Next, if $\ell=0$, writting the energy balance \eqref{eq:energy-balance2} between two arbitrary times $0 \leq t_1<t_2 \leq T$ yields
\begin{multline*}
\frac12 \|\dot u_\lambda(t_2)\|_2^2+\frac12 \|\sigma_\lambda(t_2)\|_2^2 +\int_{t_1}^{t_2} |\dot p_\lambda(t)|(\Omega)\dt+\int_{t_1}^{t_2}\int_{\partial \O}\psi_\lambda(\dot u_\lambda)\dd\HH^{n-1}\dt\\
\leq \frac12 \|\dot u_\lambda(t_1)\|_2^2+\frac12 \|\sigma_\lambda(t_1)\|_2^2+\int_{t_1}^{t_2} \int_{\Omega} f \dot u_\lambda \dxdt.
\end{multline*}
Using the definition of $\psi_\lambda$, we have
\begin{multline*}
\int_{t_1}^{t_2}\int_{\partial \O}\psi_\lambda(\dot u_\lambda)\dd\HH^{n-1}\dt  \geq \int_{t_1}^{t_2}  \int_{\partial \O} |\dot u_\lambda| \dd\HH^{n-1}\dt-\int_{t_1}^{t_2} \int_{\{|\dot u_\lambda(t)| \leq \lambda \}}|\dot u_\lambda|\dd\HH^{n-1}\dt\\
  -\frac{\lambda}{2} \int_{t_1}^{t_2}  \HH^{n-1}(\{|\dot u_\lambda(t)|>\lambda\})\dt\geq  \int_{t_1}^{t_2}  \int_{\partial \O} |\dot u_\lambda| \dd\HH^{n-1}\dt -\lambda (t_2-t_1) \HH^{n-1}(\partial \O).
\end{multline*}
Thanks to the strong convergences \eqref{eq:sc} and the sequential lower semicontinuity of the mapping
$$(u,p) \mapsto \int_{t_1}^{t_2} |\dot p(t)|(\Omega)\dt+ \int_{t_1}^{t_2}\int_{\partial \O}|\dot u|\dd\HH^{n-1}\dt$$
with respect to the convergences \eqref{eq:conv_11} and \eqref{eq:sc}  (see {\it e.g.} \cite[Appendix]{DalMasoDeSimoneMora}), we get
\begin{multline*}
\frac12 \|\dot u^{(0)}(t_2)\|_2^2+\frac12 \|\sigma^{(0)}(t_2)\|_2^2 +\int_{t_1}^{t_2} |\dot p^{(0)}(t)|(\Omega)\dt+\int_{t_1}^{t_2}\int_{\partial \O}|\dot u^{(0)}|\dd\HH^{n-1}\dt\\
\leq \frac12 \|\dot u^{(0)}(t_1)\|_2^2+\frac12 \|\sigma^{(0)}(t_1)\|_2^2+\int_{t_1}^{t_2} \int_{\Omega} f \dot u^{(0)} \dxdt,
\end{multline*}
or still, by definition of duality
$$|\dot p^{(0)}(t)|(\O) +\int_{\partial \O} |\dot u^{(0)}(t)|\dd \mathcal{H}^{n-1}
\leq [\sigma^{(0)}(t)\cdot \dot p^{(0)}(t)](\O) -\int_{\partial \O} (\sigma^{(0)}(t)\cdot \nu) \dot u^{(0)}(t)\dd\HH^{n-1}.$$
On the other hand, since $\|\sigma^{(0)}\cdot \nu\|_{L^\infty(\partial \O \times (0,T))}\leq \|\sigma^{(0)}\|_{L^\infty(\O \times (0,T))}\leq 1$, Remark \ref{rmk:ineq1} ensures that for a.e. $t \in [0,T]$, $|\dot p^{(0)}(t)|\geq [\sigma^{(0)}(t) \cdot \dot p^{(0)}(t)]$ in $\mathcal M(\O)$ and
$|\dot u^{(0)}(t)|\geq -(\sigma^{(0)}(t)\cdot \nu)\dot u^{(0)}(t)$ $\HH^{n-1}$-a.e. on $\partial \O$, from which we deduce that
$$|\dot p^{(0)}(t)|= [\sigma^{(0)}(t) \cdot \dot p^{(0)}(t)] \quad \text{ in }\mathcal M(\O)$$
and $$|\dot u^{(0)}(t)|= -(\sigma^{(0)}(t)\cdot \nu)\dot u^{(0)}(t)\quad \HH^{n-1}\text{-a.e. on }\partial \O.$$

Finally, in both cases $\ell \in \{0,+\infty\}$, the uniqueness can be recovered as in Subsection \ref{subsec:4.8}.
\end{proof}

\section{Short time regularity of the solution}
\label{sec:5}

\noindent In this section, we prove that the variational solutions to the elasto-plastic problem are smooth in short time, provided the initial data are smooth and compactly supported in space. This kind of regularity result in the context of dynamical elasto-plasticity seems to be new, and the argument strongly rests on the hyperbolic structure of the model. The general idea is similar to the proof of the fact that the (unique) entropic solution to a scalar conservation law with $BV$ initial data is actually $BV$ (instead of just $L^1$ in the Kru{\v{z}}kov theory \cite{K}). It consists in proving a comparison principle between two solutions associated to different initial data. In \cite{K}, an $L^1$-contraction principle states that, at time $t$, the $L^1$-distance between two solutions can be estimated by the $L^1$-distance of the initial data. In our context, an $L^2$-comparison principle has been established in Proposition \ref{prop_kato_ineq}. Then, translating in space the data enables one to get an $L^2$-estimate on the difference quotient of the solution in terms of the $L^2$-norm of the difference quotient of the data. In particular, if the data are $H^1$, then the solution is $H^1$ as well (see \cite[Lemma 10]{DLS} in the full space). Since we are dealing with a boundary value problem, one has to be careful that the translated solutions remain inside the domain $\O$. This is the reason why we need to ensure that, in short time, if the data are compactly supported in space, then so is the solution, which is a statement of the finite speed propagation property. In that way, the boundary of the domain can be ignored, and one can argue similarly as in the full space.

\begin{proposition}[Finite speed propagation]
\label{prop_finite_wave}
Let $(u,\sigma,p)$ be the variational solution of the elasto-plastic problem given by Theorem~\ref{final_theorem_eps_eq_zero} for the initial condition $(u_0,v_0,\sigma_0,p_0)$ and the source term $f$. Suppose that there exists a compact set $K \subset \O$ such that 
$$\supp(v_0,\sigma_0,f(t)) \subset K \quad \text{ for all }t \in [0,T].$$
Then, for all $T^* \in (0,T]$ be such that $T^* < \dist (K,\partial \O)$, there exists a compact set $K^* \subset \O$ such that $\supp(\dot u,\sigma) \subset K^* \times [0,T^*]$.
\end{proposition}

\begin{proof}
By assumption, we know that for all $x\in \partial \Omega$, we have $\dist(x,K)>T^*$, so that we can find some $r_{x}>0$ such that $\dist(x,K) = r_x + T^*$. 
Using the fact that $\partial \Omega$ is compact, we obtain the existence of $p\in \N$ and $x_1,\ldots,x_p \in \partial \O$ such that
\[
\partial \Omega \subset \bigcup_{i=1}^p B\left(x_i,\frac{r_{x_i}}{4}\right).
\]
Observe that if $y\in B\left(x_i,\frac{r_{x_i}}{2}\right)$ for some  $1\le i\le p$, then
\[
\left| \dist(y,K) - \dist(x_i,K) \right| \le |y-x_i | \le \frac{r_{x_i}}{2}, 
\]
and consequently,
\[
\dist(y,K) \ge T^* + \frac{r_{x_i}}{2} >0,
\]
which implies that $v_0=0$, $\sigma_0=0$ and $f(t)=0$ in $\cup_{i=1}^p B\left(x_i,r_i/2\right)$ for all $t \in [0,T]$. 

Define $\eta = \min_{1\le i \le p} r_{x_i}/4 >0$, and consider the  boundary layer
\[
L_\eta=\left\{y \in \O: 0< \dist(y,\partial \Omega) < \eta \right\} \subset \bigcup_{i=1}^p  B\left(x_i,\frac{r_{x_i}}{2}\right)\cap \Omega
\]
so that $v_0=0$, $\sigma_0=0$ and $f(t)=0$ in $L_\eta$ for all $t \in [0,T]$. Let us show that $\dot u=0$ and $\sigma=0$ on $L_\eta \times [0,T^*]$. To this aim, let $x_0 \in L_\eta$, and,  the set $L_\eta$ being open, let $\rho_0 \in (0,\eta/2)$ be such that $B(x_0,\rho_0) \subset  L_\eta$. Now define the function $\varphi \in W^{1,\infty}(\R^n \times (0,T^*))$ as
\[
\varphi(t,x) = 
\begin{cases}
T^*-t+ \rho_0-|x-x_0|&\text{if }
\begin{cases} 
t\in [0,T^*],\\
\rho_0<|x-x_0|<\rho_0+T^*-t,
\end{cases}\\
T^*-t &\text{if }
\begin{cases} 
t\in [0,T^*],\\
|x-x_0|<\rho_0,
\end{cases}\\
0 &\text{otherwise. }
\end{cases}
\]
Note that $\varphi\geq 0$, and 
$$\begin{cases}
\dot \varphi=-\mathds{1}_{\{(x,t) : \; t \in [0,T^*],\; |x-x_0|<\rho_0 +T^*-t\}},\\
\nabla \varphi =-\frac{x-x_0}{|x-x_0|}\mathds{1}_{\{(x,t) : \; t \in [0,T^*],\; \rho_0<|x-x_0|<\rho_0 +T^*-t\}},
\end{cases}$$
which implies that $-|\dot u|^2\dot \varphi - |\sigma |^2\dot \varphi+2\sigma\cdot \nabla \varphi\dot u \geq 0$ a.e in $\Omega \times (0,T^*)$. Consequently, using the comparison principle Proposition~\ref{prop_kato_ineq} (with $(u,\sigma,p)$ and the null solution) it follows that
\begin{multline*}
-\int_0^{T^*} \int_{B(x_0,\rho_0)} (|\dot u|^2+ |\sigma |^2)\dot \varphi \dxdt+2\int_0^{T^*} \int_{B(x_0,\rho_0)} \sigma \cdot \nabla \varphi\dot u \dxdt \\
\leq \int_{B(x_0,\rho_0+T^*)} (|v_0|^2 + |\sigma_0 |^2)\varphi(0) \dx + \int_0^{T^*} \int_{B(x_0,\rho_0+T^*)}|f| |\dot u|\varphi\dxdt,
\end{multline*}
and since the spatial derivative of $\varphi$ vanish on $B(x_0,\rho_0)$, we get that
\begin{multline*}
\int_0^{T^*} \int_{B(x_0,\rho_0)} (|\dot u|^2 + |\sigma |^2) \dxdt\\
\leq \int_{B(x_0,\rho_0+T^*)} (|v_0|^2+ |\sigma_0 |^2)\varphi(0) \dx + \int_0^{T^*} \int_{B(x_0,\rho_0+T^*)}|f| |\dot u|\varphi\dxdt.
\end{multline*}
Using the definition of $x_0$, we obtain that for every $y\in B(x_0,\rho_0+T^*)$,
\[
\left| \dist(y,K) - \dist(x_0,K)\right| \le \rho_0+T^*,
\]
and thus
\[
\dist(y, K) >T^* + \eta -\rho_0-T^* > \frac{\eta}{2}>0.
\]
Consequently,
$$\int_0^{T^*} \int_{B(x_0,\rho_0)} (|\dot u|^2 + |\sigma |^2) \dxdt \leq 0$$
which implies that both $\dot u$ and $\sigma$ vanish in $L_\eta\times (0,T^*)$. The conclusion thus follows by setting $K^*=\O \setminus L_\eta$.
\end{proof}

\begin{remark}
\label{rem_support_finite_wave}
Since $u \in W^{2,\infty}([0,T];L^2(\O))$, we have, for all $t \in [0,T]$,
$$u(t)=u_0+\int_0^t \dot u(s)\ds,$$
 where the integral is intended as a Bochner integral in $L^2(\O)$. Therefore, if further $\supp(u_0) \subset K$, it follows that $\supp(u) \subset K^* \times [0,T^*]$. As a consequence, the measure $Du$ is also compactly supported in  $K^* \times [0,T^*]$, and the additive decomposition entails that $\supp(p) \subset K^* \times [0,T^*]$ as well.
 \end{remark}

We are now in position to state the regularity result.

\begin{theorem}[Short time regularity]
Let $(u,\sigma,p)$ be the variational solution to the elasto-plastic problem given by Theorem~\ref{final_theorem_eps_eq_zero} for the initial condition $(u_0,v_0,\sigma_0,p_0)$ and the source term $f$. Suppose that there exists a compact set $K \subset \O$ such that 
$$\supp(u_0,v_0,\sigma_0,p_0,f(t)) \subset K \quad  \text{ for all }t \in [0,T],$$
and that 
$$\sigma_0 \in H^1(\Omega;\R^n), \quad f \in H^1(\O \times (0,T)).$$
Then, for all $T^* \in (0,T]$ be such that $T^* < \dist (K,\partial \O)$, we have 
$$\begin{cases}
u \in H^1([0,T^*];H^1(\O)),\\
\sigma \in L^2(0,T^*;H^1(\O;\R^n)),\\
p \in H^1([0,T^*];L^2(\O;\R^n)).
\end{cases}$$
\end{theorem}

\begin{proof}
Using Proposition~\ref{prop_finite_wave} and Remark~\ref{rem_support_finite_wave}, we know that for all $T^*<\dist (K,\partial \O)$, there exists a compact set, $K^*\subset \Omega$ such that $\supp(u,\sigma,p) \subset K^* \times [0,T^*]$.  Since $K^*$ is a compact subset of $\Omega$, there exists $\delta>0$ such that for all $h\in \R^n$ with $|h|< \delta$, the sets $K^*+ h$ are also compactly embedded in $\Omega$. Let $\Omega'$ be a bounded smooth open subset of $\R^n$ such that $\overline \Omega \subset \O'$, and for all $h \in \R^n$ with $|h|< \delta$, $\Omega + h \subset \Omega'$. 

\medskip

\noindent \textbf{Step 1: Extension on $\O' \times (0,T^*)$.} We denote by $\bar f$, $\bar{u}$ and $\bar{\sigma}$ the extensions of $f$, $u$ and $\sigma$ by zero on $ \O' \times (0,T^*)$. Clearly, one has $\bar f \in H^1([0,T^*];L^2(\O'))$, $\bar{u}\in W^{2,\infty}([0,T^*];L^2(\O'))$, $\bar{\sigma}\in W^{1,\infty}([0,T^*];L^2(\O';\R^n))$ and 
$$\|\bar \sigma \|_{L^\infty(\O' \times(0,T^*))}\le 1.$$ 
In addition,  for all $t\in [0,T^*]$, since the (inner) trace on $\partial \Omega$ of $u(t)$  vanishes, \cite[Theorem 3.87]{AmbrosioFuscoPallara} ensures that the function $\bar{u}(t) \in BV(\O')$ and $
D\bar{u}(t) = Du(t)$ in $\mathcal M(\O')$. Hence, we get that $\bar{u} \in \mathcal{C}^{0,1}([0,T^*];BV(\O'))$. Similarly, since the (inner) normal trace on $\partial \O$ of $\sigma$ vanishes, we deduce that  $\Div \bar\sigma \in L^\infty(0,T^*;L^2(\O'))$, which implies that $\bar{\sigma} \in L^\infty (0,T^*;H(\Div,\O'))$. 

For every $t\in [0,T^*]$, we define the measure $\bar{p}(t) \in \mathcal M(\O';\R^n)$ by
\[
\bar{p}(t) = D\bar{u}(t) - \bar{\sigma}(t).
\]
Using the regularity of $\bar{u}$ and $\bar{\sigma}$, we obtain that $\bar{p} \in \mathcal{C}^{0,1}([0,T^*];\mathcal{M}(\O';\R^n))$. For a.e. $t \in [0,T^*]$, we consider the Radon measure $[\bar{\sigma}(t)\cdot\dot{ \bar p}(t)]$ on $\Omega'$ as in Definition \ref{def_Sigma_p}. Clearly $[\bar{\sigma}(t)\cdot\dot{ \bar p}(t)]=[{\sigma}(t)\cdot\dot{p}(t)]$ in $\mathcal M(\O)$, while Remark \ref{rmk:ineq1} ensures that $|[\bar{\sigma}(t)\cdot\dot{ \bar p}(t)]|(\O' \setminus \O)\leq |\dot {\bar p}(t)|(\O' \setminus \O)=0$. It implies that for a.e. $t\in [0,T^*]$, we have $| \dot{ \bar p}(t) | = [ \bar{\sigma}(t)\cdot\dot{ \bar p}(t) ]$ in $\mathcal{M}(\Omega')$.

\medskip

\noindent \textbf{Step 2: Spatial translation.} For every $h\in \R^n$ be such that $|h|<\delta$, we define the translation operator $\tau_h$ of a generic function $F :\O' \times (0,T^*) \to \R$ by
\[
\tau_h F(x,t) = F(x+h,t) \quad \text{ for all }(x,t) \in \O \times (0,T^*).
\]
Then, $\tau_h\bar f \in H^1([0,T^*];L^2(\O))$,  $\tau_h \bar{u}\in W^{2,\infty}([0,T^*];L^2(\O))$ and $\tau_h \bar{\sigma} \in W^{1,\infty}([0,T^*];L^2(\O;\R^n))$ with $\| \tau_h \bar{\sigma} \|_{L^\infty(\O\times (0,T^*))}\le 1$.  According to \cite[Remark 3.18]{AmbrosioFuscoPallara}, for all $t \in [0,T^*]$, we have that $D(\tau_h \bar u)(t)={\tau_{-h}}_\# D\bar u(t)$ (the push-forward of the measure $D\bar u(t)$ by the mapping $x \mapsto x-h$) which implies that $\tau_h \bar{u} \in \mathcal{C}^{0,1}([0,T^*];BV(\O))$. Finally, since for all $t \in [0,T^*]$ the push-forward measure ${\tau_{-h}}_\# \bar{p}(t) \in \mathcal{M}(\O;\R^n)$ satisfies
\[
{\tau_{-h}}_\# \bar{p}(t) = D\tau_h \bar{u}(t) - \tau_h \bar{\sigma}(t),
\]
the regularity of $\tau_h \bar{u}$ and $\tau_h \bar{\sigma}$ ensures that ${\tau_{-h}}_\# \bar{p} \in \mathcal{C}^{0,1}([0,T^*];\mathcal{M}(\O;\R^n))$.

\medskip

\noindent \textbf{Step 3 : The translation of the solution is a solution.} We define the translation of the solution $(u_h,\sigma_h,p_h):=(\tau_h \bar u,\tau_h \bar \sigma,{\tau_{-h}}_\# \bar p)$ and $f_h:=\tau_h \bar f$. Let us show that
\begin{enumerate}
\item\label{1} Regularity properties: we have $u_h \in W^{2,\infty}([0,T^*];L^2(\O)) \cap \mathcal{C}^{0,1}([0,T^*];BV(\O))$, $\sigma_h \in W^{1,\infty}([0,T^*];L^2(\Omega;\R^n))$ and $ p_h \in \mathcal{C}^{0,1}([0,T^*];\mathcal{M}(\Omega;\R^n))$;
\item\label{2} Equation of motion: $\ddot u_h - \Div  \sigma_h = f_h$ in $L^\infty(0,T^*;L^2(\Omega))$;
\item\label{3} Additive decomposition: for every $t\in [0,T^*]$, 
$$Du_h(t)  =   \sigma_h(t)+ p_h(t) \quad \text{ in }\mathcal M(\O;\R^n);$$
\item\label{3,5} Stress constraint: for every $t\in [0,T^*]$, $|\sigma_h(t)|\leq 1$ a.e. in $\O$;
\item\label{4} Flow rule: for a.e. $t\in [0,T^*]$, $| \dot p_h(t) | = [\sigma_h(t)\cdot \dot p_h(t) ]$ in $\mathcal{M}(\Omega)$;
\item\label{5} Boundary condition: for a.e. $t\in [0,T^*]$,
\[
\sigma_h(t)\cdot \nu+\psi'_\lambda(\dot u_h(t)) =0 \quad \HH^{n-1}\text{-a.e. on }\partial \O;
\]
\item\label{6} Initial conditions:
\[
u_h(0) = \tau_hu_0, \quad \dot  u_h(0) = \tau_h v_0, \quad \sigma_h(0)  = \tau_h\sigma_0, \quad p_h(0)   =  {\tau_{-h}}_\# p_0. 
\]
\end{enumerate}
Items \ref{1}, \ref{3} and \ref{3,5} have already been proved in step 2. Due to the definition of $\tau_h$, items \ref{2} and \ref{6} are automatically satisfied. Let us examine the point \ref{5}. Since $\supp(u,\sigma) \subset K^* \times [0,T^*]$, then for all $t \in [0,T^*]$,  $\supp(u_h(t),\sigma_h(t)) \subset (K^*+h)$ which is a compact subset of $\O$ as long as $|h|<\delta$. Therefore, for all $t \in [0,T^*]$, we have $\dot u_h(t)=0$ and $\sigma_h(t)\cdot \nu=0$ on $\partial \O$, and thus \ref{5} holds. It remains to show item \ref{4}. By Definition \ref{def_Sigma_p} of duality, we know that for a.e. $t\in [0,T^*]$ and for all $\varphi \in \mathcal{C}^\infty_c(\Omega)$,
$$ \Braket{[ \dot p_h(t)\cdot \sigma_h(t) ],\varphi} 
=  - \int_\Omega \dot u_h(t) \Div \sigma_h(t) \varphi\dx - \int_\Omega \sigma_h(t)\cdot \nabla \varphi \dot u_h(t)\dx   - \int_\Omega \dot \sigma_h(t)\cdot \sigma_h(t) \varphi\dx.$$
Using a change of variables (and since $|h|<\delta$), we get that
\begin{multline*}
 \Braket{[ \dot p_h(t)\cdot \sigma_h(t) ],\varphi} 
=  - \int_\Omega \dot {\bar u}(t) \Div \bar \sigma(t) \varphi(\cdot -h)\dx\\ - \int_\Omega\bar  \sigma(t)\cdot \nabla \varphi(\cdot -h) \dot{\bar  u}(t)\dx 
 - \int_\Omega \dot {\bar \sigma}(t)\cdot \bar \sigma(t) \varphi(\cdot -h)\dx,
\end{multline*}
and according to the integration by parts formula (Proposition~\ref{IPP_sig_u_varphi}), we obtain
$$\Braket{[ \dot p_h(t)\cdot\sigma_h(t) ],\varphi} = \Braket{|\dot{ \bar p}|,\varphi(\cdot -h)} = \Braket{|\dot p_h|,\varphi},$$
where we used that $\varphi(\cdot -h) \in \mathcal{C}^\infty_c(\Omega')$ and  $| \dot{ \bar p}(t) | = \left[ \bar{\sigma}(t)\cdot \dot{ \bar p}(t) \right]$ in $\mathcal{M}(\Omega')$.

\medskip

 \textbf{Step 4 : Regularity.}
Now  that $(u_h,\sigma_h,p_h)$ is a solution associated to the initial condition $(\tau_h u_0,\tau_h v_0, \tau_h \sigma_0, {\tau_{-h}}_\# p_0)$ and source term $f_h$, Proposition~\ref{prop_kato_ineq} (with the test function $\varphi(x,t) = T^*-t$) then implies that
\begin{multline*}
\int_0^{T^*} \int_{\Omega} \left( \dot u_h -\dot {u}\right)^2\dxdt + \int_0^T \int_{\Omega} \left|\sigma_h -{\sigma}\right|^2 \dxdt \\
\le T^*\int_{\Omega} \left(\tau_h v_0 -{v_0}\right)^2  \dx + T^*\int_{\Omega} \left|\tau_h\sigma_0 -{\sigma_0}\right|^2 \dx+ 2T^* \int_0^{T^*} \int_{\Omega} |\tau_h f-f| |\tau_h \dot u -\dot u| \dxdt.
\end{multline*}
According to Young's inequality, and since $v_0 \in H^1(\Omega)$, $\sigma_0 \in H^1(\Omega;\R^n)$, $f \in H^1(\O\times (0,T^*))$, we get that
\begin{multline*}
\int_0^{T^*} \int_{\Omega} \left(  \dot u_h -\dot {u}\right)^2\dxdt + \int_0^{T^*} \int_{\Omega} \left|\sigma_h -{\sigma}\right|^2 \dxdt \\
\le C(T^*) |h|^2 \left( \left\| \nabla v_0 \right\|_2^2 + \left\| \nabla \sigma_0 \right\|_2^2+\int_0^{T^*}\|\nabla f(t)\|_2^2\dt \right),
\end{multline*}
from which we get that $\dot u\in L^2(0,T^*;H^1(\O))$ and $\sigma\in L^2(0,T^*;H^1(\O;\R^n))$. Writting for all $t \in [0,T]$,
$$u(t)=u_0 +\int_0^t \dot u(s)\ds$$
as a Bochner integral in $L^2(\O)$, we obtain that $u \in L^2(0,T^*;H^1(\O))$, or still $u \in H^1([0,T^*];H^1(\O))$. Finally, we have $p=\nabla u-\sigma\in L^2(0,T^*;L^2(\O;\R^n))$ and $\dot p=\nabla \dot u-\dot \sigma\in L^2(0,T^*;L^2(\O;\R^n))$ which shows the desired regularity $p \in H^1([0,T^*];L^2(\O;\R^n))$.
\end{proof}

\begin{remark}
A similar result holds for the variational solutions to the elasto-plastic problem with homogeneous Dirichlet or Neumann boundary conditions given by Theorem \ref{DN}.
\end{remark}

\section{Dissipative formulation}
\label{sec:6}

In this last Section, we establish precise links between the variational and the dissipative formulations by making rigorous the formal computations done in Subsection \ref{sec:entrop}. We show that any variational solution generates a dissipative solution by performing the manipulations on the approximate elasto-visco-plastic model (for which the solution is essentially smooth), and then by passing to the limit as the viscosity parameter $\e \to 0$. A partial converse statement is proved, provided the dissipative solutions are smoother in time. We then employ measure theoretic arguments to establish that a solution of the variational problem can be constructed.

\medskip

In order to define precisely the dissipative formulation of the dynamical elasto-plastic problem, we introduce the convex set
$$K:=\R \times B \subset \R^{n+1},$$
and, for $i=1,\ldots,n$, we define the matrices 
$$A_i:=-2 \boldsymbol e_1 \odot \boldsymbol e_{i+1},$$ 
where $\{\boldsymbol e_1,\ldots,\boldsymbol e_n\}$ stands for the canonical basis of $\R^n$. For all $x \in \partial \O$, we denote the boundary matrix by
$$A_\nu(x):=\sum_{i=1}^n A_i \nu_i(x),$$
where $\nu(x)$ is the outer unit normal to $\O$ at $x\in \partial \O$. In addition, for each $\lambda>0$ and all $x \in \partial \O$, we define the matrix $M(x) \in \mathbb M^{(n+1)\times (n+1)}$ by
$$M(x):=\lambda^{-1}\boldsymbol e_1 \otimes  \boldsymbol e_1 + \lambda \sum_{i,j=1}^{n} \nu_i(x) \nu_j(x) \boldsymbol e_{i+1} \odot \boldsymbol e_{j+1}.$$
In this section, we always assume that $\O$ is of class $\mathcal C^1$, so the normal $\nu \in \mathcal C(\partial \O)$, and thus both matrices $A_\nu$ and $M \in \mathcal C(\partial \O;\mathbb M^{(n+1)\times (n+1)})$.

\begin{definition}\label{def:entrop}
Let $U_0 \in L^2(\O;K)$ be an initial data and $F \in L^2(\O \times (0,T);\R^{n+1})$ be a source term. A function $U \in L^2(\Omega \times (0,T); K)$ is a {\it dissipative solution to the elasto-plastic problem} if for all constant vector $\kappa \in K$ and all $\varphi \in W^{1,\infty}(\Omega \times (0,T))$ with $\varphi \geq 0$,
\begin{multline*}
\int_{0}^T \int_{\Omega}\left(  |U-\kappa|^2\dot \varphi + \sum_{i=1}^n   A_i(U-\kappa)\cdot (U-\kappa)\partial_{x_i}\varphi\right) \dxdt \\
+ \int_\Omega \left|U_0-\kappa\right|^2 \varphi(0)\dx + 2\int_{0}^T \int_{\Omega} F\cdot (U-\kappa)\varphi\dxdt +\int_0^T\int_{\partial \Omega} M \kappa^{+}\cdot \kappa^{+} \varphi \dd \mathcal H^{n-1} \dt \ge 0,\\
\end{multline*}
where $\kappa_+$ denotes the orthorgonal projection of $\kappa$ onto $ \Ker(A_\nu+M) \cap \im A_\nu$.
\end{definition}

\begin{remark}\label{remark:algebraic}
Using elementary algebraic computations, we have for all $\kappa=(k,\tau) \in \R^{n+1}$ and all $\xi \in \R^n$
$$\sum_{i=1}^n \xi_i A_i \kappa \cdot \kappa=-2 (\tau \cdot \xi) k.$$
In addition, according to \cite[Lemma 1]{MDS}, we have that
$$\R^{n+1} = \Ker A_\nu  \oplus   \big(\Ker(A_{\nu}- {M})\cap \im A_{\nu} \big)   \oplus  \big( \Ker(A_{\nu}+ {M}) \cap \im A_{\nu} \big).$$
For each $\kappa=(k,\tau) \in \R^{n+1}$, denoting by $\kappa^{(0)}$, $\kappa^-$ and $\kappa^+$ the projection of $\kappa$ onto $ \Ker A_\nu$, $\Ker(A_{\nu}- {M})\cap \im A_{\nu} $ and $\Ker(A_{\nu}+ {M}) \cap \im A_{\nu} $, respectively, we get that $\kappa=\kappa^{(0)}+\kappa^-+\kappa^+$ where
$$\begin{cases}
\kappa^{(0)}=(0,\tau-(\tau\cdot\nu)\nu),\\
\kappa^-=\left(\frac{k-\lambda \tau\cdot \nu}{2},\left(\frac{\tau\cdot \nu}{2}-\frac{k}{2\lambda}\right)\nu\right),\\
\kappa^+=\left(\frac{k+\lambda \tau\cdot \nu}{2},\left(\frac{\tau\cdot \nu}{2}+\frac{k}{2\lambda}\right)\nu\right).
\end{cases}$$
In particular
$$M \kappa^\pm \cdot \kappa^\pm=2\lambda \left( \frac{k}{2\lambda} \pm \frac{\tau\cdot \nu}{2} \right) ^2.$$
\end{remark}

The following result states that variational solutions to the elasto-plastic problem generate dissipative solutions.

\begin{proposition}
Let $(u,\sigma,p)$  be the variational solution to the elasto-plastic problem given by Theorem \ref{final_theorem_eps_eq_zero} for the initial data $(u_0,v_0,\sigma_0,p_0)$ and the source term $f$, and define $U:=(\dot u,\sigma)$. Then $U \in L^2(\O \times (0,T);K)$ is a dissipative solution to the elasto-plastic problem according to Definition \ref{def:entrop} for the initial data $U_0=(v_0,\sigma_0)$ and the source term $F=(f,0)$.
\end{proposition}

\begin{proof}
The proof is very close to that of Proposition \ref{prop_kato_ineq}.  For fixed $\e >0$, let $(u_\e,\sigma_\e,p_\e)$ be the solution of the elasto-visco-plastic problem given by Theorem \ref{final_theorem_at_level_epsilon} for the initial condition $(u_{0},v_{0},\sigma_0,p_{0})$ and the source terms $f$ and $g_\e=\e\nabla v_0\cdot \nu$. Let $k \in \R$ and $\tau \in B$, taking $(\dot u_\e - k) \varphi \in L^2(0,T;H^1(\O))$ as test function in the variational formulation \eqref{eq_motion_eps_equal_zero}, and using the additive decomposition \eqref{decomposition_nablau_eps} together with the boundary condition \eqref{eq:bdry_cond_eps} yields
\begin{multline*}
\int_0^T \int_\Omega \ddot u_\e \left(\dot u_\e - k\right) \varphi\dxdt + \int_0^T \int_\Omega \dot \sigma_\e\cdot (\sigma_\e-\tau)\varphi \dxdt + \int_0^T \int_\O(\sigma_\e-\tau) \cdot \dot p_\e \varphi\dxdt\\
 + \e\int_0^T\int_\O |\nabla \dot u_\e|^2 \varphi\dxdt+ \int_0^T \int_\Omega (\sigma_\e-\tau)\cdot \nabla \varphi (\dot u_\e-k)\dxdt+ \int_0^T \int_\Omega \tau \cdot \nabla \big((\dot u_\e-k)\varphi\big)\dxdt\\
+\e\!\! \int_0^T\int_\O\! \nabla \dot u_\e \cdot \nabla \varphi (\dot u_\e-k)\dxdt - \!\int_0^T \int_{\partial \Omega} \!\!(\sigma_\e+\e\nabla \dot u_\e) \cdot \nu  (\dot u_\e -k) \varphi\dd \mathcal{H}^{n-1} \dt
\!= \!\!\! \int_0^T \int_\Omega f(\dot u_\e - k) \varphi\dxdt.
\end{multline*}
Using next the flow rule \eqref{snd_eq_continuous} and the fact that $\tau \in B$, we infer that $(\sigma_\e-\tau) \cdot \dot p_\e\geq 0$ a.e. in $\O \times (0,T)$, and thus
\begin{multline}\label{eq:entrop1}
\int_0^T \int_\Omega \ddot u_\e \left(\dot u_\e - k\right) \varphi\dxdt + \int_0^T \int_\Omega \dot \sigma_\e\cdot (\sigma_\e-\tau)\varphi \dxdt
+ \int_0^T \int_\Omega (\sigma_\e-\tau)\cdot \nabla \varphi (\dot u_\e-k)\dxdt\\
+ \int_0^T \int_\Omega \tau \cdot \nabla \big((\dot u_\e-k)\varphi\big)\dxdt+\e \int_0^T\int_\O \nabla \dot u_\e \cdot \nabla \varphi (\dot u_\e-k)\dxdt \\
- \int_0^T \int_{\partial \Omega} (\sigma_\e+\e\nabla \dot u_\e) \cdot \nu  (\dot u_\e -k) \varphi\dd \mathcal{H}^{n-1} \dt\leq  \int_0^T \int_\Omega f(\dot u_\e - k) \varphi\dxdt.
\end{multline}
Since $\dot u_\e\in W^{1,\infty}([0,T];L^2(\O))$ and $\sigma_\e\in W^{1,\infty}([0,T];L^2(\O;\R^n))$, integrating by parts with respect to time yields
\begin{multline}\label{eq:entrop2}
2 \int_0^T \int_\Omega \ddot u_\e (\dot u_\e-k)\varphi\dxdt =   - \int_0^T \int_\Omega (\dot u_\e-k)^2 \dot \varphi\dxdt \\
+ \int_\Omega (\dot u_\e(T)-k)^2\varphi(T) \dx - \int_\Omega (v_{0}-k)^2\varphi(0) \dx,
\end{multline}
and similarly
\begin{multline}\label{eq:entrop3}
2 \int_0^T \int_\Omega (\sigma_\e-\tau)\cdot\dot \sigma_\e \varphi \dxdt=  -  \int_0^T \int_\Omega |\sigma_\e-\tau|^2 \dot \varphi\dxdt  \\
+ \int_\Omega |\sigma_\e(T)-\tau|^2\varphi(T) \dx - \int_\Omega |\sigma_0-\tau|^2\varphi(0) \dx.
\end{multline}
In addition, since $\tau$ is constant and $\dot u_\e \in L^2(0,T;H^1(\O))$, using an integration by parts formula with respect to the space variable leads to
\begin{equation}\label{eq:entrop4}
\int_0^T\int_\O \tau \cdot \nabla \big((\dot u_\e-k)\varphi\big) \dxdt=\int_0^T\int_{\partial \O} (\tau \cdot \nu) (\dot u_\e-k) \varphi \dd\HH^{n-1}\dt.
\end{equation}
Gathering \eqref{eq:entrop1}--\eqref{eq:entrop4} yields
\begin{multline*}
\int_0^T \int_\Omega (\dot u_\e-k)^2 \dot \varphi\dxdt + \int_0^T \int_\Omega |\sigma_\e-\tau|^2 \dot \varphi\dxdt + \int_\Omega (v_{0}-k)^2\varphi(0) \dx +\int_\Omega |\sigma_0-\tau|^2\varphi(0) \dx\\
 - 2\int_0^T \int_\Omega (\sigma_\e-\tau) \cdot \nabla \varphi (\dot u_\e -k)\dxdt -2 \e \int_0^T\int_\O \nabla \dot u_\e \cdot \nabla \varphi (\dot u_\e-k)\dxdt\\
  + 2\int_0^T \int_\Omega f(\dot u_\e-k)\varphi\dxdt+2 \int_0^T \int_{\partial \Omega} \big((\sigma_\e+\e\nabla \dot u_\e) \cdot \nu-\tau \cdot \nu\big)(\dot u_\e- k) \varphi \dd \mathcal{H}^{n-1} \dt \geq 0,
\end{multline*}
and passing to the limit as $\e \to 0$, Proposition \ref{prop:strong_conv} and \eqref{eq:strong_conv5} imply that
\begin{multline*}
\int_0^T \int_\Omega (\dot u-k)^2 \dot \varphi\dxdt + \int_0^T \int_\Omega |\sigma-\tau|^2 \dot \varphi\dxdt + \int_\Omega (v_{0}-k)^2\varphi(0) \dx +\int_\Omega |\sigma_0-\tau|^2\varphi(0) \dx\\
 - 2\int_0^T \int_\Omega (\sigma-\tau) \cdot \nabla \varphi (\dot u -k)\dxdt + 2\int_0^T \int_\Omega f(\dot u-k)\varphi\dxdt \\
 +2 \int_0^T \int_{\partial \Omega} (\sigma \cdot \nu-\tau \cdot \nu)(\lambda \psi'_\lambda(\dot u)- k) \varphi \dd \mathcal{H}^{n-1} \dt \geq 0.
\end{multline*}
Thanks to algebraic manipulations, we have for a.e. $(x,t) \in \partial \O \times (0,T)$,
\begin{multline*}
 2(\sigma\cdot \nu-\tau\cdot \nu )( \lambda \psi'_\lambda(\dot u) - k) \\
=   2\lambda  \left( \left( \frac{ \lambda \psi'_\lambda(\dot u)-k}{2\lambda} + \frac{\sigma\cdot \nu-\tau\cdot \nu}{2} \right) ^2- \left( \frac{\lambda \psi'_\lambda(\dot u)-k}{2\lambda} - \frac{\sigma\cdot \nu-\tau\cdot \nu}{2} \right) ^2 \right) \\
=  2\lambda \left( \left( \frac{k}{2\lambda} + \frac{\tau\cdot \nu}{2} \right) ^2 - \left( \frac{\lambda\psi'_\lambda(\dot u)-k}{2\lambda} - \frac{\sigma\cdot \nu-\tau\cdot \nu}{2} \right) ^2\right),
\end{multline*}
where we used the boundary condition $\sigma\cdot\nu+\psi'_\lambda(\dot u)=0$ on $\partial \O \times (0,T)$. Finally, from Remark \ref{remark:algebraic}, it follows that
\begin{multline*}
\int_{0}^T \int_{\Omega}  \left|U-\kappa\right|^2\dot \varphi\dxdt +  \sum_{i=1}^n \int_{0}^T \int_{\Omega} A_i (U-\kappa)\cdot (U-\kappa)\partial_{x_i}\varphi\dxdt \\
 + \int_\Omega \left|U_0-\kappa\right|^2 \varphi(0)\dx + 2\int_{0}^T \int_{\Omega} F\cdot (U-\kappa)\varphi\dxdt  +\int_0^T\int_{\partial \Omega} M \kappa^{+}\cdot \kappa^{+} \varphi \dd \mathcal H^{n-1} \dt \ge 0,
\end{multline*}
as required.
\end{proof}

We finally show that, provided additional regularity assumptions, any dissipative solution to the elasto-plastic problem generates a variational solution.

\begin{proposition}
Assume that $\O \subset \R^n$ is a bounded open set with $\mathcal C^2$ boundary. Let $U=(v,\sigma)\in W^{1,\infty}([0,T];L^2(\O;K))$ be a dissipative solution to the elasto-plastic problem according to Definition \ref{def:entrop} for the initial data $U_0=(v_0,\sigma_0) \in L^2(\O;K)$ and the source term $F=(f,0)$ with $f \in L^\infty(0,T;L^2(\O))$. Then, for all $u_0 \in BV(\O) \cap L^2(\O)$,  there exists a triple $(u,\sigma,p)$ which is a variational solution to the elasto-plastic problem associated to the initial data $(u_0,v_0,\sigma_0,Du_0-\sigma_0)$ and the source term $f$.
\end{proposition}

\begin{proof}
We split the proof into several steps.

\medskip

\noindent {\bf Step 1: Initial conditions for $v$ and $\sigma$.} According to \cite[Lemma 3]{MDS}, we infer that the initial condition is satisfied in the essential-limit sense, {\it i.e.}, for all $\phi \in \mathcal C^1_c(\O)$,
$$\lim_{t \to 0} \lim_{\alpha \to 0} \frac{1}{\alpha} \int_{t-\alpha}^t\int_\O |U(x,s)-U_0(x)|^2\phi(x)\dx\ds=0.$$
On the other hand, since $U \in W^{1,\infty}([0,T];L^2(\O;\R^{n+1}))$, then $U(t) \to U(0)$ strongly in $L^2(\O;\R^{n+1})$ as $t \to 0$. It thus follows that $U(0)=U_0$, or still $v(0)=v_0$ and $\sigma(0)=\sigma_0$. By Definition \ref{def:entrop} of dissipative solutions, and using the regularity assumption $U \in W^{1,\infty}([0,T];L^2(\O;\R^{n+1}))$, we can integrate by parts with respect to time to get that, for all  $\varphi \in \mathcal C^1_c(\R^n \times (0,T))$ with $\varphi \geq 0$ and all constant vector $\kappa \in K$,
\begin{multline*}
\int_0^T \int_\O \left(-2 \dot U \cdot  (U-\kappa) \varphi + \sum_{i=1}^n A_i(U-\kappa)\cdot  (U-\kappa)\partial_{x_i}\varphi\right)\dxdt\\
 + 2\int_0^T\int_\O F\cdot (U-\kappa)\varphi \dxdt+\int_0^T \int_{\partial \O}M\kappa^+\cdot\kappa^+\varphi \dd\HH^{n-1}\dt \geq 0,
\end{multline*}
or still
\begin{multline*}
\int_0^T \int_\O \left(-2 (U \cdot  \dot U) \varphi +\sum_{i=1}^n  (A_i U\cdot  U) \partial_{x_i}\varphi +2(F\cdot U) \varphi \right)\dxdt\\
+ 2\kappa \cdot  \int_0^T \int_\O \left( \dot U \varphi - \sum_{i=1}^n  (A_i U) \partial_{x_i}\varphi-F\varphi \right)\!\!\dxdt+\int_0^T \int_{\partial \O}\left(M\kappa^+\cdot\kappa^++ A_\nu \kappa\cdot \kappa\right)\varphi\dd\HH^{n-1}\dt \geq 0.
\end{multline*}
According to \cite[Lemma 1]{MDS}, we have that $M\kappa^+\cdot\kappa^++ A_\nu \kappa\cdot \kappa=M\kappa^-\cdot\kappa^-$, and thus
\begin{multline}\label{eq:hv1}
\int_0^T \int_\O \left(-2 (U \cdot  \dot U) \varphi +\sum_{i=1}^n  (A_i U\cdot U) \partial_{x_i}\varphi +2(F\cdot U) \varphi \right)\dxdt\\
+ 2\kappa \cdot  \int_0^T \int_\O \left( \dot U \varphi - \sum_{i=1}^n  (A_i U) \partial_{x_i}\varphi-F\varphi \right)\dxdt+\int_0^T \int_{\partial \O}M\kappa^-\cdot\kappa^-\varphi\dd\HH^{n-1}\dt \geq 0.
\end{multline}

\medskip

\noindent {\bf Step 2: Definition of $(u,\sigma,p)$ and first properties.} For all $t \in [0,T]$, let us define the displacement as
$$u(t):=u_0+\int_0^tv(s)\ds \text{ in } L^2(\O).$$
According to the regularity assumption on $U$, we infer that $u \in W^{2,\infty}([0,T];L^2(\O))$ and that $\sigma \in W^{1,\infty}([0,T];L^2(\O;\R^n))$. In addition, for all $t \in [0,T]$, 
we have $|\sigma(t)|\leq 1$ a.e. in $\O$. Note that, by construction $u(0)=u_0$, and by step 1, $\dot u(0)=v_0$. We also define the plastic strain by
\begin{equation}\label{eq:def_p}
p:=Du-\sigma \in W^{1,\infty}([0,T];H^{-1}(\O;\R^n)).
\end{equation}
We now use Remark \ref{remark:algebraic} to rewrite \eqref{eq:hv1} as 
\begin{multline}\label{eq:hv2}
-\int_0^T \int_\O \big(\dot u\ddot u\varphi +\sigma \cdot \dot \sigma \varphi +( \sigma\cdot \nabla \varphi)\dot u -f\dot u \varphi \big)\dx\\
\hfill+ k  \int_0^T \int_\O \left( \ddot u \varphi +\sigma\cdot \nabla \varphi-f\varphi \right)\dxdt+\tau\cdot \int_0^T \int_\O \left( \dot \sigma \varphi +\dot u \nabla \varphi \right)\dxdt\\
\hfill+\lambda\int_0^T \int_{\partial \O} \left( \frac{k}{2\lambda} - \frac{\tau\cdot \nu}{2} \right) ^2\varphi\dd\HH^{n-1}\dt \geq 0,
\end{multline}
for all $\kappa=(k,\tau) \in \R \times B$ and  all $\varphi \in \mathcal C^1_c(\R^n \times (0,T))$ with $\varphi \geq 0$. Choosing $k=0$ and $\tau=0$, we deduce that 
$$-\int_0^T \int_\O \big( \dot u\ddot u\varphi+\sigma \cdot \dot \sigma \varphi + (\sigma\cdot \nabla \varphi)\dot u -f\dot u \varphi \big)\dxdt\geq 0.$$
In particular, one can localize in time to get,  for all $\phi \in \mathcal C^1_c(\R^n)$ with $\phi \geq 0$ and for a.e. $t \in [0,T]$,
$$-\int_\O \big(\dot u(t)\ddot u(t)\phi+\sigma(t) \cdot \dot \sigma(t) \phi + (\sigma(t)\cdot \nabla \phi) \dot u(t) -f(t)\dot u(t) \phi \big)\dx\geq 0.$$
As a consequence, for a.e. $t \in [0,T]$, there exists a non-negative measure $\mu(t) \in \mathcal M(\R^n)$ compactly supported in $\overline \O$ such that for all $\phi \in \mathcal C^1_c(\R^n)$
\begin{equation}\label{defmut}
\langle\mu(t),\phi\rangle=-\int_\O \big(\dot u(t)\ddot u(t)\phi+\sigma(t) \cdot \dot \sigma(t) \phi + (\sigma(t)\cdot \nabla\phi) \dot u(t) -f(t)\dot u(t) \phi \big)\dx.
\end{equation}
Particular, according to Fubini's Theorem, the function $t \mapsto \langle\mu(t),\phi\rangle$ is measurable, and by density, this property remains true for all $\phi \in \mathcal C_c(\R^n)$. It shows the weak* measurability of the mapping $\mu:t \mapsto \mu(t) \in \mathcal M(\R^n)$. In addition, since $\mu(t)$ has compact support in $\overline \O$, we can take $\phi\equiv 1$ as test function in \eqref{defmut} which ensures that
$$\mu(t)(\R^n)=\langle \mu(t),1\rangle=-\int_\O \big(\dot u(t)\ddot u(t)+\sigma(t) \cdot \dot \sigma(t) -f(t)\dot u(t)  \big)\dx.$$
Using next the Cauchy-Schwarz inequality, we get that  
\begin{equation}\label{mut}
\supess_{t \in [0,T]} \mu(t)(\R^n) <+\infty
\end{equation}
which shows that $\mu \in L^\infty_{w*}(0,T;\mathcal M(\R^n))$.

\medskip

\noindent {\bf Step 3: Equation of motion.} In \eqref{eq:hv2}, choosing $\tau=0$ and $k \in \R$ arbitrary leads to
$$\int_0^T \int_\O \left( \ddot u \varphi +\sigma\cdot \nabla \varphi-f\varphi \right)\dxdt=0$$
for every $\varphi \in \mathcal C^1_c(\O \times (0,T))$, which implies that 
$$\ddot u-\Div \sigma=f \text{ in }L^\infty(0,T;L^2(\O)).$$
In particular $\sigma \in L^\infty(0,T;H(\Div,\O))$, and since by the stress constraint $\sigma \in L^\infty(\O \times (0,T);\R^n)$, then $\sigma \cdot \nu \in L^\infty(\partial \O\times (0,T))$. 
As a consequence, for all $\varphi \in \mathcal C^1_c(\R^n \times (0,T))$,
$$\int_0^T \int_\O \left( \ddot u \varphi +\sigma\cdot \nabla \varphi-f\varphi \right)\dxdt
=\int_0^T \int_{\partial \O} (\sigma\cdot \nu)\varphi\dd\HH^{n-1}\dt.$$
Reporting into \eqref{eq:hv2} leads to
\begin{multline}\label{eq:hv4}
-2\int_0^T \int_\O \big(\dot u\ddot u\varphi+\sigma \cdot \dot \sigma \varphi +( \sigma\cdot \nabla \varphi)\dot u -f\dot u \varphi \big)\dxdt\\
+ 2k\int_0^T \int_{\partial \O} (\sigma\cdot \nu)\varphi\dd\HH^{n-1}\dt+2\tau\cdot \int_0^T \int_\O \left( \dot \sigma \varphi +\dot u \nabla \varphi \right)\dxdt\\
\hfill+2\lambda\int_0^T \int_{\partial \O} \left( \frac{k}{2\lambda} - \frac{\tau\cdot \nu}{2} \right) ^2\varphi\dd\HH^{n-1}\dt \geq 0.
\end{multline}

\medskip

\noindent {\bf Step 3: Flow rule and additional regularity.} Choosing next $k=0$ in \eqref{eq:hv4}, we get that for all $\tau \in B$ and all  $\varphi \in \mathcal C^1_c(\O \times (0,T))$ with $\varphi\geq 0$,
$$-\tau\cdot \int_0^T \int_\O \left( \dot \sigma \varphi +\dot u \nabla \varphi \right)\dxdt \leq \int_0^T \int_\O \varphi \dd \mu(t)\dt.$$
Using the definition \eqref{eq:def_p} of $p$, we infer that
$$\tau \int_0^T \langle \dot p(t),\varphi(t) \rangle \dt \leq  \int_0^T \int_\O \varphi \dd \mu(t)\dt.$$
Since $\dot p \in L^2(0,T;H^{-1}(\O))$, we can localize with respect to time to get that, for a.e. $t \in [0,T]$ and all $\phi \in \mathcal C^1_c(\O)$ with $\phi\geq 0$, 
$$\tau \cdot \langle \dot p(t),\phi \rangle\leq  \int_\O \phi \dd \mu(t).$$
Passing to the supremum with respect to $\tau \in B$ and using \eqref{mut} yields
\begin{equation}\label{H-1}
| \langle \dot p(t),\phi \rangle| \leq  \int_\O \phi \dd \mu(t) \leq C_* \|\varphi\|_\infty,
\end{equation}
where $C_*>0$ is independent of $t$, which shows that $\dot p(t) \in \mathcal M(\O;\R^n)$. In particular, since $\mu(t) \geq 0$, we obtain that 
\begin{equation}\label{1618}
|\dot p(t)|\leq \mu(t) \quad \text{ in }\mathcal M(\O).
\end{equation}
We already know that $p \in W^{1,\infty}([0,T];H^{-1}(\O;\R^n))$ which ensures the measurability of the function $t \mapsto \langle \dot p(t),\phi\rangle$ 
 for all $\phi \in \mathcal C^1_c(\O)$. Then by density this property remains true for all $\phi \in \mathcal C_c(\O)$. It shows the weak* measurability of the mapping $t \mapsto \dot p(t) \in \mathcal M(\O;\R^n)$, and by \eqref{H-1}, 
 that $\dot p \in L^\infty_{w*}(0,T;\mathcal M(\O;\R^n))$. Then by definition of the distributional derivative we have that for all $\phi \in \mathcal C_c(\O)$, 
$t \mapsto  \langle p(t),\phi \rangle \in W^{1,\infty}([0,T];\R^n)$ and $\frac{\dd}{\dt}\langle p(t),\phi \rangle=\langle \dot p(t),\phi \rangle$ for a.e. $t \in [0,T]$. It thus shows that for all $0 \leq t_1 \leq t_2 \leq T$,
$$|\langle p(t_2)- p(t_1),\phi \rangle|= |\langle p(t_2),\phi \rangle-\langle p(t_1),\phi \rangle|=\left|\int_{t_1}^{t_2} \langle \dot p(t),\phi \rangle\dt\right| \leq C_*\|\phi\|_\infty(t_2-t_1).$$
Dividing the previous inequality by $\|\phi\|_\infty$ and passing to the supremum with respect to $\phi \in \mathcal C_c(\O)$ shows that $p\in \mathcal C^{0,1}([0,T];\mathcal M(\O;\R^n))$. By construction, we have $Du(t)=\sigma(t)+p(t)$ in $\mathcal M(\O;\R^n)$ for all $t \in [0,T]$, and therefore, thanks to the already established regularity of $u$ and $\sigma$, we infer that $u \in \mathcal C^{0,1}([0,T];BV(\O))$.

According to the definition \eqref{defmut} of the measure $\mu(t)$ and the equation of motion, we have that for all $\phi \in \mathcal C^1_c(\O)$ and for a.e. $t \in [0,T]$,
\begin{eqnarray*}
\langle\mu(t),\phi\rangle & =& -\int_\O \big(\dot u(t)\Div \sigma(t) \phi+\sigma(t) \cdot \dot \sigma(t) \phi + (\sigma(t)\cdot \nabla \phi)\dot u(t) \big)\dx\\
& =& \langle[\sigma(t)\cdot \dot p(t)],\phi\rangle
\end{eqnarray*}
which is well defined according to Definition \ref{def_Sigma_p} since, for a.e. $t \in [0,T]$, $\sigma(t) \in H(\Div,\O) \cap L^\infty(\O;\R^n)$, and $D\dot u(t)=\dot \sigma(t)+\dot p(t)$ with $\dot u(t) \in BV(\O) \cap L^2(\O)$, $\dot \sigma(t) \in L^2(\O;\R^n)$, and $\dot p(t)\in \mathcal M(\O;\R^n)$. Therefore, \eqref{1618} yields $|\dot p(t)|\leq [\sigma(t)\cdot \dot p(t)]$ in $\mathcal M(\O)$. On the other hand, using the stress constraint $\|\sigma(t)\|_\infty \leq 1$ and Remark \ref{rmk:ineq1}, the other inequality $|\dot p(t)|\geq [\sigma(t)\cdot \dot p(t)]$ in $\mathcal M(\O)$ follows, so that, finally
$$|\dot p(t)|=  [\sigma(t)\cdot \dot p(t)] \quad \text{ in }\mathcal M(\O).$$

\medskip

\noindent {\bf Step 4: Boundary condition.} Let $\varphi \in \mathcal C^1_c(\R^n \times (0,T))$ with $\varphi \geq 0$. 
Using, the integration by parts formula in $BV$, we get that 
$$ \int_0^T \int_\O \left( \dot \sigma \varphi +\dot u \nabla \varphi \right)\dxdt =-\int_0^T \int_\O \varphi \dd\dot p(t)\dt+ \int_0^T \int_{\partial \O} \varphi \dot u \nu\dd\HH^{n-1}\dt,$$
while the integration by parts formula given by Proposition \ref{IPP_sig_u_varphi} together with the equation of motion yields 
$$\langle\mu(t),\phi\rangle  = \langle[\sigma(t)\cdot \dot p(t)],\phi\rangle - \int_{\partial \O} (\sigma(t)\cdot\nu)\dot u(t)\phi\dd\HH^{n-1}.$$
Reporting inside \eqref{eq:hv4}, it follows that
\begin{multline*}
\int_0^T \int_\O \varphi \dd[(\sigma(t)-\tau)\cdot \dot p(t)]\dt \geq \int_0^T\int_{\partial \O} (\sigma\cdot\nu-\tau\cdot\nu)\dot u\varphi\dd\HH^{n-1}\dt\\
- k\int_0^T \int_{\partial \O} (\sigma\cdot \nu)\varphi\dd\HH^{n-1}\dt-\lambda\int_0^T \int_{\partial \O} \left( \frac{k}{2\lambda} - \frac{\tau\cdot \nu}{2} \right) ^2\varphi\dd\HH^{n-1}\dt,
\end{multline*}
and localizing in time, we get that for a.e. $t \in [0,T]$, and all $\phi \in W^{1,\infty}(\O)$ with $\phi \geq 0$. 
\begin{multline}\label{eq:relBC}
 \int_\O \phi \dd[(\sigma(t)-\tau)\cdot \dot p(t)] \geq \int_{\partial \O} (\sigma(t)\cdot\nu-\tau\cdot\nu)\dot u(t)\phi\dd\HH^{n-1}\\
- k\int_{\partial \O} (\sigma(t)\cdot \nu)\phi\dd\HH^{n-1}-\lambda \int_{\partial \O} \left( \frac{k}{2\lambda} - \frac{\tau\cdot \nu}{2} \right) ^2\phi\dd\HH^{n-1}.
\end{multline}
We next wish to localize the previous relation in space. To this aim, we define,
$$\phi_\e(x):=
\begin{cases}
0 & \text{ if } \dist(x,\partial \O)\geq\e,\\
\frac{\e-\dist(x,\partial \O)}{\e} & \text{ if }\dist(x,\partial \O)< \e.
\end{cases}$$
Note that $\phi_\e \in W^{1,\infty}(\O)$, $\phi_\e \geq 0$ and since $\O$ is of class $\mathcal C^2$, then for $\e>0$ small enough, we have
$\nabla \phi_\e(x-\e s \nu(x))= \nu(x)/\e$ for all $x \in \partial \O$ and all $s \in [0,1]$. Using that $\phi_\e=1$ on $\partial \O$, we get that for all $\zeta \in W^{1,\infty}(\O)$ with $\zeta	 \geq 0$,
\begin{multline}\label{1421}
 \int_\O \phi_\e\zeta \dd[(\sigma(t)-\tau)\cdot \dot p(t)] =
 -\int_\O (\sigma(t)-\tau)\cdot \dot \sigma(t) \phi_\e\zeta\dx -\int_\O \dot u(t) \Div \sigma(t) \phi_\e\zeta \dx\\
 -\int_\O (\sigma(t)-\tau)\cdot (\nabla\zeta) \phi_\e \dot u(t)\dx  -\int_\O (\sigma(t)-\tau)\cdot (\nabla \phi_\e)\zeta \dot u(t)\dx\\
   + \int_{\partial \O} (\sigma(t)\cdot \nu-\tau\cdot \nu)\dot u(t)\zeta\dd\HH^{n-1}.
\end{multline}
Since $\phi_\e \to 0$ strongly in $L^1(\O)$, the dominated convergence theorem ensures that the three first integrals in the right hand side of \eqref{1421} tend to zero as $\e \to 0$. Then, according to the coarea formula (see \cite[Lemma 3.2.34]{Federer}), denoting $\O_\e:=\{x \in \O : \dist(x,\partial \O)< \e\}$, the fourth integral writes as
\begin{multline*}
\int_\O (\sigma(t)-\tau)\cdot (\nabla \phi_\e)\zeta \dot u(t)\dx=\int_{\O_\e} (\sigma(t)-\tau)\cdot (\nabla \phi_\e)\zeta \dot u(t)\dx\\
=\int_0^1 \int_{\partial \O} \big(\sigma(x-s\e\nu(x),t) -\tau \big) \cdot \nu(x) \dot u(x-s\e\nu(x),t)\zeta(x-s\e \nu(x))\dd\HH^{n-1}(x)\ds.
\end{multline*}
Therefore, according to \eqref{eq:sigmanu} and \eqref{eq:traceu}, we obtain
$$\lim_{\e \to 0}\int_\O (\sigma(t)-\tau)\cdot (\nabla \phi_\e)\zeta \dot u(t)\dx=\int_{\partial \O} (\sigma(t)\cdot \nu -\tau\cdot \nu) \dot u(t)\zeta\dd\HH^{n-1}.$$
Inserting inside \eqref{1421}, we get that
$$\int_\O \phi_\e\zeta \dd[(\sigma(t)-\tau)\cdot \dot p(t)]  \to 0,$$
and \eqref{eq:relBC} yields for all $\zeta \in W^{1,\infty}(\O)$ with $\zeta \geq 0$,
$$0 \geq \int_{\partial \O} (\sigma(t)\cdot\nu-\tau\cdot\nu)\dot u(t)\zeta\dd\HH^{n-1}
- k\int_{\partial \O} (\sigma(t)\cdot \nu)\zeta\dd\HH^{n-1}-\lambda \int_{\partial \O} \left( \frac{k}{2\lambda} - \frac{\tau\cdot \nu}{2} \right)^2\zeta\dd\HH^{n-1}.$$
As a consequence, we get that for a.e. $t \in [0,T]$ and all $(k,\tau) \in \R \times B$, then
\begin{equation}\label{eq:loc}
- (\sigma(t)\cdot\nu-\tau\cdot\nu)\dot u(t) \geq
- k \sigma(t)\cdot \nu-\lambda  \left( \frac{k}{2\lambda} - \frac{\tau\cdot \nu}{2} \right)^2\quad \HH^{n-1}\text{-a.e. on }\partial \O.
\end{equation}
Note that the maximum of the right-hand side with respect to $k \in \R$, is attained at some $k^* \in \R$, which satisfies according to the first order condition
$$-\sigma(t)\cdot \nu-\left( \frac{k^*}{2\lambda} - \frac{\tau\cdot \nu}{2} \right)=0.$$
Replacing in \eqref{eq:loc}, and taking $\tau=-z\nu$ for any $z \in [-1,1]$, leads to
$$-(\sigma(t)\cdot \nu+z) \geq \lambda (\sigma(t)\cdot \nu)^2+\lambda z\sigma(t)\cdot \nu,$$
or still, thanks to Young's inequality
$$\frac{\lambda}{2}z^2 \geq \frac{\lambda}{2}(\sigma(t)\cdot \nu)^2 +\dot u(t)(z+\sigma(t)\cdot \nu).$$
Finally, from Remark \ref{rem:psi}, we deduce that $\dot u(t) \in\partial  \psi^*_\lambda(-\sigma(t)\cdot \nu)$, or equivalently that
$\sigma(t)\cdot \nu+\psi'_\lambda(\dot u(t))=0$.
\end{proof}

\begin{remark}
In the variational framework, the initial data must satisfy the hypotheses~\eqref{eq:initial_conditions}. Here, the constraint is also satisfied by the initial data. Note that since the hyperbolic variables are the velocity and the stress, the additive decomposition of the initial data has been ensured by the construction of the plastic strain. However, we did not suppose that the initial data satisfies the boundary condition.
\end{remark}

\section{Acknowledgements}
The authors wish to express their gratitude to Bruno Despr\'es and Nicolas Seguin for helpful and stimulating discussions on this paper.

\bibliographystyle{plain}
\bibliography{bibtex_MB_ARMA}   

\end{document}